\documentclass[12pt]{amsart}   
\usepackage{amsmath}
\usepackage{amsxtra}
\usepackage{amscd}
\usepackage{amsthm}
\usepackage{amsfonts}
\usepackage{amssymb}
\usepackage{eucal}
\usepackage{verbatim}
\usepackage[all]{xy}
\usepackage{color}
\definecolor{darkspringgreen}{rgb}{0.09, 0.45, 0.27}
\definecolor{Red}{rgb}{0.7,0,0}
\usepackage{stmaryrd}

\usepackage[pdftex,bookmarks=false,colorlinks=true,citecolor=darkspringgreen,debug=true,  naturalnames=true,pdfnewwindow=true]{hyperref}
\usepackage{geometry}
\usepackage[utf8]{inputenc}
\usepackage[T1]{fontenc}
\usepackage[english]{babel}
\usepackage{microtype} 
\usepackage{tabularx} 
\usepackage{booktabs} 
\usepackage{amsmath, amsfonts,amssymb,amsthm,mathbbol, textcomp, yfonts}
\usepackage[all]{xy}
\usepackage{tikz,nicematrix}
\usetikzlibrary{decorations.pathreplacing}
\usepackage{appendix}
\usepackage{etoolbox}%
\usepackage[T1]{fontenc}
\usepackage{lmodern}
\patchcmd{\footnotemark}{\stepcounter{footnote}}{\refstepcounter{footnote}}{}{}

\usepackage[normalem]{ulem}
\newcommand{\stkout}[1]{\ifmmode\text{\sout{\ensuremath{#1}}}\else\sout{#1}\fi}
\usepackage{xcolor,cancel}

\newtheorem{thm}{Theorem}[subsection]
\newtheorem{cor}[thm]{Corollary}

\newtheorem{lem}[thm]{Lemma}
\newtheorem{prop}[thm]{Proposition}
\newtheorem{conj}[thm]{Conjecture}

\theoremstyle{defn}
\newtheorem{defn}[thm]{Definition}

\theoremstyle{definition} 
\newtheorem{rem}[thm]{Remark}

\newcommand{\nc}{\newcommand}
\nc{\renc}{\renewcommand} \nc{\ssec}{\subsection}
\nc{\sssec}{\subsubsection}

 \nc{\wh}{\widehat}
\nc\ol{\overline} \nc\ul{\underline} \nc\wt{\widetilde}

\nc{\BA}{{\mathbb{A}}} \nc{\BC}{{\mathbb{C}}} \nc{\BF}{{\mathbb{F}}}
\nc{\BD}{{\mathbb{D}}} \nc{\BG}{{\mathbb{G}}} \nc{\BQ}{{\mathbb{Q}}}
\nc{\BM}{{\mathbb{M}}} \nc{\BN}{{\mathbb{N}}} \nc{\BO}{{\mathbb{\bfO}}}
\nc{\BP}{{\mathbb{P}}} \nc{\BR}{{\mathbb{R}}}
\nc{\BZ}{{\mathbb{Z}}} \nc{\BS}{{\mathbb{S}}} \nc{\BW}{{\mathbb{W}}}

\nc{\CA}{{\mathcal{A}}} \nc{\CB}{{\mathcal{B}}} \nc{\CalD}{{\mathcal{D}}}
\nc{\CE}{{\mathcal{E}}} \nc{\CF}{{\mathcal{F}}}
\nc{\CG}{{\mathcal{G}}} \nc{\CH}{{\mathcal{H}}}
\nc{\CI}{{\mathcal{I}}} \nc{\CK}{{\mathcal{K}}} \nc{\CL}{{\mathcal{L}}}
\nc{\CM}{{\mathcal{M}}} \nc{\CN}{{\mathcal{N}}}
\nc{\CO}{{\mathcal{\bfO}}} \nc{\CP}{{\mathcal{P}}}
\nc{\CQ}{{\mathcal{Q}}} \nc{\CR}{{\mathcal{R}}}
\nc{\CS}{{\mathcal{S}}} \nc{\CT}{{\mathcal{T}}}
\nc{\CU}{{\mathcal{U}}} \nc{\CV}{{\mathcal{V}}}  \nc{\CY}{{\mathcal Y}}
\nc{\CW}{{\mathcal{W}}} \nc{\CZ}{{\mathcal{Z}}}

\nc{\cM}{{\check{\mathcal M}}{}} \nc{\csM}{{\check{\mathcal A}}{}}
\nc{\oM}{{\overset{\circ}{\mathcal M}}{}}
\nc{\obM}{{\overset{\circ}{\mathbf M}}{}}
\nc{\oCA}{{\overset{\circ}{\mathcal A}}{}}
\nc{\obA}{{\overset{\circ}{\mathbf A}}{}}
\nc{\ooM}{{\overset{\circ}{M}}{}}
\nc{\osM}{{\overset{\circ}{\mathsf M}}{}}
\nc{\vM}{{\overset{\bullet}{\mathcal M}}{}}
\nc{\nM}{{\underset{\bullet}{\mathcal M}}{}}
\nc{\oD}{{\overset{\circ}{\mathcal D}}{}}
\nc{\obD}{{\overset{\circ}{\mathbf D}}{}}
\nc{\oA}{{\overset{\circ}{\mathbb A}}{}}
\nc{\op}{{\overset{\bullet}{\mathbf p}}{}}
\nc{\cp}{{\overset{\circ}{\mathbf p}}{}}
\nc{\oU}{{\overset{\bullet}{\mathcal U}}{}}
\nc{\ofZ}{{\overset{\circ}{\mathfrak Z}}{}}

\nc{\ff}{{\mathfrak{f}}} \nc{\fv}{{\mathfrak{v}}}
\nc{\fa}{{\mathfrak{a}}} \nc{\fb}{{\mathfrak{b}}}
\nc{\fd}{{\mathfrak{d}}} \nc{\fe}{{\mathfrak{e}}}
\nc{\fg}{{\mathfrak{g}}} \nc{\fgl}{{\mathfrak{gl}}}
\nc{\fh}{{\mathfrak{h}}} \nc{\fri}{{\mathfrak{i}}}
\nc{\fj}{{\mathfrak{j}}} \nc{\fk}{{\mathfrak{k}}} \nc{\fl}{{\mathfrak{l}}}
\nc{\fm}{{\mathfrak{m}}} \nc{\fn}{{\mathfrak{n}}}
\nc{\ft}{{\mathfrak{t}}} \nc{\fu}{{\mathfrak{u}}}
\nc{\fw}{{\mathfrak{w}}} \nc{\fz}{{\mathfrak{z}}}
\nc{\fp}{{\mathfrak{p}}} \nc{\fq}{{\mathfrak{q}}} \nc{\frr}{{\mathfrak{r}}}
\nc{\fs}{{\mathfrak{s}}} \nc{\fsl}{{\mathfrak{sl}}}
\nc{\hsl}{{\widehat{\mathfrak{sl}}}}
\nc{\hgl}{{\widehat{\mathfrak{gl}}}}
\nc{\hg}{{\widehat{\mathfrak{g}}}}
\nc{\chg}{{\widehat{\mathfrak{g}}}{}^\vee}
\nc{\hn}{{\widehat{\mathfrak{n}}}}
\nc{\chn}{{\widehat{\mathfrak{n}}}{}^\vee}

\nc{\fA}{{\mathfrak{A}}} \nc{\fB}{{\mathfrak{B}}} \nc{\fC}{{\mathfrak{C}}}
\nc{\fD}{{\mathfrak{D}}} \nc{\fE}{{\mathfrak{E}}}
\nc{\fF}{{\mathfrak{F}}} \nc{\fG}{{\mathfrak{G}}} \nc{\fH}{{\mathfrak{H}}}
\nc{\fI}{{\mathfrak{I}}} \nc{\fJ}{{\mathfrak{J}}}
\nc{\fK}{{\mathfrak{K}}} \nc{\fL}{{\mathfrak{L}}}
\nc{\fM}{{\mathfrak{M}}} \nc{\fN}{{\mathfrak{N}}}
\nc{\frP}{{\mathfrak{P}}} \nc{\fQ}{{\mathfrak{Q}}}
\nc{\fS}{{\mathfrak{S}}} \nc{\fT}{{\mathfrak{T}}} \nc{\fU}{{\mathfrak{U}}}
\nc{\fV}{{\mathfrak{V}}} \nc{\fW}{{\mathfrak{W}}}
\nc{\fX}{{\mathfrak{X}}} \nc{\fY}{{\mathfrak{Y}}}
\nc{\fZ}{{\mathfrak{Z}}}

\nc{\ba}{{\mathbf{a}}}
\nc{\bb}{{\mathbf{b}}} \nc{\bc}{{\mathbf{c}}}
\nc{\be}{{\mathbf{e}}} \nc{\bj}{{\mathbf{j}}} \nc{\bm}{{\mathbf{m}}}
\nc{\bn}{{\mathbf{n}}} \nc{\bp}{{\mathbf{p}}}
\nc{\bq}{{\mathbf{q}}} \nc{\br}{{\mathbf{r}}} \nc{\bt}{{\mathbf{t}}}
\nc{\bfu}{{\mathbf{u}}} \nc{\bv}{{\mathbf{v}}}
\nc{\bx}{{\mathbf{x}}} \nc{\by}{{\mathbf{y}}} \nc{\bz}{{\mathbf{z}}}
\nc{\bw}{{\mathbf{w}}} \nc{\bA}{{\mathbf{A}}}
\nc{\bB}{{\mathbf{B}}} \nc{\bC}{{\mathbf{C}}}
\nc{\bD}{{\mathbf{D}}} \nc{\bF}{{\mathbf{F}}} \nc{\bG}{{\mathbf{G}}}
\nc{\bH}{{\mathbf{H}}} \nc{\bI}{{\mathbf{I}}} \nc{\bJ}{{\mathbf{J}}}
\nc{\bK}{{\mathbf{K}}} \nc{\bM}{{\mathbf{M}}} \nc{\bN}{{\mathbf{N}}}
\nc{\bO}{{\mathbf{\bfO}}} \nc{\bS}{{\mathbf{S}}} \nc{\bT}{{\mathbf{T}}}
\nc{\bU}{{\mathbf{U}}} \nc{\bV}{{\mathbf{V}}} \nc{\bW}{{\mathbf{W}}}
\nc{\bX}{{\mathbf{X}}}
\nc{\bY}{{\mathbf{Y}}} \nc{\bP}{{\mathbf{P}}}
\nc{\bZ}{{\mathbf{Z}}} \nc{\bh}{{\mathbf{h}}}

\nc{\sA}{{\mathsf{A}}} \nc{\sB}{{\mathsf{B}}}
\nc{\sC}{{\mathsf{C}}} \nc{\sD}{{\mathsf{D}}}
\nc{\sE}{{\mathsf{E}}} \nc{\sF}{{\mathsf{F}}} \nc{\sG}{{\mathsf{G}}}
\nc{\sI}{{\mathsf{I}}} \nc{\sK}{{\mathsf{K}}} \nc{\sL}{{\mathsf{L}}}
\nc{\sfm}{{\mathsf{m}}} \nc{\sM}{{\mathsf{M}}} \nc{\sO}{{\mathsf{\bfO}}}
\nc{\sQ}{{\mathsf{Q}}} \nc{\sP}{{\mathsf{P}}}
\nc{\sT}{{\mathsf{T}}} \nc{\sZ}{{\mathsf{Z}}}
\nc{\sV}{{\mathsf{V}}} \nc{\sW}{{\mathsf{W}}}
\nc{\sfp}{{\mathsf{p}}} \nc{\sq}{{\mathsf{q}}} \nc{\sr}{{\mathsf{r}}}
\nc{\st}{{\mathsf{t}}} \nc{\sfb}{{\mathsf{b}}}
\nc{\sfc}{{\mathsf{c}}} \nc{\sd}{{\mathsf{d}}}
\nc{\sz}{{\mathsf{z}}}

\nc{\tA}{{\widetilde{\mathbf{A}}}}
\nc{\tB}{{\widetilde{\mathcal{B}}}}
\nc{\tg}{{\widetilde{\mathfrak{g}}}} \nc{\tG}{{\widetilde{G}}}
\nc{\TM}{{\widetilde{\mathbb{M}}}{}}
\nc{\tO}{{\widetilde{\mathsf{\bfO}}}{}}
\nc{\tU}{{\widetilde{\mathfrak{U}}}{}} \nc{\TZ}{{\tilde{Z}}}
\nc{\tx}{{\tilde{x}}} \nc{\tbv}{{\tilde{\bv}}}
\nc{\tfP}{{\widetilde{\mathfrak{P}}}{}} \nc{\tz}{{\tilde{\zeta}}}
\nc{\tmu}{{\tilde{\mu}}}

\nc{\urho}{\underline{\rho}} \nc{\uB}{\underline{B}}
\nc{\uC}{{\underline{\mathbb{C}}}} \nc{\ui}{\underline{i}}
\nc{\uj}{\underline{j}} \nc{\ofP}{{\overline{\mathfrak{P}}}}
\nc{\oB}{{\overline{\mathcal{B}}}}
\nc{\og}{{\overline{\mathfrak{g}}}} \nc{\oI}{{\overline{I}}}

\nc{\eps}{\varepsilon} \nc{\hrho}{{\hat{\rho}}}
\nc{\blambda}{{\boldsymbol{\lambda}}} \nc{\bmu}{{\boldsymbol{\mu}}} \nc{\bnu}{{\boldsymbol{\nu}}}

\nc{\one}{{\mathbf{1}}} \nc{\two}{{\mathbf{t}}}
\nc{\Cat}{\mathop{\operatorname{\rm 1-Cat}}}
\nc{\Sym}{\mathop{\operatorname{\rm Sym}}}
\nc{\Tot}{{\mathop{\operatorname{\rm Tot}}}}
\nc{\Spec}{\mathop{\operatorname{\rm Spec}}}
\nc{\Ker}{{\mathop{\operatorname{\rm Ker}}}}
\nc{\Isom}{{\mathop{\operatorname{\rm Isom}}}}
\nc{\Hilb}{{\mathop{\operatorname{\rm Hilb}}}}
\nc{\deeq}{{\mathop{\operatorname{\rm deeq}}}}
\nc{\End}{{\mathop{\operatorname{\rm End}}}}
\nc{\Ext}{{\mathop{\operatorname{\rm Ext}}}}
\nc{\Hom}{{\mathop{\operatorname{\rm Hom}}}}
\nc{\CHom}{{\mathop{\operatorname{{\mathcal{H}}\it om}}}}
\nc{\GL}{{\mathop{\operatorname{\rm GL}}}}
\nc{\Mir}{{\mathop{\operatorname{\rm Mir}}}}
\nc{\St}{{\mathop{\operatorname{\rm St}}}}
\nc{\oblv}{{\mathop{\operatorname{\rm oblv}}}}
\nc{\gr}{{\mathop{\operatorname{\rm gr}}}}
\nc{\Id}{{\mathop{\operatorname{\rm Id}}}}
\nc{\perf}{{\mathop{\operatorname{\rm perf}}}}
\nc{\defi}{{\mathop{\operatorname{\rm def}}}}
\nc{\length}{{\mathop{\operatorname{\rm length}}}}
\nc{\supp}{{\mathop{\operatorname{\rm supp}}}}
\nc{\colim}{{\mathop{\operatorname{\rm colim}}}}
\nc{\Fun}{{\mathop{\operatorname{\rm Funct}}}}
\nc{\Hei}{{\mathop{\operatorname{\rm Heis}}}}
\nc{\HC}{{\mathcal H}{\mathcal C}}
\nc{\ren}{{\mathsf{ren}}}
\nc{\locc}{{\mathsf{loc.c}}}
\nc{\pr}{{\operatorname{pr}}}
\nc{\Cliff}{{\mathsf{Cliff}}}
\nc{\loc}{{\operatorname{loc}}}
\nc{\Fl}{{\mathbf{Fl}}} \nc{\Ffl}{{\mathcal{F}\ell}}
\nc{\Fib}{{\mathsf{Fib}}}
\nc{\Coh}{{\mathsf{Coh}}} \nc{\FCoh}{{\mathsf{FCoh}}}
\nc{\Perf}{{\mathsf{Perf}}}

\nc{\reg}{{\text{\rm reg}}}
\nc{\gvee}{{\mathfrak g}^{\!\scriptscriptstyle\vee}}
\nc{\tvee}{{\mathfrak t}^{\!\scriptscriptstyle\vee}}
\nc{\nvee}{{\mathfrak n}^{\!\scriptscriptstyle\vee}}
\nc{\bvee}{{\mathfrak b}^{\!\scriptscriptstyle\vee}}
       \nc{\rhovee}{\rh\bfO^{\!\scriptscriptstyle\vee}}

\nc{\cplus}{{\mathbf{C}_+}} \nc{\cminus}{{\mathbf{C}_-}}
\nc{\cthree}{{\mathbf{C}_*}} \nc{\Qbar}{{\bar{Q}}}

\nc{\Gtimes}{\vphantom{j^{X^2}}\smash{\overset{G}{\vphantom{\rule{0pt}{0.30em}}\smash{\times}}}}
\nc{\sGtimes}{\vphantom{j^{X^2}}\smash{\overset{\mathsf G}{\vphantom{\rule{0pt}{0.30em}}\smash{\times}}}}

\nc{\bOmega}{{\overline{\Omega}}}

\nc{\seq}[1]{\stackrel{#1}{\sim}}
\nc{\nilp}{{\operatorname{Nilp}}}
\nc{\aff}{{\operatorname{aff}}}
\nc{\fin}{{\operatorname{fin}}}
\nc{\mir}{{\operatorname{mir}}}
\nc{\triv}{{\operatorname{triv}}}
\nc{\ext}{{\operatorname{ext}}}
\nc{\righ}{{\operatorname{right}}}
\nc{\lef}{{\operatorname{left}}}
\nc{\forg}{{\operatorname{forg}}}
\nc{\fid}{{\operatorname{fd}}}
\nc{\modu}{{\operatorname{-mod}}}
\nc{\Mor}{{\operatorname{Mor}}}
\nc{\Gr}{{\mathbf{Gr}}}
\nc{\FT}{{\operatorname{FT}}}
\nc{\Mat}{{\operatorname{Mat}}}
\nc{\MSt}{{\operatorname{MSt}}}
\nc{\sph}{{\operatorname{sph}}}
\nc{\GR}{{\mathbf{Gr}}}
\nc{\Perv}{{\operatorname{Perv}}}
\nc{\Rep}{{\operatorname{Rep}}}
\nc{\Ind}{{\operatorname{Ind}}}
\nc{\IC}{{\operatorname{IC}}}
\nc{\Bun}{{\operatorname{Bun}}}
\nc{\Proj}{{\operatorname{Proj}}}
\nc{\colimc}{{\operatorname{colim}^{co}}}
\nc{\Stab}{{\operatorname{Stab}}}
\nc{\pt}{{\operatorname{pt}}}
\nc{\bfmu}{{\boldsymbol{\mu}}}
\nc{\bfomega}{{\boldsymbol{\omega}}}
\nc{\calM}{\mathcal M}
\nc{\calA}{\mathcal A}
\nc{\calO}{\mathcal O}
\nc{\cC}{\mathcal C}
\nc{\CC}{\mathbb C}
\nc{\calN}{\mathcal N}
\nc{\grg}{\mathfrak g}
\nc{\tslash}{/\!\!/\!\!/}
\nc\grt{\mathfrak t}
\nc\bfM{\mathbf M}
\nc\bfN{\mathbf N}
\let\xra\xrightarrow
\nc\ZZ{\mathbb{Z}}
\nc\calC{\mathcal C}
\nc\calF{\mathcal F}
\nc\calX{\mathcal X}
\nc\calY{\mathcal Y}
\nc\QCoh{\operatorname{QCoh}}
\nc\IndCoh{\operatorname{IndCoh}}
\nc\Maps{\operatorname{Maps}}
\nc\Dmod{D-\operatorname{mod}}
\newcommand\Hecke{\operatorname{Hecke}}
\nc{\calD}{\mathcal D}
\nc\bfO{\mathbf O}
\nc\bfF{\mathbf F}

\nc\GG{\mathbb G}
\nc\calK{\mathcal K}
\nc{\calG}{\mathcal G}
\nc\RHom{\operatorname{RHom}}
\nc\Res{\operatorname{Res}}
\nc\Av{\operatorname{Av}}
\nc\grs{\mathfrak s}
\nc{\tilX}{\widetilde X}
\nc\calB{\mathcal B}
\nc\calS{\mathcal S}
\nc\calT{\mathcal T}
\nc\calZ{\mathcal Z}
\nc\LS{\operatorname{LocSys}}
\nc\bfL{\on{\mathbf L}}

\newcommand*\circled[1]
{*{#1}}
\makeatletter
\newcommand{\raisemath}[1]{\mathpalette{\raisem@th{#1}}}
\newcommand{\raisem@th}[3]{\raisebox{#1}{$#2#3$}}
\nc{\binlim}[2][]{\def\@tempa{#1}\@ifnextchar^{\@binlim{#2}}{\@binlim{#2}^{}}}
\def\@binlim#1^#2{\mathbin{\@ifempty{#2}{\mathop{#1}}{\mathop{#1}\@xp\displaylimits\@tempa^{#2}}}}

\makeatother

\nc\cX{{\mathcal X}}

\nc\Gm{{\mathbb G_m}}
\nc{\isoto}
    {\stackrel{\displaystyle\raisemath{-.2ex}{\sim}}\to}
\nc\cD\CalD
\nc\setm\setminus
\nc\cE\CE
\renc\Hecke{\mathit{\CH\kern-.2ex ecke}}
\nc\bsl\backslash
\nc\Fq{\mathbb F_q}
\nc\x\times
\nc\bGO{{\bG_\bO}}
\nc\bla\blambda
\nc\blt\bullet
\nc\opp{{\on{op}}}
\nc\srel\stackrel
\nc\tbx{\binlim{\widetilde\boxtimes{}}}
\nc\chkbx{\binlim{\check\boxtimes}}
\nc\into\hookrightarrow
\nc\otto\leftrightarrow
\renc\setminus\smallsetminus
\nc\phitau{\varphi\tau}
\newenvironment{i-ii-iii}{%
\begin{enumerate}%
\renewcommand{\theenumi}{\roman{enumi}}}%
{\end{enumerate}}
\nc\ceil[1]{\lceil#1\rceil}  \nc\floor[1]{\lfloor#1\rfloor}
\nc\Lie{\on{Lie}}
 \let\arXiv\arxiv
\nc\on\operatorname
\nc\kap{\kappa}
\nc\gra{\mathfrak a}
\nc\gl{\mathfrak{gl}}
\nc\sTr{\operatorname{sTr}}
\nc\hatG{\widehat{G}}
\nc\calL{\mathcal L}
\nc\Whit{\operatorname{Whit}}
\nc\KL{\operatorname{KL}}

\makeatletter

\allowdisplaybreaks
\nc\mto{\mapsto }
\nc\en{\enspace }

\numberwithin{equation}{section}
\newtheorem*{rep@theorem}{\rep@title}
\newcommand{\newreptheorem}[2]{%
\newenvironment{rep#1}[1]{%
 \def\rep@title{#2 \ref{##1}}%
 \begin{rep@theorem}}%
 {\end{rep@theorem}}}
\makeatother
\usepackage{tikz-cd}

\usepackage{wasysym, stackengine, makebox, graphicx}

\usepackage{xcolor,verbatim}
 \newcommand{\ncmd}{\newcommand*}

\ncmd{\DMO}{\DeclareMathOperator}
\ncmd{\ncmdd}[2]{\ncmd{#1}{{#2}}}

\makeatletter
\ncmd{\DefOps}[1]{\def\OPERATOR@NAME##1{\DeclareMathOperator{##1}{\expandafter\@gobble\string##1}}
    \def\OPERATOR@LIST##1{\ifcat\noexpand\relax\noexpand##1\OPERATOR@NAME##1\expandafter\OPERATOR@LIST\fi}
    \OPERATOR@LIST#1.}

\ncmd{\DefRm}[1]{\def\OPERATOR@NAME##1{\ncmd{##1}{\mathrm{\expandafter\@gobble\string##1}}}
    \def\OPERATOR@LIST##1{\ifcat\noexpand\relax\noexpand##1\OPERATOR@NAME##1\expandafter\OPERATOR@LIST\fi}
    \OPERATOR@LIST#1.}

\nc\negquad{\mkern-18mu}
\nc\lhs{\quad&\negquad}
\ncmd{\phtr}[2]{\lefteqn{#1{\phantom{#2}}}#2}

\def\Alphabet{ABCDEFGHIJKLMNOPQRSTUVWXYZ}
\def\newalph#1#2{\begingroup
    \def\procL@tt@r##1{%
        \@xp\gdef\csname#1\endcsname{#2}}%
    \proc@lph@bet\endgroup}
\def\proc@lph@bet{\@xp\prlist@\Alphabet\relax}
\def\prlist@#1{\ifx#1\relax\else\procL@tt@r{#1}\@xp\prlist@\fi}

\ncmd{\SmSub}[2][]{_\bgroup #2\smsub@{#1}}
\def\smsub@#1{\@ifnextchar_{#1\@smsb}{\egroup}}
\def\@smsb_#1{#1\smsub@{}}

\ncmd{\SmSup}[2][]{^\bgroup #2\smsup@{#1}}
\def\smsup@#1{\@ifnextchar^{#1\@smsp}{\@ifnextchar'{\prime\@xp\smsup@\@xp{\@xp}\@gobble}{\egroup}}}
\def\@smsp^#1{#1\smsup@{}}

\def\@binlim#1_#2{\mathbin{\@ifempty{#2}{\mathop{#1}}{\mathop{#1}\@xp\displaylimits\@tempa_{#2}}}}

\theoremstyle{defn}
\newtheorem{ex}[thm]{Example}
\theoremstyle{remark}
\newalph{c#1}{\ensuremath{{\mathcal #1}}}
\newalph{s#1}{{\mathsf #1}}
\newalph{#1#1}{{\mathbb #1}}
\newalph{e#1}{{\EuScript #1}}
\newalph{g#1}{{\mathfrak #1}}
\newalph{r#1}{{\ensuremath{\mathscr #1}}}
   \let\D\cD
\ncmd{\Fp}{{\FF_p}}

\DMO{\cHom}{\text{\textrm{\itshape{\cH}\kern-.2ex{}om}}}
\DMO{\cEnd}{\text{\textrm{\itshape{\cE}\kern-.2ex{}nd}}}    \DMO{\cExt}{\text{\textrm{\itshape{\cE}\kern-.2ex{}xt}}}

\ncmd{\sff}{\mathsf f}

\let\Im\undefined   \let\det\undefined
\ncmd{\young}[1]{\vcenter{\begin{Young}#1\crcr\end{Young}}}
\DefOps   { \REnd  \Aut
             \det \tr \rk \curv  \Coker \Im  \ad \Ad }
\DefRm{ \id \Vect \Pic \Loc \Hitch \Higgs \Eu \Fr}

\ncmd{\RGam}{\text{\upshape R}\Gamma}
\DMO{\chr}{char}
\ncmd{\angs}[1]{\langle#1\rangle}
\ncmd{\hmod}{\text{\upshape-mod}}
\ncmd{\cxym}[1]{\ensuremath{\vcenter{\xymatrix{#1}}}}
 \let\arXiv\arxiv
\makeatother

\title{Untwisted Gaiotto equivalence}
\date{}
\author[R.Travkin]{Roman Travkin}
\address{Skolkovo Institute of Science and Technology, Moscow, Russia}
\email{roman.travkin2012@gmail.com}
\author[R.Yang]{Ruotao Yang}
\address{Skolkovo Institute of Science and Technology, Moscow, Russia}
\email{yruotao@gmail.com}

\begin{document}

\begin{abstract}
This is a successive paper of \cite{[BFGT]}. We prove an equivalence between the category of finite-dimensional representations of degenerate supergroup $\underline{\GL}(M|N)$ and the category of $(\GL_M(\bfO) \ltimes U_{M, N}(\bfF), \chi_{M, N})$-equivariant D-modules on $\Gr_{N}$. We also prove that we can realize the category of finite-dimensional representations of degenerate supergroup $\underline{\GL}(M|N)$ as a category of D-modules on the mirabolic subgroup $\Mir_L(\bfF)$ with certain equivariant conditions for any $L$ bigger than $N$ and $M$.
\end{abstract}
\maketitle


\setcounter{tocdepth}{2}

\section{Introduction}
\subsection{Notations}
We denote by $\bfF=\mathbb{C}(\!(t)\!)$ the field of Laurent series and by $\bfO=\mathbb{C}[\![t]\!]$ the ring of formal power series. Given a group scheme $G$, we denote by $G(\bfF)$ the loop group of $G$ such that $\mathbb{C}$-points of $G(\bfF)$ are given by $\mathbb{C}(\!(t)\!)$-points of $G$, and denote by $G(\bfO)$ the arc group of $G$ such that $\mathbb{C}$-points of $G(\bfO)$ are given by $\mathbb{C}[\![t]\!]$-points of $G$. Set $\Gr_G:= G(\bfF)/G(\bfO)$ the affine Grassmannian of $G$.
It is known that $\Gr_G$ is a formally smooth ind-scheme.

Assume $G= \GL_N$. In this case, $\Gr_N:= \Gr_{\GL_N}$ admits a left action of $\GL_N(\bfF)$ by multiplication. In particular, any subgroup $H$ of $\GL_N(\bfF)$ acts on $\Gr_N$, and we can consider the corresponding derived equivariant category. Denote by $D^{H}(\Gr_N)$ the DG-category of left $H$-equivariant D-modules on $\Gr_N$.

\subsection{Reminder on mirabolic Satake equivalence}
\label{reminder}
Let $\text{Sym}^\bullet (\mathfrak{gl}_N[-2])$ {denote} the symmetric DG-algebra generated by $\mathfrak{gl}_N:= \text{Lie}(\GL_N)$, such that the generator is placed in degree 2 and the differential is 0. A famous result (\textit{derived Satake equivalence}) of \cite{[BF]} says that there is an equivalence between the bounded derived category of locally compact $\GL_N(\bfO)$-equivariant D-modules on $\Gr_N$ and the derived category of perfect $\GL_N$-equivariant modules over $\text{Sym}^\bullet (\mathfrak{gl}_N[-2])$. Namely,
\begin{equation*}\label{der Sat}
    D^{\GL_N(\bfO),\locc}(\Gr_N)\simeq D_{\perf}^{\GL_N}(\text{Sym}^\bullet (\mathfrak{gl}_N[-2])).
\end{equation*}
Here, \textit{locally compact} means compact when regarded as a plain D-module on $\Gr_N$.

In \cite{[BFGT]} and \cite{[FGT]}, the authors considered the mirabolic version of the above equivalence. On the D-module side in the mirabolic case, instead of considering $\GL_N(\bfO)$-equivariant D-modules on the affine Grassmannian $\Gr_N$, we need to consider $\GL_N(\bfO)$-equivariant D-modules on the mirabolic affine Grassmannian $\Gr_N\times \bfF^N$. Here, the action of $\GL_{N}(\bfO)$ is given by the diagonal action on $\Gr_N\times \bfF^N$. 

The definition of the mirabolic Satake category $D^{\GL_N(\bfO)}(\Gr_N\times \bfF^N)$ needs to be taken care of since any compact object (also, locally compact object) of the mirabolic Satake category is supported on an infinite-dimensional space. Indeed, any $\GL_N(\bfO)$-orbit in $\Gr_N\times (\bfF^N-0)$ is infinite-dimensional. 

The authors of \cite{[BFGT]} defined three different versions of the mirabolic Satake category: $D^{\GL_{N}(\bfO),\locc}_{!}(\Gr_N\times \bfF^N)$, $D^{\GL_{N}(\bfO),\locc}_{*}(\Gr_N\times \bfF^N)$, and $D^{\GL_{N}(\bfO),\locc}_{!*}(\Gr_N\times \bfF^N)$, according to different renormalizations. They are equipped with monoidal structures $\textasteriskcentered$, $\circledast$ and fusion product, respectively. Roughly speaking, $\textasteriskcentered$ is given by the convolution product and $\circledast$ is a combination of the convolution product on $\Gr_N$ and the tensor product on $\bfF^N$.  For more details of the definitions, see Section \ref{sec 3.1}.

In order to introduce the coherent side of the mirabolic Satake equivalence, we should introduce a certain Lie supergroup. 

Set $\mathbb{C}^{N|N}$ {to be} the super vector space with even vector space $\mathbb{C}^N$ and odd vector space $\mathbb{C}^N$. We denote by $\gl(N|N)$ the Lie superalgebra of endormorphisms of $\mathbb{C}^{N|N}$. One can define a degenerate version of the (super) Lie bracket on $\gl(N|N)$ where the supercommutator of an even element with any element is the same as in $\gl(N|N)$, but the supercommutator of any two odd elements is zero. The resulting Lie superalgebra is denoted by $\underline{\gl}(N|N)$, and the corresponding Lie supergroup  ({with the even part isomorphic to $\GL_N\x\GL_N$}) is denoted  by $\underline{\GL}(N|N)$.

It is proved  in \cite{[BFGT]} that $D_{!*}^{\GL_N(\bfO),\locc}(\Gr_N\times \bfF^N)$ is equivalent to $\Rep^\fin(\underline{\GL}(N|N))$, the bounded derived category of finite-dimensional representations of $\underline{\GL}(N|N)$.

This equivalence is $t$-exact with respect to the natural $t$-structure of $\Rep^\fin(\underline{\GL}(N|N))$ with the heart $\Rep^\fin(\underline{\GL}(N|N))^\heartsuit$ and the perverse $t$-structure of $D_{!*}^{\GL_N(\bfO),\locc}(\Gr_N\times \bfF^N)$. Furthermore, it is an equivalence of monoidal categories with respect to the fusion product on $D_{!*}^{\GL_N(\bfO),\locc}(\Gr_N\times \bfF^N)$ and the tensor product on $\Rep^\fin(\underline{\GL}(N|N))$.

Similar to the case of the classical Satake category, the category of finite-dimensional modules over the degenerate Lie supergroup admits a Koszul dual realization $D_\perf^{\GL_N\times \GL_N}(\mathfrak{B}_{N|N})$, where $\mathfrak{B}_{N|N}$ is a DG-algebra on an affine space with an action of $\GL_N\times \GL_N$. If we forget the degrees, $\mathfrak{B}_{N|N}$ is isomorphic to functions on $\mathfrak{gl}_N\times \mathfrak{gl}_N$. By imposing different gradings, there are three different versions of the symmetric algebras $\mathfrak{B}_{N|N}^{2,0}, \mathfrak{B}_{N|N}^{1,1}, \mathfrak{B}_{N|N}^{0,2}$. 

In Section \ref{convolution}, we will recall two monoidal structures $\overset{A}{\star}$ and $\overset{B}{\star}$ of $D_\perf^{\GL_N\times \GL_N}(\mathfrak{B}_{N|N})$ which are defined in \cite[Section 3.2]{[BFGT]}.

The following Theorem is proved in \cite{[BFGT]}.
\begin{thm}\label{bgft1}
There are monoidal equivalences of DG-categories,
\begin{equation}\label{NN}
\begin{split}
        (D^{\GL_N(\bfO),\locc}_{!}(\Gr_N\times \bfF^N), \circledast)\simeq (D_\perf^{\GL_N\times \GL_N}(\mathfrak{B}^{2,0}_{N|N}), \overset{A}{\star}),\\
        (D^{\GL_N(\bfO),\locc}_{*}(\Gr_N\times \bfF^N), \textasteriskcentered)\simeq (D_\perf^{\GL_N\times \GL_N}(\mathfrak{B}^{0,2}_{N|N}), \overset{B}{\star}).
\end{split}
\end{equation}
which commute with left and right convolutions with $\Perv_{\GL_N(\bO)}(\Gr_N)\simeq \Rep^\fin(\GL_N)^\heartsuit$.
\end{thm}

The above equivalences are also known as the untwisted Gaiotto conjecture in the case of $N=M$.

\subsection{Untwisted Gaiotto conjecture in the case $M=N-1$}
We denote by $\gl(N-1|N)$ the Lie superalgebra of endormorphisms of $\mathbb{C}^{N-1|N}$ and by $\underline{\gl}(N-1|N)$ the degenerate version of it. It is proved in \cite[Theorem 4.1.1]{[BFGT]} that there is an equivalence between the bounded derived category of finite-dimensional representations of  $\underline{\GL}(N-1|N)$ and the category of locally compact $\GL_{N-1}(\bfO)$-equivariant D-modules on $\Gr_N$. 

This equivalence is $t$-exact with respect to the naive $t$-structure of $\Rep^\fin(\underline{\GL}(N-1|N))$ with the heart $\Rep^\fin(\underline{\GL}(N-1|N))^\heartsuit$ and the perverse $t$-structure of $D^{\GL_{N-1}(\bfO),\locc}(\Gr_N)$. 

Note that $\Rep^\fin(\underline{\GL}(N-1|N))$ admits a left action of $\Rep^\fin(\GL_{N-1})$ and a right action of $\Rep^\fin(\GL_N)$, $D^{\GL_{N-1}(\bfO),\locc}(\Gr_N)$ admits a left action of $D^{\GL_{N-1}(\bfO), \locc}(\Gr_{N-1})$ and a right action of $D^{\GL_{N}(\bfO), \locc}(\Gr_{N})$. It is also proved in \cite[Theorem 4.1.1]{[BFGT]} that the equivalence between $\Rep^\fin(\underline{\GL}(N-1|N))$ and $D^{\GL_{N-1}(\bfO),\locc}(\Gr_N)$ is compatible with respect to the above actions via classical Satake equivalence.

Similar to the case $N=M$, the category of finite-dimensional modules over the degenerate Lie supergroup admits a Koszul realization $D_\perf^{\GL_{N-1}\times \GL_{N}}(\mathfrak{B}_{N-1|N})$. So \cite[Theorem 4.1.1]{[BFGT]} can be stated as follows.

\begin{thm}\label{bfgt2}
There is an equivalence of categories,
\begin{equation}\label{N-1}
    D^{\GL_{N-1}(\bfO),\locc}(\Gr_N)\simeq D_\perf^{\GL_{N-1}\times \GL_N}(\mathfrak{B}_{N-1|N}^{2,0}).
\end{equation}
\end{thm}

This equivalence is deduced in \cite{[BFGT]} from the mirabolic Satake equivalence recalled
in~\S\ref{reminder}. We can present $\Rep^\fin(\underline{\GL}(N-1|N))$ as a colimit of a full subcategory of $\Rep^\fin(\underline{\GL}(N|N))$ with the transition functor being tensoring with the determinant module. On the D-module side, $D^{\GL_{N-1}(\bfO),\locc}(\Gr_N)$ admits a colimit presentation given by $D^{{\GL}_{N}(\bfO),\locc}(\Gr_N\times (\bfO^N\setminus t\bfO^N))$ that is a full subcategory of $D^{\GL_{N}(\bfO),\locc}(\Gr_N\times \bfF^N)$. It is proved in {\em loc.cit.}\ that there exists an equivalence between these two full subcategories and the equivalence is compatible with taking colimits.

\subsection{Untwisted Gaiotto conjecture in general case}
In order to state the general untwisted Gaiotto conjecture, we need to introduce some notations.

Given a character 
$\chi: H'\longrightarrow \mathbb{A}^1 $
, we denote by $D^{H',\chi}(\Gr_N)$ the category of $(H', \chi^!(\exp))$-equivariant D-modules on $\Gr_N$. Here, $\exp$ is the exponential D-module on $\mathbb{A}^1$.

Assuming $M<N$, we denote by $P_{M,N}(\bfF)$ the parabolic subgroup of $\GL_N(\bfF)$ corresponding to the partition $(M+1, 1,1,...,1)$, and denote by $U_{M,N}(\bfF)$ the unipotent radical of $P_{M,N}(\bfF)$. 

There is a group morphism from $U_{M,N}(\bfF)$ to $\bfF$ which sends $(u_{i,j})\in U_{M,N}(\bfF)$ to $\sum_{i} u_{i,i+1}$. We compose this group morphism with taking residue
\begin{equation}
    \begin{split}
        \bfF&\longrightarrow \mathbb{C}\\
        \sum a_i t^i&\mapsto a_{-1}.
    \end{split}
\end{equation}
The resulting morphism is a character of $U_{M,N}(\bfF)$ denoted by $\chi_{M,N}$.

The conjugation action of $\GL_M(\bfO)$ on $U_{M,N}(\bfF)$ preserves $U_{M,N}(\bfF)$. Furthermore, the character $\chi_{M,N}$ is stable under the conjugation action. Hence, $\chi_{M,N}$ gives a character on 
\begin{equation}\label{1.3 GL}
{\GL}_{M}({\bfO}) \ltimes U_{M, N}(\bfF)=
\begin{pNiceArray}{wc{2.9em}ccwc{4em}|cccc}[last-col,code-for-last-col = \quad\,]
\Block{4-4}<\huge>{\GL_M(\bO)} & &&& 0 &*& \Cdots  &*  & \Block[l]{4-1}{\quad\, M} \\
\\
 & &   && \Vdots & \Vdots & &\Vdots &\\
 &&& 
 &0&*&\Cdots&* \\ 
\Block{4-4}<\huge>{0} &          &   && 1&*&\Cdots &* &\Block[l]{4-1}{\quad\, N-M} \\
 &          &   && 0 &\Ddots& \Ddots &\Vdots &  \\
&          &   && \Vdots &\Ddots&& * \\
&		&& & 0 & \Cdots  & 0 &1
\CodeAfter
%
%
\UnderBrace[shorten,yshift=0.5ex]{1-1}{8-4}{M}
\UnderBrace[shorten,yshift=0.5ex]{1-5}{8-8}{N-M}
\SubMatrix .{1-8}{4-8}{\}}[right-xshift=1em]
\SubMatrix .{5-8}{8-8}{\}}[right-xshift=1em]
\end{pNiceArray}  
\end{equation}

\vspace{1.6em}
%
%
%

Note that there is a forgetful functor from $D^{\GL_M(\bfO) \ltimes U_{M, N}(\bfF), \chi_{M, N}}\left(\Gr_{N}\right)$ to $D^{U_{M, N}(\bfF), \chi_{M, N}}\left(\Gr_{N}\right)$. We denote by $(D^{U_{M, N}(\bfF), \chi_{M, N}}\left(\Gr_{N}\right))^{\GL_M(\bfO),\locc}$ the full subcategory of $D^{\GL_M(\bfO) \ltimes U_{M, N}(\bfF), \chi_{M, N}}\left(\Gr_{N}\right)$ generated by the preimage of compact objects of $D^{U_{M, N}(\bfF), \chi_{M, N}}\left(\Gr_{N}\right)$.

In this paper, we will prove the following Gaiotto conjecture:
\begin{thm}\label{Gaiotto}
Assume that $N>M$. Then, the untwisted Gaiotto category $(D^{U_{M, N}(\bfF), \chi_{M, N}}\left(\Gr_{N}\right))^{\GL_M(\bfO),\locc}$ is equivalent to  the bounded derived category of finite-dimensional modules
over the degenerate supergroup $\underline{{\GL}}(M|N))$. In other words, there is an equivalence 
\begin{equation}
    (D^{U_{M, N}(\bfF), \chi_{M, N}}\left(\Gr_{N}\right))^{\GL_M(\bfO),\locc}\simeq D_\perf^{\GL_{M}\times \GL_N}(\mathfrak{B}_{M|N}^{2,0}).
\end{equation}
\end{thm}

We will also establish an equivalence about the category of compact objects of $D^{\GL_M(\bfO) \ltimes U_{M, N}(\bfF), \chi_{M, N}}\left(\Gr_{N}\right)$. It corresponds to a full subcategory of $\Rep^\fin(\underline{\GL}(M|N))$ with a nilpotent support condition, see Theorem \ref{thm 4.4.1}.

We should mention that the untwisted Gaiotto conjecture is a particular case of the Periods—$L$-functions duality conjectures of D.~Ben-Zvi, Y.~Sakellaridis and A.~Venkatesh, cf., \cite[Conjecture 7.5.1]{[BZSV]}.

\subsection{Strategy of the proof}
Fix $N$, the proof of Theorem \ref{Gaiotto} proceeds by descending induction in $M$. In~\S\ref{convolution}
and~\S\ref{conv cons}, we will recall the monoidal structures of $D^{\GL_M(\bfO),\locc}(\Gr_M\times \bfF^M)$ and of the category of modules over $\underline{\GL}(M|M)$. Furthermore, we will see that these categories act on $(D^{ U_{M,N}(\bfF),\chi_{M,N}}(\Gr_N))^{\GL_M(\bfO),\locc}$ and the category of modules over $\underline{\GL}(M|N)$ respectively.

Hence, we can divide the proof of Theorem \ref{Gaiotto} into three steps. 

The first step: we present $\Rep^\fin(\underline{\GL}(M-1|N))$ as a colimit of subcategories. Namely, we show that $\Rep^\fin(\underline{\GL}(M-1|N))$ can be obtained from $\Rep^\fin(\underline{\GL}(M-1|M))$ and $\Rep^\fin(\underline{\GL}(M|N))$ by taking tensor product and colimits. It is proved in Proposition \ref{prop 2.1}.

The second step: in~\S\ref{cons sid}, we present $(D^{U_{M-1, N}(\bfF), \chi_{M-1, N}}\left(\Gr_{{\GL}_{N}}\right))^{\GL_{M-1}(\bfO) \locc}$ as a colimit of subcategories. We prove an analog of Proposition \ref{prop 2.1} {on the} D-module side. Namely, we prove that there is an equivalence between the category $(D^{U_{M-1, N}(\bfF), \chi_{M-1, N}}\left(\Gr_{{\GL}_{N}}\right))^{\GL_{M-1}(\bfO) \locc}$ and the tensor product of $D^{\GL_{M-1}(\bfO),\locc}\left(\Gr_{{\GL}_{M}}\right)$ and $(D^{U_{M, N}(\bfF), \chi_{M, N}}\left(\Gr_{{\GL}_{N}}\right))^{\GL_M(\bfO), \locc}$ over the mirabolic Satake category of rank $M$. This step is the key point of the proof, and we will work it out using the Fourier transform. It is proved in Theorem \ref{claim1}.

In ~\S\ref{act com}, to finish the proof, we only need to compare two tensor products of categories. We only need to show that the equivalence~\cite[Theorem 4.1.1]{[BFGT]} $D^{\GL_{M-1}(\bfO),\locc}\left(\Gr_{{\GL}_{M}}\right)\simeq \Rep^\fin(\underline{\GL}(M-1|M))$ is compatible with the actions of the mirabolic Satake categories (Proposition \ref{act}). To prove this statement, we will use colimit presentations of the mirabolic Satake category of rank $M-1$ and the category of finite-dimensional modules over $\underline{\GL}(M-1|M-1)$. Then the proof can be obtained from the monoidal structure of the equivalence of~Theorem~\ref{bgft1}.

Using a similar method, we can prove Theorem \ref{thm 4.4.1}. 

\subsection{Symmetric definition of $D^{\GL_M(\bfO) \ltimes U_{M, N}(\bfF), \chi_{M, N}}\left(\Gr_{N}\right)$}
In the statement of Theorem \ref{Gaiotto} and \ref{thm 4.4.1}, we need to require $M<N$: otherwise, the subgroup $U_{M,N}(\bfF)\subset \GL_N(\bfF)$ makes no sense.

However, the supergroup $\underline{\GL}(M|N)$ is well-defined for any $N $ and $M$. Hence, we hope to find a good definition of the corresponding category on the D-module side which does not need to require $M<N$.

We can fix this problem by choosing an integer $L$ which is strictly bigger than $N$ and $M$, and define a category of D-modules on the mirabolic subgroup $\Mir_L(\bfF)$ of $\GL_L(\bfF)$ with certain equivariant conditions. To be more precise, we can define $C_{M,N,L}$ to be the category of D-modules on $\Mir_L(K)$ that are  left $(\GL_M(\bfO)\ltimes U_{M,L}, \chi_{M,L})$-equivariant and right $ (\GL_N(\bfO)\ltimes U_{N,L}, \chi_{N,L})$-equivariant.

In~\S\ref{5}, we prove that the resulting category $C_{M,N,L}$ is independent of the choice of $L$ and is equivalent to the category of $D^{\GL_M(\bfO) \ltimes U_{M, N}(\bfF), \chi_{M, N}}\left(\Gr_{{N}}\right)$ if $M<N$ and $D^{\GL_N(\bfO)}(\Gr_N\times \bfF^N)$ if $M=N$. 

\subsection{Iwahori version of the conjecture}
Let $I_N$ be the Iwahori subgroup of $\GL_N(\bfO)$ {and $\Fl_N \simeq \GL_N(\bfF)/I_N$ the affine flag variety for $\GL_N$}.. In \cite{[B]}, the author proved that there is an equivalence between $D^{I_N}(\Fl_N)$ and the derived category of $\GL_N$-equivariant D-modules on the DG version Steinberg variety. Motivated by this result, \cite{[BFGT]} proposed the mirabolic version of Bezrukanikov's equivalence, and \cite{[BFT]} proposed the orthosymplectic version of it.

Mimicking \cite[Conjecture 1.4.1]{[BFGT]}, we propose the following conjecture.

For $M< N$, let $V_1$ be {an} $M$-dimensional vector space and $V_2$ be {an} $N$-dimensional vector space. Let $\Ffl_i, i=1,2$ denote the variety of complete flags in $V_i$. 

We let $\St_{\Mir, M,N}$ be the subvariety of $\Hom(V_1, V_2)[1]\times \Hom(V_2, V_1)[1]\times \Ffl_1\times \Ffl_2$, such that $(A,B, F_1=(F_1^{(1)}\subset F_1^{(2)}...\subset F_1^{(M)}= V_1), F_2=(F_2^{(1)}\subset F_2^{(2)}...\subset F_2^{(N)}= V_2))$ belongs to $\St_{\Mir, M,N}$ if and only if $AB(F_2^{(i)})\subset F_2^{(i)}$, $i=1,2,...,N$, and $BA(F_1^{(j)})\subset F_1^{(j)}$, $j=1,2,...,M$.

Let $D^{I_M\times U_{M,N}(\bfF),\chi_{M,N}}(\Fl_N)$ be the derived category of $(I_M\times U_{M,N}(\bfF),\chi_{M,N})$-equivariant  sheaves on $\Fl_N$. We propose the following conjecture.
\begin{conj}
There is an equivalence of categories,
\begin{equation}
    D^{I_M\times U_{M,N}(\bfF),\chi_{M,N}}(\Fl_N)\simeq D^{\GL_M\times \GL_N}(\St_{\Mir, M,N}).
\end{equation}
\end{conj}



\section{Coherent side}
\subsection{Degenerate supergroup}
{Let us recall} the  degenerate Lie superalgebra $\underline{\mathfrak{gl}}({M|N})$ {from~ \cite{[BFGT]}}. By definition, {its underlying super vector space} is the same as the underlying super vector space of ${\mathfrak{gl}}({M|N})$, and the supercommutators of even elements with any element in $\underline{\mathfrak{gl}}({M|N})$ and ${\mathfrak{gl}}({M|N})$ are same. But the {supercommutator} of any two odd elements in $\underline{\mathfrak{gl}}({M|N})$ {is}  set to be zero. 

The bounded derived category of finite-dimensional representations of the corresponding supergroup $\underline{\GL}({M|N})$ {(with the even part isomorphic to $\GL_M \x \GL_N$)} is denoted by $\Rep^\fin(\underline{\GL}({M|N}))$.
\subsection{Symmetric algebra realization}\label{sect 2.2}
Assume that $V_1$ is {an} $M$-dimensional vector space with {a basis} $e_1, e_2,...,e_M$ and $V_2$ is {an} $N$-dimensional vector space. Set the symmetric algebra in the category of complexes $\mathfrak{B}_{M|N}^{1,1}:=  \operatorname{Sym}^{\bullet} (\Hom(V_1, V_2)[-1])\otimes \operatorname{Sym}^{\bullet}(\Hom(V_2. V_1)[-1])= \operatorname{Sym}^{\bullet} (\underline{\mathfrak{gl}}(M|N)_{\bar{1}}[-1])$ $=\Sym^\blt(\Pi(\Hom(V_1,V_2)\oplus\Hom(V_2,V_1))[-1])$ where $\Pi$ stands for change of parity.

Let us denote by $\Rep^{\mathsf{loc.fin}}(\underline{{\GL}}(M|N))$ the ind-completion of $\Rep^\fin(\underline{{\GL}}(M|N))$. More precisely, the category $\Rep^{\mathsf{loc.fin}}(\underline{{\GL}}(M|N))$ is the unbounded derived category of locally finite representations of $\underline{\mathfrak{gl}}(M|N)$, and any object of $\Rep^{\mathsf{loc.fin}}(\underline{{\GL}}(M|N))$ is a filtered colimit of objects in $\Rep^\fin(\underline{{\GL}}(M|N))$.

By the Koszul duality, we have a monoidal equivalence
\begin{equation}
    \Rep^\fin(\underline{{\GL}}(M|N))\simeq D_{\perf}^{\GL_{V_1}\times \GL_{V_2}}(\mathfrak{B}_{M|N}^{1,1}).
\end{equation}
Here $D_{\perf}^{\GL_{V_1}\times \GL_{V_2}}(\mathfrak{B}_{M|N}^{1,1})$ denotes the derived category of {$(\GL_{V_1}\times \GL_{V_2})$-equivariant} perfect modules over $\mathfrak{B}_{M|N}^{1,1}$. 

\begin{rem}
The above equivalence induces an equivalence of their ind-completions
\begin{equation}
    \Rep^{\mathsf{loc.fin}}(\underline{{\GL}}(M|N))\simeq {{D}}^{\GL_{V_1}\times \GL_{V_2}}(\mathfrak{B}_{M|N}^{1,1})\simeq \Ind(D_{\perf}^{\GL_{V_1}\times \GL_{V_2}}(\mathfrak{B}_{M|N}^{1,1})).
\end{equation}

Here $D^{\GL_{V_1}\times \GL_{V_2}}(\mathfrak{B}_{M|N}^{1,1})$ denotes the derived category of $\GL_{V_1}\times \GL_{V_2}$-invariants modules over $\mathfrak{B}_{M|N}^{1,1}$. 
\end{rem}

In addition to $\mathfrak{B}_{M|N}^{1,1}$, we can consider the symmetric algebras:
\[\mathfrak{B}_{M|N}^{2,0}:=\operatorname{Sym}^{\bullet} (\Hom(V_1, V_2)[-2])\otimes \operatorname{Sym}^{\bullet}(\Hom(V_2. V_1)),\]
\[\mathfrak{B}_{M|N}^{0,2}:=\operatorname{Sym}^{\bullet} (\Hom(V_1, V_2))\otimes \operatorname{Sym}^{\bullet}(\Hom(V_2. V_1)[-2]).\]

The categories of $(\GL_M\x\GL_N)$-equivariant modules over $\mathfrak{B}_{M|N}^{2,0}$, $\mathfrak{B}_{M|N}^{1,1}$, and $\mathfrak{B}_{M|N}^{0,2}$ are equivalent via the functors of changing cohomological degrees{, see \cite[Section 3.6]{[BFGT]}}.  


\begin{rem}\label{rem 2.2.2}  {Up to the degree shearing operation (cf. \cite[Appendix A]{[AG]})},  the categories ${{D}}^{\GL_{V_1}\times \GL_{V_2}}(\mathfrak{B}_{M|N}^{i,2-i})$ are equivalent to $\IndCoh^{\GL_{V_1}\times \GL_{V_2}}(\Hom(V_1, V_2)\times \Hom(V_2, V_1))$ via identifying $\Hom(V_1, V_2)$ with $\Hom(V_2,V_1)^*$ and $\Hom(V_2, V_1)$ with $\Hom(V_1,V_2)^*$.
\end{rem}

Sometimes, in order to indicate the vector spaces used in the definition, we will denote $\mathfrak{B}_{M|N}^{i,2-i}$ by $\mathfrak{B}_{V_1|V_2}^{i,2-i}$. 
\begin{rem} {One  can show that for any other choice $V_1',V_2'$  of vector spaces of the same dimensions $M,N$, there is a {\em canonical} equivalence \[ D_\perf^{\GL_{V_1}\times \GL_{V_2}}(\mathfrak{B}_{V_1|V_2}^{i,2-i})\simeq D_\perf^{\GL_{V_1'}\times \GL_{V'_2}}(\mathfrak{B}_{V_1'|V_2'}^{i,2-i})\quad  ((i,2-i)\in\{(0,2),(1,1),(2,0)\}).\]
(Due to the $(\GL(V_1)\x\GL(V_2))$-equivariance, it is independent of the choice of isomorphisms $V_1\simeq V_1'$, $V_2\simeq V_2'$.)} 
\end{rem}

\subsection{Nilpotent support} We introduce a nilpotent cone $\nilp_{V_1|V_2}$ of $\Hom(V_1, V_2)\times \Hom(V_2, V_1)$. We denote by 
\begin{equation}\label{q}
    q: \Hom(V_1, V_2)\times \Hom(V_2, V_1)\longrightarrow \End(V_2)
\end{equation}
 the morphism sending $A_{1,2}, A_{2,1}\in \Hom(V_1, V_2)\times \Hom(V_2, V_1)$ to $A_{1,2}A_{2,1}\in \End(V_2)$.

We denote by $\nilp_{V_1|V_2}$ the subscheme of $\Hom(V_1, V_2)\times \Hom(V_2, V_1)$ such that the image $A_{1,2}A_{2,1}$ is nilpotent (equivalently, $A_{2,1}A_{1,2}\in \End(V_1)$ is nilpotent).

We denote by $D_{\perf}^{\GL_{V_1}\times \GL_{V_2}}(\mathfrak{B}_{M|N}^{i,2-i})_{\nilp}$ (resp. ${{D}}^{\GL_{V_1}\times \GL_{V_2}}(\mathfrak{B}_{M|N}^{i,2-i})_{\nilp}$) the full subcategory of $D_{\perf}^{\GL_{V_1}\times \GL_{V_2}}(\mathfrak{B}_{M|N}^{i,2-i})$ (resp. ${{D}}^{\GL_{V_1}\times \GL_{V_2}}(\mathfrak{B}_{M|N}^{i,2-i})$) spanned by the objects whose {cohomologies are} set-theoretically supported on $\nilp_{V_1|V_2}$. 

{
\begin{rem}
   According to \cite[Corollary 9.1.7]{[AG]}, the set-theoretical nilpotent support condition is Koszul dual to  the nilpotent singular support condition of ind-coherent sheaves. In particular, the category ${{D}}^{\GL_{V_1}\times \GL_{V_2}}(\mathfrak{B}_{M|N}^{i,2-i})_{\nilp}$ corresponds to the full subcategory of $\IndCoh((\Hom(V_1, V_2)\oplus \Hom(V_1, V_2)[-1])/(\GL_{V_1}\times \GL_{V_2}))$ with the nilpotent singular support condition.
\end{rem}
}

\subsection{Convolution}\label{convolution}
In this section, let us recall the definitions of the convolution products given in \cite[Section 3.4]{[BFGT]}.

Let us denote by $\mathcal{Q}_A$ (resp. $\mathcal{Q}_B$) the closed subvariety  in
$\mathcal{H}:=\Hom(W_1, W_2)\times$ $\Hom(W_2, W_1)\times \Hom(W_2, W_3)\times \Hom(W_3, W_2)\times \Hom(W_3, W_1)\times \Hom(W_1, W_3)$, such that, for any point
\[(A_{1,2}, A_{2,1}, A_{2,3}, A_{3,2}, A_{3,1}, A_{1,3})\in \mathcal{Q}_A,\] we have
\begin{equation}
    A_{1,3}=A_{2,3}A_{1,2}, A_{2,1}=A_{3,1}A_{2,3}, \textnormal{ and } A_{3,2}=A_{1,2}A_{3,1}.
\end{equation}
 \begin{equation}
 \textnormal{(resp. }\   A_{3,1}=A_{2,1}A_{3,2}, A_{1,2}=A_{3,3}A_{1,3}, \textnormal{ and } A_{2,3}=A_{1,3}A_{2,1}. \textnormal{)}
\end{equation}

Consider the following diagram  
\begin{center}
    \xymatrix @C = -2em{
&\mathcal{Q}_A\ar[ld]_{p_{1,2}}\ar[rd]^{p_{2,3}}\ar[dd]^{p_{1,3}}&\\
\Hom(W_1, W_2)\times \Hom(W_2, W_1)&&\Hom(W_2, W_3)\times \Hom(W_3, W_2)\\
&\Hom(W_1, W_3)\times \Hom(W_3, W_1)&
}
\end{center}

In \cite{[BFGT]}, the authors {consider} convolution products
\begin{equation}
    \begin{split}
        \overset{A}{\star}: {{D}}^{\GL_{W_1}\times \GL_{W_2}}(\mathfrak{B}_{W_1|W_2}^{2,0})\times {{D}}^{\GL_{W_2}\times \GL_{W_3}}(\mathfrak{B}_{W_2|W_3}^{2,0})\longrightarrow {{D}}^{\GL_{W_1}\times \GL_{W_3}}(\mathfrak{B}_{W_1|W_3}^{2,0})\\
    M_{1,2}\overset{A}{\star}M_{2,3}:= p_{1,3,*}(p_{1,2}^*M_{1,2}\underset{\mathbb{C}[\mathcal{H}]^\bullet_{2,0}}{\otimes}\mathbb{C}[\mathcal{Q}_A]_{2,0}\underset{\mathbb{C}[\mathcal{H}]^\bullet_{2,0}}{\otimes} p_{2,3}^* M_{2,3})^{\GL_{W_2}},\\
        \overset{B}{\star}: {{D}}^{\GL_{W_1}\times \GL_{W_2}}(\mathfrak{B}_{W_1|W_2}^{0,2})\times {{D}}^{\GL_{W_2}\times \GL_{W_3}}(\mathfrak{B}_{W_2|W_3}^{0,2})\longrightarrow {{D}}^{\GL_{W_1}\times \GL_{W_3}}(\mathfrak{B}_{W_1|W_3}^{0,2})\\
    M_{1,2}\overset{B}{\star}M_{2,3}:= p_{1,3,*}(p_{1,2}^*M_{1,2}\underset{\mathbb{C}[\mathcal{H}]^\bullet_{0,2}}{\otimes}\mathbb{C}[\mathcal{Q}_B]_{0,2}\underset{\mathbb{C}[\mathcal{H}]^\bullet_{0,2}}{\otimes} p_{2,3}^* M_{2,3})^{\GL_{W_2}}.
    \end{split}
\end{equation}

Here, \\
$\begin{aligned}
\mathbb{C}[\mathcal{H}]_{2,0}^{\bullet}=\operatorname{Sym}\left(\operatorname{Hom}\left(W_1, W_2\right)\right) \otimes \operatorname{Sym}\left(\operatorname{Hom}\left(W_2, W_1\right){[-2]}\right) \otimes \operatorname{Sym}\left(\operatorname{Hom}\left(W_3, W_2\right)\right) \\
\otimes \operatorname{Sym}\left(\operatorname{Hom}\left(W_2, W_3\right)[-2]\right) \otimes \operatorname{Sym}\left(\operatorname{Hom}\left(W_1, W_3\right)\right) \otimes \operatorname{Sym}\left(\operatorname{Hom}\left(W_3, W_1\right)[-2]\right),
\end{aligned}
$
\\
$\begin{aligned}
\mathbb{C}[\mathcal{H}]_{0,2}^{\bullet}=\operatorname{Sym}\left(\operatorname{Hom}\left(W_1, W_2\right){[-2]} \right) \otimes \operatorname{Sym}\left(\operatorname{Hom}\left(W_2, W_1\right)\right) \otimes \operatorname{Sym}\left(\operatorname{Hom}\left(W_3, W_2\right){[-2]} \right) \\
\otimes \operatorname{Sym}\left(\operatorname{Hom}\left(W_2, W_3\right)\right) \otimes \operatorname{Sym}\left(\operatorname{Hom}\left(W_1, W_3\right){[-2]}\right) \otimes \operatorname{Sym}\left(\operatorname{Hom}\left(W_3, W_1\right)\right),
\end{aligned}
$
\[\mathbb{C}[\mathcal{Q}_A]_{2,0}:=\operatorname{Sym}\left(\operatorname{Hom}\left(W_1, W_2\right)\right) \otimes\operatorname{Sym}\left(\operatorname{Hom}\left(W_2, W_3\right)[-2]\right) \otimes\operatorname{Sym}\left(\operatorname{Hom}\left(W_3, W_1\right)[-2]\right),\]
\[\mathbb{C}[\mathcal{Q}_B]_{0,2}:= \operatorname{Sym}\left(\operatorname{Hom}\left(W_2, W_1\right)\right) \otimes\operatorname{Sym}\left(\operatorname{Hom}\left(W_3, W_2\right){[-2]}\right)\otimes \operatorname{Sym}\left(\operatorname{Hom}\left(W_1, W_3\right){[-2]}\right).\]

The operation $\overset{A}{\star}$ (resp. $\overset{B}{\star}$) gives the category $ {{D}}^{\GL_{W_1}\times \GL_{W_2}}(\mathfrak{B}_{W_1|W_2}^{2,0})$ (resp. $ {{D}}^{\GL_{W_1}\times \GL_{W_2}}(\mathfrak{B}_{W_1|W_2}^{0,2})$) a monoidal category structure if $W_1$ and $W_2$ are of the same dimensions, and the resulting monoidal category can act on ${{D}}^{\GL_{W_3}\times \GL_{W_1}}(\mathfrak{B}_{W_3|W_1}^{2,0})$ (resp. ${{D}}^{\GL_{W_3}\times \GL_{W_1}}(\mathfrak{B}_{W_3|W_1}^{0,2})$) from right, and act on ${{D}}^{\GL_{W_2}\times \GL_{W_3}}(\mathfrak{B}_{W_2|W_3}^{2,0})$ (resp. ${{D}}^{\GL_{W_2}\times \GL_{W_3}}(\mathfrak{B}_{W_2|W_3}^{0,2})$) from left for any finite-dimensional vector space $W_3$.

\begin{rem}\label{2.4.1}
By construction, the convolution products induce the {following} functors
\begin{equation}\label{eq 2.7}
    \begin{split}
        \overset{A}{\star}: {{D}}^{\GL_{W_1}\times \GL_{W_2}}(\mathfrak{B}_{W_1|W_2}^{2,0})\times {{D}}^{\GL_{W_2}\times \GL_{W_3}}(\mathfrak{B}_{W_2|W_3}^{2,0})_{\nilp}\longrightarrow {{D}}^{\GL_{W_1}\times \GL_{W_3}}(\mathfrak{B}_{W_1|W_3}^{2,0})_{\nilp},\\
        \overset{A}{\star}:{{D}}^{\GL_{W_1}\times \GL_{W_2}}(\mathfrak{B}_{W_1|W_2}^{2,0})_{\nilp}\times {{D}}^{\GL_{W_2}\times \GL_{W_3}}(\mathfrak{B}_{W_2|W_3}^{2,0})\longrightarrow {{D}}^{\GL_{W_1}\times \GL_{W_3}}(\mathfrak{B}_{W_1|W_3}^{2,0})_{\nilp},\\
         \overset{B}{\star}:{{D}}^{\GL_{W_1}\times \GL_{W_2}}(\mathfrak{B}_{W_1|W_2}^{0,2})\times {{D}}^{\GL_{W_2}\times \GL_{W_3}}(\mathfrak{B}_{W_2|W_3}^{0,2})_{\nilp}\longrightarrow {{D}}^{\GL_{W_1}\times \GL_{W_3}}(\mathfrak{B}_{W_1|W_3}^{0,2})_{\nilp},\\
        \overset{B}{\star}: {{D}}^{\GL_{W_1}\times \GL_{W_2}}(\mathfrak{B}_{W_1|W_2}^{0,2})_{\nilp}\times {{D}}^{\GL_{W_2}\times \GL_{W_3}}(\mathfrak{B}_{W_2|W_3}^{0,2})\longrightarrow {{D}}^{\GL_{W_1}\times \GL_{W_3}}(\mathfrak{B}_{W_1|W_3}^{0,2})_{\nilp}.
    \end{split}
\end{equation}
\end{rem}
\begin{rem}
The convolution products $\overset{A}{\star}$ and $\overset{B}{\star}$ preserve compact objects. 
\end{rem}

\subsection{Endofunctor of ${{D}}^{\GL_M\times \GL_N}(\mathfrak{B}_{M|N})$}

From now to the end of this section, we assume that $\dim W_1= \dim W_2=M$ and $\dim W_3=N$. Let us take a $\GL_{W_1}\times \GL_{W_2}\times \GL_{W_3}$-invariant {subvariety} $\mathcal{Q}_A^0\subset \mathcal{Q}_A$ by requiring $A_{1,2}$ to be non-invertible. 

Consider the following diagram
\begin{center}
    \xymatrix @C = -2em{
&\mathcal{Q}^0_A\ar[ld]_{p^0_{1,2}}\ar[rd]^{p^0_{2,3}}\ar[dd]^{p^0_{1,3}}&\\
\Hom(W_1, W_2)\times \Hom(W_2, W_1)&&\Hom(W_2, W_3)\times \Hom(W_3, W_2)\\
&\Hom(W_1, W_3)\times \Hom(W_3, W_1).&
}
\end{center}

We denote by $T_{M,N}$ the endofunctor
\begin{equation}
    \begin{split}
        T_{M,N}: {{D}}^{\GL_{W_2}\times \GL_{W_3}}(\mathfrak{B}_{W_2|W_3}^{2,0})\longrightarrow& {{D}}^{\GL_{W_1}\times \GL_{W_3}}(\mathfrak{B}_{W_1|W_3}^{2,0})\\
        M_{2,3}\mapsto& p_{1,3,*}^0(p_{2,3}^{0,*} (M_{2,3}))^{\GL_{W_2}}.
    \end{split}
\end{equation}

By definition, the endofunctor $T_{M,N}$ is isomorphic to the endofunctor which is given by left {convolution} with $\tilde{c}\in {{D}}^{\GL_{W_1}\times \GL_{W_2}}(\mathfrak{B}_{W_1|W_2}^{2,0})$, i.e.,
\begin{equation}
\begin{split}
    T_{M,N}: {{D}}^{\GL_{W_2}\times \GL_{W_3}}(\mathfrak{B}_{W_2|W_3}^{2,0})&\longrightarrow {{D}}^{\GL_{W_1}\times \GL_{W_3}}(\mathfrak{B}_{W_1|W_3}^{2,0})\\ 
    M&\mapsto \tilde{c}\overset{A}{\star} M.
\end{split}
\end{equation}
Here, $\tilde{c}$ is the structure sheaf on the locus such that $A_{1,2}$ is non-invertible.

According to Remark \ref{2.4.1}, the endofunctor $T_{M,N}$ induces an endofunctor of the subcategory of perfect complexes with nilpotent support condition,
\begin{equation}
     T_{M,N}: D_\perf^{\GL_{W_2}\times \GL_{W_3}}(\mathfrak{B}_{W_2|W_3}^{2,0})_{\nilp}\longrightarrow D_\perf^{\GL_{W_1}\times \GL_{W_3}}(\mathfrak{B}_{W_1|W_3}^{2,0})_{\nilp},
\end{equation}
and an endofunctor of its ind-completion
\begin{equation}
     T_{M,N}: {{D}}^{\GL_{W_2}\times \GL_{W_3}}(\mathfrak{B}_{W_2|W_3}^{2,0})_{\nilp}\longrightarrow {{D}}^{\GL_{W_1}\times \GL_{W_3}}(\mathfrak{B}_{W_1|W_3}^{2,0})_{\nilp}.
\end{equation}

By construction, $\tilde{c}\overset{A}{\star} \tilde{c}\simeq \tilde{c}$. That is to say, $\tilde{c}$ is an idempotent algebra object with respect to $\overset{A}{\star}$. In particular, the Kleisli category ${{D}}(T_{M,N})$ of ${{D}}^{\GL_M\times \GL_N}(\mathfrak{B}_{M|N}^{2,0})$ is the full subcategory of ${{D}}^{\GL_M\times \GL_N}(\mathfrak{B}_{M|N}^{2,0})$ spanned by the essential image of $T_{M,N}$. Since convolution commutes with colimits, $T_{M,N}$ preserves colimits. In particular, ${{D}}(T_{M,N})$ is cocomplete.

Using the method applied in \cite[Section 4.3]{[BFGT]}, we can prove that ${{D}}(T_{M,N})$ can also be regarded as a full cocomplete subcategory ${{D}}^{\GL_{M-1}\times \GL_N, \geq 0}(\mathfrak{B}_{M-1|N}^{2,0})$ of ${{D}}^{\GL_{M-1}\times \GL_N}(\mathfrak{B}_{M-1|N}^{2,0})$ which is compactly generated by the objects $\{V^\lambda_{\GL_{M-1}}\otimes {\mathfrak{B}_{M-1|N}^{2,0}}\otimes V^\mu_{\GL_N}|\ \lambda \textnormal{ is a partition of {length}}\ M-1, \mu \textnormal{ is a signature of {length}}\ N\}$. Here, $V^\lambda_{\GL_{M-1}}$ and $V^\mu_{\GL_N}$ are the irreducible representations of $\GL_{M-1}$ and $\GL_N$.

To be self-contained, let us briefly recall the proof. Note that ${{D}}(T_{M,N})$ and ${{D}}^{\GL_{M-1}\times \GL_N, \geq 0}(\mathfrak{B}_{M-1|N}^{2,0})$ are compactly generated. It is sufficient to prove that the full subcategories of their compact objects are equivalent. 

\begin{lem}\label{lemma 2.1}
For $M\leq N$, {there exists an equivalence of categories}
\begin{equation}\label{eq 2.12}
   {\Psi:} D_{\perf}(T_{M,N}){\overset{\sim}{\longrightarrow}}\  D_{\perf}^{\GL_{M-1}\times \GL_N, \geq 0}(\mathfrak{B}_{M-1|N}^{2,0}){,}
\end{equation}
{which is compatible with the natural forgetful functors to $D_{\perf}^{\GL_{M-1}\times \GL_N}(\mathfrak{B}_{M-1|N}^{2,0})$.}
Here $D_\perf(T_{M,N})$ denotes the Kleisli subcategory of $D_\perf^{\GL_M\times \GL_N}(\mathfrak{B}_{M|N}^{2,0})$ associated with $T_{M,N}$.
\end{lem}
\begin{proof}
Denote by $\bar{W}_1$ the subspace of $W_1$ spanned by $e_1, e_2,..., e_{M-1}$. Let us consider the functor given by taking the convolution with $c_{\bar{1},1}$ which is the structure sheaf on $\Hom(W_{1}, \bar{W}_1)\times \Hom(\bar{W}_1, W_1)$.
\begin{equation}\label{2.6}
    \begin{split}
        D_{\perf}^{\GL_{W_1}\times \GL_{W_3}}(\mathfrak{B}_{W_1|W_3}^{2,0})&\longrightarrow D_{\perf}^{\GL_{\Bar{W}_1}\times \GL_{W_3}}(\mathfrak{B}_{\bar{W}_1|W_3}^{2,0})
   \\ M_{1,3}&\mapsto {c}_{\bar{1},1}\overset{A}{\star} M_{1,3}.
    \end{split}
\end{equation}

We denote by 
\begin{equation}\label{2.7}
    \Psi:D_{\perf}(T_{M,N})\longrightarrow D_{\perf}^{\GL_{\Bar{W}_1}\times \GL_{W_3}}(\mathfrak{B}_{\bar{W}_1|W_3}^{2,0})
\end{equation}
 the restriction of \eqref{2.6} to $D_{\perf}(T_{M,N})$. We claim that $\Psi$ is an equivalence between $D_\perf(T_{M,N})$ and $D_\perf^{\GL_{M-1}\times \GL_N, \geq 0}(\mathfrak{B}_{M-1|N}^{2,0})$.

For essential surjectivity, we only need to notice that the image of $V_{\GL_M}^\lambda\otimes \mathfrak{B}_{W_1|W_3}^{2,0}\otimes V_{\GL_N}^\mu$ is isomorphic to $V_{\GL_{M-1}}^{\bar{\lambda}}\otimes \mathfrak{B}_{\bar{W}_1|W_3}^{2,0}\otimes V_{\GL_M}^\mu$. Here, ${\bar{\lambda}}=(\lambda_1\geq \lambda_2\geq...\geq \lambda_{M-1})$ if $\lambda= (\lambda_1\geq \lambda_2\geq...\geq \lambda_{M-1}\geq 0)$, and $V_{\GL_{M-1}}^{\bar{\lambda}}=0$ otherwise. 

The full faithfullness follows from the isomorphism
\begin{equation}
    \begin{split}
        &\Hom_{D_\perf(T_{M,N})}(V_{\GL_M}^\lambda\otimes \mathfrak{B}_{W_2|W_3}^{2,0}\otimes V_{\GL_N}^\mu, V_{\GL_M}^{\lambda'}\otimes \mathfrak{B}_{W_2|W_3}^{2,0}\otimes V_{\GL_N}^{\mu'})\\
        \simeq &(V_{\GL_{W_2}}^{*,\lambda}\otimes V_{\GL_{W_1}}^{\lambda'}\otimes (\mathbb{C}[\mathcal{Q}_A^0])\otimes V_{\GL_{W_3}}^{\mu'}\otimes V_{\GL_{W_3}}^{*,\mu})^{\GL_{W_1}\times \GL_{{W}_{{2}}}\times \GL_{W_3}}\\
          \simeq &(V_{\GL_{M-1}}^{*,\bar{\lambda}}\otimes V_{\GL_{M-1}}^{\bar{\lambda}'}\otimes \mathfrak{B}_{M-1|N}^{2,0}\otimes V_{\GL_M}^{\mu'}\otimes V_{\GL_M}^{*,\mu})^{\GL_{\bar{W}_1}\times \GL_{W_2}}.
    \end{split}
\end{equation}

The last isomorphism follows from \cite[Lemma 3.13]{[BFGT]}
\end{proof}

As a corollary, we have
\begin{cor}\label{cor 2.5.2}
For $M\leq N$, {we have}
\begin{equation}\label{2.8}
 D_{\perf}(T_{M,N})_{\nilp}\simeq D_{\perf}^{\GL_{M-1}\times \GL_N, \geq 0}(\mathfrak{B}_{M-1|N}^{2,0})_{\nilp}{.}
\end{equation}
Here $D_{\perf}(T_{M,N})_{\nilp}$ denotes the Kleisli category of $D_\perf^{\GL_M\times \GL_N}(\mathfrak{B}_{M|N}^{2,0})_\nilp$ associated with $T_{M,N}$, and $D_{\perf}^{\GL_{M-1}\times \GL_N, \geq 0}(\mathfrak{B}_{M-1|N}^{2,0})_{\nilp}$ denotes the full subcategory of $D_{\perf}^{\GL_{M-1}\times \GL_N, \geq 0}(\mathfrak{B}_{M-1|N}^{2,0})$ with nilpotent support condition.
\end{cor}
\begin{proof}
We only need to show the restriction of $\Psi$ of \eqref{2.7} to $D_\perf(T_{M,N})_\nilp$ is essential surjective. 


Note that pullback along $q$ in \eqref{q} gives rise to a homomorphism, 
\[\Sym(\gl_N[-2])\longrightarrow \mathfrak{B}_{M|N}^{2,0}\].

{U}nder the functor $\Psi$ {in \eqref{eq 2.12}}, $V_{\GL_M}^\lambda\otimes (\mathfrak{B}_{M|N}^{2,0}\otimes_{\Sym(\gl_N[-2])^{\GL_N}}\mathbb{C})\otimes V_{\GL_N}^\mu$ goes to $V_{\GL_{M-1}}^{\bar{\lambda}}\otimes (\mathfrak{B}_{M-1|N}^{2,0}\otimes_{\Sym(\gl_N[-2])^{\GL_N}}\mathbb{C})\otimes V_{\GL_N}^\mu$.

 We have \[\mathfrak{B}_{M|N}^{2,0}\otimes_{\Sym(\gl_N[-2])^{\GL_N}}\mathbb{C}\simeq \mathcal{O}_{\nilp_{M|N}},\]and
\[\mathfrak{B}_{M-1|N}^{2,0}\otimes_{\Sym(\gl_N[-2])^{\GL_N}}\mathbb{C}\simeq \mathcal{O}_{\nilp_{M-1|N}}.\]


In particular, the image of \eqref{2.8} contains the collection of objects $\{V^\lambda_{\GL_{M-1}}\otimes \mathcal{O}_{\nilp_{M-1|N}}\otimes V^\mu_{\GL_N}| \lambda \textnormal{ is a partition of {length}}\ M-1, \mu \textnormal{ is a signature of {length}}\ N\}$. Now the claim follows from the fact that the category $D_{\perf}^{\GL_{M-1}\times \GL_N, \geq 0}(\mathfrak{B}_{M-1|N}^{2,0})_{\nilp}$ is generated by $\{V^\lambda_{\GL_{M-1}}\otimes \mathcal{O}_{\nilp_{M-1|N}}\otimes V^\mu_{\GL_N}|\ \lambda \textnormal{ is a partition of}\ M-1, \mu \textnormal{ is a signature of}\ N\}$.
\end{proof}
Let ${{D}}(T_{M,N})_\nilp$ be the Kleisli subcategory of ${{D}}^{\GL_M\times \GL_N}(\mathfrak{B}_{M|N}^{2,0})$ associated with the idempotent monad $T_{M,N}$, and let ${{D}}^{\GL_{M-1}\times \GL_N, \geq 0}(\mathfrak{B}_{M-1|N}^{2,0})_{\nilp}$ be the ind-completion of $D_{\perf}^{\GL_{M-1}\times \GL_N, \geq 0}(\mathfrak{B}_{M-1|N}^{2,0})_{\nilp}$.
\begin{cor}\label{cor 2.5.3}
For $M\leq N$, we have
\begin{equation}
    {{D}}(T_{M,N})_\nilp\simeq {{D}}^{\GL_{M-1}\times \GL_N, \geq 0}(\mathfrak{B}_{M-1|N}^{2,0})_{\nilp}.
\end{equation}
\end{cor}
\begin{proof}
It follows from the fact that ${{D}}(T_{M,N})_\nilp$ is equivalent to the ind-completion of $c_{\bar{1},1}\star D_{\perf}^{\GL_M\times \GL_N}(\mathfrak{B}_{M|N}^{{2,0}})_\nilp\simeq D_{\perf}(T_{M,N})_{\nilp}$.
\end{proof}

\subsection{Tensor product of categories of representations}
All the categories that we consider are DG-categories\footnote{By DG-category, we mean a stable $(\infty,1)$-category enhanced over the bounded complexes of finite dimensional vector spaces; and by cocomplete stable DG-category, we mean a cocomplete stable $(\infty,1)$-category enhanced over the complexes of vector spaces.}.  Given a {(resp. cocomplete)} monoidal 
 category $\cA$ and a right {(resp. cocomplete)} $\cA$-module category $\cM_1$ and a left {(resp. cocomplete)} $\cA$-module category $\cM_2$. We let the relative tensor product $\cM_1\otimes_{\cA} \cM_2$ be the geometric realization of $\cM_1 \otimes \cA^n \otimes \cM_2$ {inside the category of DG-categories} {(resp. cocomplete DG-categories with continuous functors)}, 
\begin{equation*}
\xymatrix{
  \cdots\cM_1 \otimes \cA\otimes \cA \otimes \cM_2\ar@<-1ex>[r] \ar@<1ex>[r]\ar[r]  &\cM_1 \otimes \cA \otimes \cM_2 \ar@<-.5ex>[r] \ar@<.5ex>[r]& \cM_1 \otimes \cM_2.
}
\end{equation*}
See \cite[Section 4.4]{[L]}.

\begin{prop}\label{prop 2.1}
For any $M\leq N$, there is an equivalence induced by \eqref{eq 2.7}
\begin{equation}\label{2.10}
    D^{\GL_{M-1}\times \GL_M}(\mathfrak{B}_{M-1|M}^{2,0})\otimes_{D^{\GL_{M}\times \GL_M}(\mathfrak{B}_{M|M}^{2,0})} D^{\GL_{M}\times \GL_N}(\mathfrak{B}_{M|N}^{2,0})\simeq D^{\GL_{M-1}\times \GL_N}(\mathfrak{B}_{M-1|N}^{2,0}).
\end{equation}
Similar claims hold for ${{D}}_\perf$, $D_{\perf}$ with nilpotent support, and ${{D}}$ with nilpotent support.
\end{prop}
\begin{proof}
In \cite[Section 4]{[BFGT]}, the authors constructed an invertible endofunctor $\eta$ of $D_\perf^{\GL_{M-1}\times \GL_N}(\mathfrak{B}_{M-1|N}^{2,0})$ which corresponds to tensoring with the determinant module up to a cohomological shift. This functor restricts to an endofunctor of $D_\perf^{\GL_{M-1}\times \GL_N, \geq 0}(\mathfrak{B}_{M-1|N}^{2,0})$. 

{Given a stable category $\cC$ and an endofunctor $F\colon\cC\to\cC$, we let $\colim'_F\cC$ denote the colimit  of the sequence $\cC\xra F\cC\xra F\cC\xra F\ldots$ in the $(\infty,1)$-category of stable $(\infty,1)$-categories.}  Similarly, {given a stable cocomplete category $\cC$ and a continuous endofunctor $F\colon\cC\to\cC$, we let $\colim_F\cC$ denote the colimit  of the sequence $\cC\xra F\cC\xra F\cC\xra F\ldots$ in the $(\infty,2)$-category of stable cocomplete $(\infty,1)$-categories with continous functors.} It is well-known that taking ind-completion preserves colimits.

For any object $M\in D_\perf^{\GL_{M-1}\times \GL_N}(\mathfrak{B}_{M-1|N}^{2,0})$, $\eta^n( M)\in D_\perf^{\GL_{M-1}\times \GL_N,\geq 0}(\mathfrak{B}_{M-1|N}^{2,0})$ for sufficient large $n$. Hence, the fully faithful embedding: \[D_\perf^{\GL_{M-1}\times \GL_N,\geq 0}(\mathfrak{B}_{M-1|N}^{2,0})\longrightarrow D_\perf^{\GL_{M-1}\times \GL_N}(\mathfrak{B}_{M-1|N}^{2,0})\] induces an equivalence
\begin{equation}
\begin{split}
        \colim'_{\eta}\ D_\perf^{\GL_{M-1}\times \GL_N, \geq 0}(\mathfrak{B}_{M-1|N}^{2,0})&\simeq \colim'_{\eta}\ D_\perf^{\GL_{M-1}\times \GL_N}(\mathfrak{B}_{M-1|N}^{2,0})\\
        &\simeq  D_\perf^{\GL_{M-1}\times \GL_N}(\mathfrak{B}_{M-1|N}^{2,0}).
\end{split}
\end{equation}

In particular, $\colim'_{\eta}\ D_\perf^{\GL_{M-1}\times \GL_M, \geq 0}(\mathfrak{B}_{M-1|M}^{2,0})\simeq D_\perf^{\GL_{M-1}\times \GL_N}(\mathfrak{B}_{M-1|M}^{2,0})$.

By taking ind-completion, we obtain
\[\colim_{\eta}\ D^{\GL_{M-1}\times \GL_N, \geq 0}(\mathfrak{B}_{M-1|N}^{2,0})\simeq D^{\GL_{M-1}\times \GL_N}(\mathfrak{B}_{M-1|N}^{2,0}),\]
\[\colim_{\eta}\ D^{\GL_{M-1}\times \GL_M, \geq 0}(\mathfrak{B}_{M-1|M}^{2,0})\simeq D^{\GL_{M-1}\times \GL_N}(\mathfrak{B}_{M-1|M}^{2,0}).\]

To prove \eqref{2.10}, it suffices to prove the following equivalence:
\begin{equation}\label{2.12}
\begin{split}
    D^{\GL_{M-1}\times \GL_M, \geq 0}(\mathfrak{B}_{M-1|M}^{2,0})&\otimes_{D^{\GL_{M}\times \GL_M}(\mathfrak{B}_{M|M}^{2,0})} D^{\GL_M\times \GL_N}(\mathfrak{B}_{M|N}^{2,0})\\ \simeq& \\D^{\GL_{M-1}\times \GL_N, \geq 0}&(\mathfrak{B}_{M-1|N}^{2,0}).
\end{split}
\end{equation}
Then, the equivalence \eqref{2.10} follows by taking colimits in the category of $(\infty,1)$-stable {cocomplete} categories with continuous functors on both sides.

To prove the equivalence \eqref{2.12}, it suffices to use the convolution product defined in~\S\ref{convolution} and Lemma \ref{lemma 2.1}.

Indeed, by Corollary \ref{cor 2.5.2}, 
\begin{equation*}
\begin{split}
      \textnormal{LHS of } \eqref{2.12}&\simeq D(T_{M,M})\otimes_{D^{\GL_{M}\times \GL_M}(\mathfrak{B}_{M|M}^{2,0})} D^{\GL_M\times \GL_N}(\mathfrak{B}_{M|N}^{2,0})\\
      &\simeq \tilde{c}\overset{A}{\star} D^{\GL_{M}\times \GL_M}(\mathfrak{B}_{M|M}^{2,0}) \otimes_{D^{\GL_{M}\times \GL_M}(\mathfrak{B}_{M|M}^{2,0})} D^{\GL_M\times \GL_N}(\mathfrak{B}_{M|N}^{2,0})\\
      &\simeq \tilde{c}\overset{A}{\star} D^{\GL_M\times \GL_N}(\mathfrak{B}_{M|N}^{2,0})\\
      &\simeq D(T_{M,N})\\
      &\simeq D^{\GL_{M-1}\times \GL_N, \geq 0}(\mathfrak{B}_{M-1|N}^{2,0})\\
      &\simeq \textnormal{RHS of }\eqref{2.12}.
\end{split}
\end{equation*}

The proof for other sheaf categories follows from the same proof and Lemma \ref{lemma 2.1} and Corollary \ref{cor 2.5.3}.
\end{proof}

\section{D-module side}\label{cons sid}
In this section, we will prove Theorem \ref{claim1}, which is the D-module side analog of Proposition \ref{prop 2.1}. 

\subsection{Definition of D-modules on ind-pro-schemes}\label{sec 3.1}
Since we need to handle with schemes of ind-pro-finite type, we need to take care when we define $D^{\GL_M(\bfO)}(\Gr_M\times \bfF^M)$, $D^{\GL_M(\bfO) \ltimes U_{M, N}(\bfF), \chi_{M, N}}\left(\Gr_{N}\right)$, etc. In this section, let us recall the definition of the category of D-modules on ind-pro-schemes. 

Assume $Y=\colim\ Y_i$ is an ind-pro-scheme, and schemes $Y_i$ are of pro-finite type. Let us denote the transition map by
\[\iota_{i,i'}: Y_i\hookrightarrow Y_{i'}.\]

Since each scheme $Y_i$ is of pro-finite type, we can write $Y_i$ as a limit 
\[Y_i= {\lim}\ Y_{i,j}\]
such that each $Y_{i,j}$ is of finite type and the transition functors
\[\pi_{i, j,j'}: Y_{i,j}\twoheadrightarrow Y_{i,j'}\]
are smooth projections.

Let us define \[D_!(Y):= \text{lim}_{\iota^!}\colim_{\pi^!} D(Y_{i,j}),\] and
\[D_*(Y):= \colim_{\iota_*}\text{lim}_{\pi_*} D(Y_{i,j}).\]

There is a canonical equivalence \[(D_!(Y))^\vee= D_*(Y).\]

Furthermore, choosing a dimension theory of $Y$ (see \cite[Section 3.3.7]{[Ber]}) gives rise to an equivalence between $D_!(Y)$ and $D_*(Y)$. In particular, if $Y$ is an ind-scheme, there is a canonical identification of $D_!(Y)$ and $D_*(Y)$, we denote it by $D(Y)$ for short. 
\begin{lem}
Assume $G$ is a {group}  ind-pro-scheme, we denote by \[m: G\times G\longrightarrow G\] the multiplication map, and by \[\Delta: G\longrightarrow G\times G\] the diagonal embedding. 

Then $D_*(G)$ is a Hopf algebra with the product $m_*$ and the coproduct $\Delta_*$. Similarly, $D_!(G)$ is a Hopf algebra with the product $\Delta^!$ and the coproduct $m^!$.
\end{lem}
\subsection{Strong $G$-invariants}
Assume that $\mathcal{C}$ is a category with an action of $D_*(G)$, we define the category of $G$-invariants of $\cC$ as 
\[\cC^G:= \Hom_{D_*(G)}(\Vect, \cC).\] In other words, it is the totalization of
\begin{equation*}
\xymatrix{
  \cC \ar@<-.5ex>[r] \ar@<.5ex>[r] &\Hom(D_*(G), \cC)\ar@<-1ex>[r] \ar@<1ex>[r]\ar[r] & \Hom(D_*(G)\otimes D_*(G), \cC)\cdots.
}
\end{equation*}

If $\mu: G\longrightarrow \mathbb{G}_a$ is an additive character, we denote by $\mu^!(\exp)$ the character D-module on $G$ corresponding to $\mu$.  Let us denote by $\Vect_\mu$ the $D_*(G)$-module category with the underlying category $\Vect$ and the coaction corresponding to $\mu^!(\exp)$.

We define the category of $G$-invariants against $\mu$ of $\cC$ as 
\[\cC^{G,\mu}:= \Hom_{D_*(G)}(\Vect_\mu, \cC).\]

\subsection{Locally compact objects and renormalization} Given a category $\mathcal{C}$ with an action of $G$, there is a natural forgetful functor 
\[\mathcal{C}^G\longrightarrow \mathcal{C}.\]
We denote by $\mathcal{C}^{G, \mathsf{loc.c}}$ the full subcategory of $\mathcal{C}^G$ generated by the preimage of $\mathcal{C}^c$. We denote its ind-completion by $\mathcal{C}^{G, \ren}$.

Now we introduce some basic properties of $\mathcal{C}^{G, \ren}$ and $\cC^{G, \locc}$. 
\begin{prop}\label{prop 3.3.1}
If $U$ is pro-unipotent, $U$ acts on $\cC$, and $\cC^U$ is compactly generated, then 
\[\cC^{U, \ren}\simeq \cC^U.\]
\end{prop}
\begin{proof}
Since both $\cC^{U, \ren}$ and $\cC^U$ are compactly generated, we only need to verify the equivalence at compact objects level. Namely, we need to show that $c\in \cC^U$ is compact in $\cC$ if and only if $c$ is compact in $\cC^U$.

Assume $c\in \cC^U$ is compact in $\cC$, i.e.,
\[\Hom_{\cC}(\oblv(c), \colim\ d_i)\simeq \colim\ \Hom_{\cC}(\oblv(c), d_i)\]
for any filtered colimit of $d_i\in \cC$. By adjointness, 
\[\Hom_{\cC^U}(c, \Av_*^U(\colim\ d_i))\simeq \colim\ \Hom_{\cC}(c, \Av_*^U(d_i)).\]

By the assumption on $U$, $\Av_*^U$ is a continuous idempotent monad, so there is \[\Hom_{\cC^U}(c, \colim\Av_*^U( d_i))\simeq \colim\ \Hom_{\cC}(c, \Av_*^U(d_i)).\] 

For any object $d'\in \cC^U$, we have $\Av_*^U\circ \oblv (d')\simeq d'$. Combined with the above formula, \[\Hom_{\cC^U}(c, \colim\ d_i')\simeq \colim\ \Hom_{\cC}(c, d_i')\] for any $d_i'\in \cC^U$.

On the other hand, {let $c$ be} compact in $\cC^U$. Then $c$ is compact in $\cC$, since the forgetful functor $\cC^U\longrightarrow \cC$ admits a continuous right adjoint functor $\Av_*^U$. 
\end{proof}

\begin{rem}
The above proposition is false if $U$ is not of pro-finite type. For example, if $U$ is the loop group of a unipotent group, then the forgetful functor {$\cC^U\longrightarrow\cC$} does not preserve compact objects. 

The above proposition is also false for non-unipotent group (even in finite type case). For example, it is false in the case $\cC=\Vect$ and $U=\mathbb{G}_m$. In this case, the forgetful functor $\cC^U\longrightarrow \cC$ is conservative, so $\Vect^{\mathbb{G}_m, \locc}$ is generated by the essential image of the left adjoint functor $\Av_!^{\mathbb{G}_m}$ of the forgetful functor. The constant sheaf $\mathbb{C}$ is not in $\Vect^{\mathbb{G}_m, c}$, but in $\Vect^{\mathbb{G}_m, \locc}$.
\end{rem}

Recall the following lemma about invariants.
\begin{lem}\label{3.3 lem}
Assume $G:= H\ltimes S$ is a {group}  ind-pro-scheme. Let us take a character $\mu_1$ of $S$ and a character $\mu_2$ of $H$. We assume $\mu_1$ is stable under the conjugation action of $G$ on $S$, then $\mu_1$ and $\mu_2$ give rise to a character $\mu$ of $G$.

Let $\cC$ be a category admitting an action of $G$, then the category $\cC^{S,\mu_1}$ admits an action of $H$ and 
\begin{equation}
    (\cC^{S,\mu_1})^{H, \mu_2}\simeq (\cC)^{G,\mu}.
\end{equation}
\end{lem}
\begin{proof}
It follows from \cite[Lemma 6.5.4, Remark 6.5.6]{[Ber]} that for any category $\cC'$ admitting an action of $G$, we have 
\begin{equation}
    ((\cC')^{S})^{H}\simeq (\cC')^{G}.
\end{equation}
The lemma follows from the above equivalence immediately by letting $\cC':= \cC\otimes \Vect_\mu$.
\end{proof}

\begin{prop}\label{prop 3.3.4}
In the above lemma, if we assume $S$ is pro-unipotent, then \[\cC^{G,\mu, \ren}\simeq (\cC^{S,\mu_1})^{H,\mu_2,\ren}.\]
\end{prop}
\begin{proof}
With loss of generality, we only need to show the case without characters. The category $(\cC^{S})^{H,\ren}$ is the ind-completion of $(\cC^{S})^{H,\locc}$. By Proposition \ref{prop 3.3.1}, an object in $\cC^S$ is compact if and only if it is compact when regarded as an object in $\cC$. So, both $(\cC^{S})^{H,\locc}$ and $\cC^{G, \locc}$ are full subcategories of $\cC^G$ generated by the objects which are compact in $\mathcal{C}$.
\end{proof}

Now, we consider categories over quotient stacks (ref. \cite[Section 6.1]{[Ber]}). Let $X$ be a scheme and $G$ be an {algebraic}  group {(or: a group scheme)} which acts on $X$. By definition, a category $\cC$ is over $X/G$ if it admits an action of the monoidal category $D_!(X)\rtimes D_*(G)$.
\begin{ex}
Given a group $G$ and a subgroup $S\subset G$. Then for any category $\cC$ admitting an action of $S$, the category $\cC\otimes_{D_*(S)}D_*(G)$ is a category over $X/G$. Here $X:= G/S$ is the quotient scheme.
\end{ex}

The following lemma is from \cite[Theorem 6.4.2]{[Ber]}.
\begin{lem}\label{lem 3.3.5}
{Let} $\cC$ be a category over $X/G$. Assume $X$ is a pro-scheme and $G$ is a pro-group scheme which acts transitively on $X$. Then 
\[\cC^G\simeq (\cC_x)^S.\]
Here $\cC_x:= \cC\underset{D_!(X)}{\otimes}\Vect$ denotes cofiber of $\cC$ at $x\in X$, and $S=\Stab_G(x)$ is the stabilizer of $x$.
\end{lem}
\begin{prop}\label{prop 3.3.6 fin}
In the above situation, assume $\cC\simeq \cC_x\underset{D_*(S)}{\otimes} D_*(G)$. Then there is
\begin{equation}\label{ren 3.3 ren}
    \cC^{G,\locc}\subset (\cC_x)^{S, \locc}.
\end{equation}
\end{prop}
\begin{proof}
We need to prove  {that if} 
\begin{equation}\label{eq 1}
    \Hom_\cC(\oblv^G(c),\colim\ d)\simeq \colim\ \Hom_\cC(\oblv^G(c), d)
\end{equation}
for any $d\in \cC$, {then} 
\begin{equation}\label{eq 2}
    \Hom_{\cC_x}(i_x^!\oblv^G(c), \colim\ d')\simeq \colim\ \Hom_{\cC_x}(i_x^! \oblv^G (c), d')
\end{equation}
for any $d'\in \cC_x$. Here $i_x^!: \cC\longrightarrow \cC_x$ denotes the pullback of categories associated with $x\to X$.

Consider the following commutative diagram
\begin{equation}
    \xymatrix{
\cC^G\ar[r]^{(i_x^!)^G}\ar[d]^{\oblv^G}& (\cC_x)^S\ar[d]^{\oblv^S}\\
\cC\ar[r]^{i_x^!}& \cC_x,
}
\end{equation}
so that we have
\begin{equation}\label{comm 3.5}
    i_x^! \oblv^G\simeq \oblv^S (i_x^!)^G.
\end{equation}

Note that by Lemma \ref{lem 3.3.5}, $(i_x^!)^G$ is an equivalence, and therefore its inverse is its continuous right adjoint, denoted $(i_x^!)^{G,R}$.

Since $G$ and $S$ are {of pro-finite type}, $\oblv^G$ and $\oblv^S$ admit continuous right adjoint functors $\Av_*^G$ and $\Av_*^S$. So
\eqref{eq 1} {can be rewritten as follows:} 
\begin{equation}\label{eq 3}
\Hom_{\cC^G}(c, \colim\ \Av_*^G(d))\simeq \colim\ \Hom_{\cC^G}(c, \Av_*^G(d)),    
\end{equation}
and {similarly, \eqref{eq 2} becomes}
\begin{equation}\label{eq 4}
    \Hom_{(\cC_x)^S}((i_x^!)^G(c), \colim\ \Av_*^S(d'))\simeq \colim\ \Hom_{(\cC_x)^S}((i_x^!)^G(c), \Av_*^S(d')).
\end{equation}

By passing to the right adjoint of \eqref{comm 3.5}, we have
\begin{equation}\label{right 3.6}
    \Av_*^G (i_x^!)^R\simeq (i_x^!)^{G,R} \Av_*^S.
\end{equation}

Let $d= (i_x^!)^R (d')$, the left-hand side of \eqref{eq 3} is
\begin{equation}
\begin{split}
     \Hom_{\cC^G}(c, \colim\ \Av_*^G (i_x^!)^R d')&\simeq \Hom_{\cC^G}(c, \colim\ (i_x^!)^{G,R} \Av_*^S (d'))\\
     &\simeq \Hom_{\cC^G}(c, (i_x^!)^{G,R}\colim\ \Av_*^S (d')),
\end{split}
\end{equation}
and the right-hand side of \eqref{eq 3} is
\begin{equation}
\begin{split}
   \colim\  \Hom_{\cC^G}(c, \Av_*^G (i_x^!)^R d')&\simeq \colim\ \Hom_{\cC^G}(c,  (i_x^!)^{G,R} \Av_*^S (d')).
\end{split}
\end{equation}

By adjointness, it implies 
\begin{equation}
    \Hom_{(\cC_x)^S}((i_x^!)^G(c), \colim\ \Av_*^S (d'))\simeq \colim\ \Hom_{(\cC_x)^S}((i_x^!)^G (c), \Av_*^S(d')).
\end{equation}
\end{proof}

\begin{prop}\label{3.3.7}
In the above assumption, if $G$ and $S$ are of finite type,
then $\cC^{G, \locc}\supset (\cC_x)^{S, \locc}$. In particular, we have 
\[\cC^{G, \locc}\simeq (\cC_x)^{S, \locc},\]
and
\[\cC^{G, \ren}\simeq (\cC_x)^{S, \ren}.\]
\end{prop}
\begin{proof}
We only need to show that if the restriction of a $G$-equivariant object $c$ of $\cC\simeq \cC_x\otimes_{D_*(S)} D_*(G)$ to $\cC_x\simeq \cC_x\otimes_{D_*(S)} D_*(S)$ is compact, then $c$ is compact.

Assume $U':=\coprod U_i'$ is an \'{e}tale cover of $G$, such that each $U_i'$ is isomorphic to $S\times U_i$, the map $j_{U'\to G}\colon U'=S\times U:= S\times \coprod U_i \longrightarrow G$ is $S$-invariant and $S\times U_i$ is finite on its image. We claim that the restriction of $c$ to $\cC_x\otimes_{D_*(S)} D_*(S\times U)$ is compact.

Indeed, since the $S$-invariant map $U'\overset{j_{U'\to G}}{\hookrightarrow} G\overset{p_G}{\longrightarrow} \pt$ factors through the $S$-invariant map $U'\overset{p_{U'\to S}}{\longrightarrow} S\overset{p_S}{\longrightarrow} \pt$, the restriction of $c$ to $\cC_x\otimes_{D_*(S)} D_*(U')$ is isomorphic to the pullback of $c'\in \cC_x\otimes_{D_*(S)} \Vect$ along
\[\cC_x\otimes_{D_*(S)} \Vect\longrightarrow \cC_x\simeq \cC_x\otimes_{D_*(S)} D_*(S)\longrightarrow \cC_x\otimes_{D_*(S)} D_*(U').\]
Here $c'$ is descended from $c$, i.e., $p_G^!(c')\simeq c$.

By our assumption, $p_S^!(c')$ is compact in $\cC_x$. Furthermore, since $p_{U'\to S}$ is smooth, {$S\x U_i$} and $S$ are of finite type, the $*$-pullback functor $p^*_{U'\to S}$ is well-defined and equals $p^!_{U'\to S}$ up to a cohomology shift. In particular, $p^!_{U'\to S}$ preserves compactness. So the restriction of $c$ to $\cC_x\otimes_{D_*(S)} D_*(U')$ is compact.

Consider the \v{C}ech complex associated with $\{U_i'\}$ removing diagonals in the relative Cartesian products. Since finite limit commutes with filtered colimit, compactness of $j_{U'\to G}^!(c)$ implies the compactness of $c$.
%
%
%
\end{proof}
As a corollary of Proposition \ref{3.3.7}, we have 
\begin{cor}\label{prop 3.3.6}
If $X$ and $G$ are of pro-finite type, and $G= G_r\ltimes G_u$, $S= S_r\ltimes S_u$. Here $S_r\subset G_r$ are reductive and $S_u\subset G_u$ are pro-unipotent. Assume $\cC\simeq \cC_x\underset{D_*(S)}{\otimes} D_*(G)$. Then there is
\begin{equation}\label{ren 3.3}
    \cC^{G,\ren}\simeq (\cC_x)^{S, \ren}.
\end{equation}
\end{cor}
\begin{proof}
By Proposition \ref{prop 3.3.4}, we have
\begin{equation}
    (\cC)^{G, \ren}\simeq (\cC^{G_u})^{G_r,\ren},
\end{equation}
and
\begin{equation}
    (\cC_x)^{S, \ren}\simeq (\cC^{S_u})^{S_r,\ren}.
\end{equation}

Since $\cC\simeq \cC_x\underset{D_*(S)}{\otimes} D_*(G)$, there is
\begin{equation}\label{3.19}
    \begin{split}
        \cC^{G_u}\simeq \cC_x\otimes_{D_*(S)} D(G/G_u)
        \simeq \cC_x \otimes_{D_*(S_r)\otimes D_*(S_u)} D_*(G_r).
    \end{split}
\end{equation}
Since $S_u\subset G_u$ and $G_u$ is a normal subgroup of $G$, the action of $S_u$ on $G/G_u$ is trivial. In particular, $D_*(S_u)$ acts on $D_*(G_r)$ trivially. There is
\begin{equation}\label{3.20}
    \cC_x \otimes_{D_*(S_r)\otimes D_*(S_u)} D_*(G_r)\simeq (\cC_x)^{S_u} \otimes_{D_*(S_r)} D_*(G_r).
\end{equation}

Note that $(\cC_x)^{S_u} \otimes_{D_*(S_r)} D_*(G_r)$ is a category over $X_r/ G_r$. Here $X_r:= G_r/S_r$. By Proposition \ref{3.3.7}, we have
\begin{equation}\label{3.21}
    ((\cC_x)^{S_u} \otimes_{D_*(S_r)} D_*(G_r))^{G_r, \ren}\simeq (\cC_x^{S_u})^{S_r, \ren}.
\end{equation}

Apply Proposition \ref{prop 3.3.4} again, we obtain
\begin{equation*}
(\cC_x^{S_u})^{S_r, \ren}\simeq (\cC_x)^{S, \ren}.
\end{equation*}

Now, by \eqref{3.19} \eqref{3.20} and \eqref{3.21}, we have
\[\cC^{G,\ren}\simeq (\cC_x^{S_u})^{S_r, \ren}\simeq (\cC_x)^{S, \ren}.\]
\end{proof}






\subsection{Conventions}\label{3.3.1} 
Similar to the coherent side, there are four categories on the D-module side corresponding to $D_{\perf}$, ${{D}}$, $D_{\perf}$ with nilpotent condition, and ${{D}}$ with nilpotent condition.

In order to simplify the notation, let us make the following conventions.

If $M<N$, we denote by $C_{M|N}$ the DG category of $({\GL}_{M}({\bfO}) \ltimes U_{M, N}(\bfF), \chi_{M, N})$-equivariant D-modules on $\Gr_N$; if $M>N$, we denote by $C_{M|N}$ the DG category of right $({\GL}_{N}({\bfO}) \ltimes U_{N, M}(\bfF), \chi_{N, M})$-equivariant D-modules on $\GL_M(\bfO)\backslash \GL_M(\bfF)$; and we denote by $C_{M|M}$ the DG category of $\GL_M(\bfO)$-equivariant D-modules on $\Gr_{M}\times \bfF^M$.\footnote{It does not matter if we let $C_{M|M}$ be $D_!^{\GL_M(\bO)}(\Gr_M\times \bF^M)$ or $D_*^{\GL_M(\bO)}(\Gr_M\times \bF^M)$, since there is a (monoidal) equivalence between them.} The categories of their compacts are denoted by $C_{M|N}^c$, $C_{M|M,!}^c$ and $C_{M|M,*}^c$, respectively.

If $M<N$, there is a forgetful functor from $C_{M|N}$ to the category of $(U_{M,N}(\bF),\chi_{M,N})$-equivariant D-modules on $\Gr_N$. We denote by $C_{M|N}^{\mathsf{loc.c}}$ the full subcategory of $C_{M|N}$ spanned by the preimage of compact $(U_{M,N}(\bF),\chi_{M,N})$-equivariant D-modules. Furthermore, we denote by $C_{M|N}^{\mathsf{ren}}$ the ind-completion of $C_{M|N}^{\mathsf{loc.c}}$. Similarly for $M>N$ and $M=N$.

In this paper, we will prove that 
\begin{enumerate}
    \item $C^{\mathsf{ren}}_{M|N}$ is equivalent to ${{D}}^{\GL_{M}\times \GL_N}(\mathfrak{B}_{M|N}^{2,0})$, 
    \item $C^{\mathsf{loc.c}}_{M|N}$ is equivalent to $D_\perf^{\GL_{M}\times \GL_N}(\mathfrak{B}_{M|N}^{2,0})$, 
    \item $C_{M|N}$ is equivalent to ${{D}}^{\GL_{M}\times \GL_N}(\mathfrak{B}_{M|N}^{2,0})_\nilp$, 
    \item $C_{M|N}^c$ is equivalent to $D_\perf^{\GL_{M}\times \GL_N}(\mathfrak{B}_{M|N}^{2,0})_\nilp$.
\end{enumerate}

The following easy lemma establishes an equivalence between $C_{M|N}$ and $C_{N|M}$ (similarly for $C^c$, $C^\locc$ and $C^\ren$).

\begin{lem}\label{lem 3.4.1}
Pushforward along the map $g\in \GL_N(\bfF)\longrightarrow g^{-1}\in \GL_N(\bfF)$ induces an equivalence functor $C_{M|N}\simeq C_{N|M}.$\footnote{Replace $\chi_{N,M}$ by $-\chi_{N,M}$ in the definition of $C_{N|M}$.}
\end{lem}
\begin{proof}
Pushforward along the map $g\in \GL_N(\bfF)\longrightarrow g^{-1}\in \GL_N(\bfF)$ sends left $\GL_M(\bfO)$-equivariant sheaves to right $\GL_M(\bfO)$-equivariant sheaves, and sends left $(U_{M, N}(\bfF), \chi_{M, N})$-equivariant sheaves to right $(U_{N,M}(\bfF), -\chi_{M, N})$-equivariant sheaves. It induces a functor $C_{M|N}\longrightarrow C_{N|M}$. Similarly, pushforward along the map $g\in \GL_N(\bfF)\longrightarrow g^{-1}\in \GL_N(\bfF)$ induces a functor $C_{N|M}\longrightarrow C_{M|N}$. One checks easily that these two functors are inverse to each other.
\end{proof}

\subsection{Convolution monoidal structure of $C_{M|M}$}\label{conv cons}
If $M=N$, we can define convolution monoidal structures for any sheaf theory of $C^{\mathsf{ren}}_{M|M}$, $C^{\mathsf{loc.c}}_{M|M}$, $C_{M|M}$ and $C_{M|M}^c$. For simplicity, we only {demonstrate the case} of $C_{M|M}$ in Section \ref{conv cons}-\ref{sec 3.7}. In these sections, let us recall the definitions of the monoidal category structures of $C_{M|M}$, and the actions of $C_{M|M}$ on $C_{M|N}$ (from left) and $C_{M-1|M}$ (from right).

The monoidal category structures of $C_{M|M}$ are defined in \cite[Section 3.2]{[BFGT]}. Actually, there are two different monoidal category structures on $C_{M|M}$, which are related by the Fourier transform.

The first one is the na\"{i}ve convolution product induced by the multiplication. For groups $G\subset H$, there is a convolution product functor
\begin{equation}\label{3.2}
    D(G\backslash H/G)\times D(G\backslash H/G)\longrightarrow D(G\backslash H\overset{G}{\times} H/G)\longrightarrow D(G\backslash H/G).
\end{equation} 
Let $H$ be the transpose of the mirabolic subgroup
$\Mir^t_{M+1}(\bfF)=
\left(
\begin{matrix}
\GL_M(\bfF) &* \\
  &1
\end{matrix}
\right)
$ of $\GL_{M+1}(\bfF)$, and  $G$ be $\GL_M(\bfO)=\left(
\begin{matrix}
\GL_M(\bfO) & \\
  &1
\end{matrix}
\right){{}\subset\GL_{M+1}(\bfO)\subset\GL_{M+1}(\bfF)} $. {Applying~\eqref{3.2}, we thus obtain the convolution product on} \[C_{M|M}= D^{\GL_M(\bfO)}(\Gr_M\times \bfF^M)= \linebreak 
D_*^{\GL_M(\bfO)}(\Mir^t_{M+1}(\bfF)/\GL_M(\bfO))={D_*(G\bsl H/G)}.\]

To be more precise, consider the following diagram, 
\[\xymatrix{H\times H/G\ar[r]^-{ q}\ar[d]^{p\times id}& H\overset{G}{\times} H/G\ar[d]^m\\
H/G\times H/G& H/G,
}\]
where $p$, $q$ denote the projection maps, $m$ denotes the multiplication map.

Given $\mathcal{F}_1, \mathcal{F}_2\in D_?^{\GL_M(\bfO)}(\Gr_M\times \bfF^M)$ (here ?=!,*), the pullback of $\mathcal{F}_1\boxtimes \mathcal{F}_2$ along $p\times id$ descends to a sheaf $\mathcal{F}_1\tilde{\boxtimes} \mathcal{F}_2$ on $\Mir^t_{M+1}(\bfF)\overset{\GL_M(\bfO)}{\times} \Mir^t_{M+1}(\bfF)/\GL_M(\bfO)$. We set $\mathcal{F}_1\overset{}{\textasteriskcentered}\mathcal{F}_2:= m_*(\mathcal{F}_1\tilde{\boxtimes} \mathcal{F}_2)$.

\begin{rem} We can also use the {(non-transposed)}  mirabolic subgroup $\Mir_{M+1}(\bfF)=
\left(
\begin{array}{ccc}
\GL_M(\bfF) & \\
*  &1
\end{array}
\right)$ to define the above monoidal structure. Indeed, we only need to notice that the composition of taking transpose and taking inverse induces a monoidal {equivalence} between $D(\GL_M(\bfO)\backslash \Mir^{t}_{M+1}(\bfF)/\GL_M(\bfO))$ and $D(\GL_M(\bfO)\backslash \Mir_{M+1}(\bfF)/\GL_M(\bfO))$.
\end{rem}

\subsection{The second monoidal structure}\label{restr of mon}
There is another monoidal category structure on $D^{\GL_M(\bfO)}(\Gr_M\times \bfF^M)$ defined in \cite{[BFGT]}.

We note that $\Gr_M\times \bfF^M= \GL_M(\bfF)\overset{\GL_M(\bfO)}{\times} \bfF^M$. For $g\in \GL_M(\bfF), v\in \bfF^M$, we denote by $[g,v]$ the corresponding point in $\Gr_M\times \bfF^M$. For $h\in \GL_M(\bfF)$, $h\cdot [g,v]= [h\cdot g,v]$. Under the isomorphism $\Gr_M\times \bfF^M= \Mir^t_{M+1}(\bfF)/\GL_M(\bfO)$, $[g,v]$ is the image of $(g,v)$ under the map $\GL_M(\bfF)\ltimes \bfF^M =\Mir^t_{M+1}(\bfF)\longrightarrow \Mir^t_{M+1}(\bfF)/\GL_M(\bfO)\longrightarrow \Gr_M\times \bfF^M$.

{Consider the following maps}
\begin{equation*}
    \xymatrix{\GL_M(\bfF)\times (\Gr_M\times \bfF^M)\ar[r]^-{\mathfrak{q}}\ar[d]^{\mathfrak{p}}& \GL_M(\bfF)\overset{\GL_M(\bfO)}{\times} \Gr_M\times \bfF^M\ar[d]^{\mathfrak{m}}\\
(\Gr_M\times \bfF^M)\times (\Gr_M\times \bfF^M)& \Gr_M\times \bfF^M.}
\end{equation*}

Here, \[\mathfrak{p}(g_1, [g_2,v])= ([g_1, g_2v], [g_2,v]),\]
\[\mathfrak{q}(g_1, [g_2,v])= [g_1,[g_2,v]],\]
\[\mathfrak{m}([g_1, [g_2, v]])=[g_1g_2, v].\]

Given $\mathcal{F}_1, \mathcal{F}_2\in D_?^{\GL_M(\bfO)}(\Gr_M\times \bfF^M)$, the pullback of $\mathcal{F}_1\boxtimes \mathcal{F}_2$ along $\mathfrak{p}$ descends to a sheaf $\mathcal{F}_1\chkbx  \mathcal{F}_2$ on $\GL_M(\bfF)\overset{\GL_M(\bfO)}{\times} \Gr_M\times \bfF^M$. We set $\mathcal{F}_1\overset{}{\circledast}\mathcal{F}_2:= \mathfrak{m}_!(\mathcal{F}_1\chkbx  \mathcal{F}_2)=\mathfrak{m}_*(\mathcal{F}_1\chkbx  \mathcal{F}_2)$.

\begin{rem}
Up to formal issues related to infinite-dimensionality, the $\circledast$  monoidal structure is a particular case of the monoidal structure on $D(H\backslash(G\times X)/H)$ where $X$ is a variety acted on by an algebraic group $G$ with a subgroup $H\subset G$, the right $H$-action on $G\times X$ is along the $G$ factor, while the left $H$-action is diagonal. (This is applied to $G=\GL_M(\bfF),\ H=\GL_M(\bfO)$.)
\end{rem}

\begin{rem}\label{rem 3.7.1}
On the subcategory of D-modules supported on $\Gr_M\times \{0\} \subset \Gr_M\times \bfF^M$, the functors\[\circledast, \textasteriskcentered\colon D_?^{\GL_M(\bfO)}(\Gr_M\times \bfF^M)\times D_?^{\GL_M(\bfO)}(\Gr_M\times \bfF^M)\longrightarrow D_?^{\GL_M(\bfO)}(\Gr_M\times \bfF^M)\]restrict to the usual convolution product of the spherical Hecke category.
\end{rem}

On the subcategory of constructible sheaves supported on $\{1\}\times\bfF^M \subset \Gr_M\x\bfF^M$ (i.e.,  require $g=1$), the functor \[\circledast\colon D_!^{\GL_M(\bfO)}(\Gr_M\times \bfF^M)\times D_!^{\GL_M(\bfO)}(\Gr_M\times \bfF^M)\longrightarrow D_!^{\GL_M(\bfO)}(\Gr_M\times \bfF^M)\] restricts to the  tensor product  of $D_!^{\GL_M(\bfO)}({\bfF^M})$, and the functor \[\textasteriskcentered\colon D_*^{\GL_M(\bfO)}(\Gr_M\times \bfF^M)\times D_*^{\GL_M(\bfO)}(\Gr_M\times \bfF^M)\longrightarrow D_*^{\GL_M(\bfO)}(\Gr_M\times \bfF^M)\] restricts to the convolution product of $D_*^{\GL_M(\bfO)}({\bfF^M})$.

\begin{rem} $\circledast, \textasteriskcentered$ are related by the Fourier transform. That is to say, the Fourier transform functor
\begin{equation}
    \begin{split}
        FT\colon D_!^{\GL_M(\bfO)}(\Gr_M\times \bfF^M)\longrightarrow D_*^{\GL_M(\bfO)}(\Gr_M\times (\bfF^M)^\vee)
    \end{split}
\end{equation}
 intertwines $\circledast$ and $\textasteriskcentered$. 
 
 By a choice of basis, we can identify $\bfF^M$ with $(\bfF^M)^\vee$. Accordingly, the action of $\GL_M(\bO)$ becomes the original action composed with $g\to (g^t)^{-1}$. For more details, see \cite[Section 3.15]{[BFGT]}.
\end{rem}

Regarding Theorem \ref{NN}, later on, when we use the monoidal structure $\circledast$, we identify $C_{M|M}$ as $D_!^{\GL_M(\bfO)}(\Gr_M\times \bfF^M)$, and when we use $\textasteriskcentered$, we identify $C_{M|M}$ as $D_*^{\GL_M(\bfO)}(\Gr_M\times \bfF^M)$.

\subsection{Left action of $C_{M|M}$}\label{def of mon}
In this section, we will define the left action of $C_{M|M}$ on $C_{M|N}$.  Since $FT$ is an equivalence of monoidal categories, the category of the categories admitting an action of $(D_*^{\GL_M(\bfO)}(\Gr_M\times \bfF^M), \textasteriskcentered)$ is equivalent to the category of the categories admitting an action of $(D_!^{\GL_M(\bfO)}(\Gr_M\times \bfF^M), \circledast)$. 
Hence, we only need to construct such an action for $(C_{M|M}, \textasteriskcentered)$.

Assume $N>M$, let us consider the following diagram,
\begin{equation*}
    \xymatrix{\Mir^t_{M+1}(\bfF)\times \Gr_N\ar[r]^{}\ar[d]^{p\times id}& \Mir^t_{M+1}(\bfF)\overset{\GL_M(\bfO)}{\times} \Gr_N\ar[d]^{m_N}\\
\Mir^t_{M+1}(\bfF)/\GL_M(\bfO)\times \Gr_N& \Gr_N.
}
\end{equation*}

Given $\mathcal{F}_1\in D_*^{\GL_M(\bfO)}(\Gr_M\times \bfF^M)$ and $\mathcal{F}_2\in D^{{\GL}_M({\bfO}) \ltimes U_{M, N}(\bfF), \chi_{M, N}}(\Gr_N)$, the pullback of $\mathcal{F}_1\boxtimes \mathcal{F}_2$ along $p\times id$ descends to a sheaf $\mathcal{F}_1\tilde{\boxtimes} \mathcal{F}_2$ on $\Mir^t_{M+1}(\bfF)\overset{\GL_M(\bfO)}{\times} \Gr_N$. We define $\mathcal{F}_1\overset{}{\textasteriskcentered}\mathcal{F}_2:= m_{N,*}(\mathcal{F}_1\tilde{\boxtimes} \mathcal{F}_2)$.

Since $U_{M,N}(\bfF)$ is a normal subgroup of ${\GL}_M({\bfO}) \ltimes U_{M, N}(\bfF)$ and $\chi_{M,N}$ is stable under the conjugation action of $\GL_M(\bfO)$, the resulting sheaf $\mathcal{F}_1\overset{}{\textasteriskcentered}\mathcal{F}_2$ is still $({\GL}_M({\bfO}) \ltimes U_{M, N}(\bfF), \chi_{M, N})$-equivariant. We get an action of $C_{M|M}$ on $C_{M|N}$.

\subsection{Right action on $C_{M-1|M}$}\label{sec 3.7}
In this section, we will construct a right action of $C_{M|M}$ on $C_{M-1|M}$.

Let us denote by $j_*\in D_!^{\GL_M(\bfO)}(\Gr_M\times \bfF^M)$ the $*$-extension of the constant sheaf on $1\times (\bfO^M\setminus t\bfO^M)$ along the locally closed embedding 
\[j\colon 1\times (\bfO^M\setminus t\bfO^M) \longrightarrow \Gr_M\times \bfF^M,\] 
then taking the convolution with $j_*$ gives an endofunctor of $C_{M|M}$,
\begin{equation}
\begin{split}
        j_*\circledast -\colon C_{M|M}\longrightarrow C_{M|M}\\
        \mathcal{F}\mapsto j_* \circledast \mathcal{F}.
\end{split}
\end{equation}

It is easy to see that $j_*$ is an idempotent algebra object of ($C_{M|M}, \circledast$), i.e., $j_*\circledast j_*\simeq j_*$ with the associativity. In particular, the Kleisli category $j_*\circledast C_{M|M}$ is a full subcategory of $C_{M|M}$.

By definition, $j_*\circledast C_{M|M}$ is equivalent to $D^{\Mir_M(\bfO)}(\Gr_M)$. Indeed, since $\Stab_{\GL_M(\bfO)}(e_{M})= \Mir_M(\bfO)$, we have
\begin{equation}
    D^{\Mir_M(\bfO)}(\Gr_M)\simeq D_!^{\Mir_M(\bfO)}(\Gr_M\times e_{M})\simeq D_!^{\GL_M(\bfO)}(\Gr_M\times (\bfO^M\setminus t\bfO^M)).
\end{equation}
The second  isomorphism is by \cite[Corollary 6.2.5]{[Ber]}.

Let $\xi_M$ be the automorphism of $\Gr_M$ which is induced by sending $e_M$ to $e_M$ and $e_i$ to $te_i$ for any $i=1,2,\ldots,M-1$. Pushforward along $\xi_M$ induces an endofunctor of $D^{\Mir_M(\bfO)}(\Gr_M)$.

By \cite[Section 4]{[BFGT]}, we have
\begin{equation}\label{3.5-}
    C_{M-1|M}\simeq \colim_{\xi_{M,*}}\ D^{\Mir_M(\bfO)}(\Gr_M).
\end{equation}
We denote by 
\begin{equation}
    \zeta_M\colon j_*\circledast C_{M|M}\longrightarrow j_*\circledast C_{M|M}
\end{equation}
the corresponding transition functor corresponding to $\xi_{M,*}$ under the equivalence $D^{\Mir_M(\bfO)}(\Gr_M)\simeq j_*\circledast C_{M|M}$.

Note that the transition functor $\xi_{M,*}$ is given by left multiplication with a diagonal element
$D_M=\operatorname{diag}(t,t,\ldots,t,1)$ in $\GL_M(\bfF)$, hence the right action of $C_{M|M}$ on $j_*\circledast C_{M|M}$ (resp. $D^{\Mir_M(\bfO)}(\Gr_M)$) is compatible with respect to the transition functor $\zeta_M$ (resp. $\xi_{M,*}$). By taking the colimit of the action of $C_{M|M}$ on $j_* \circledast C_{M|M}$, we obtain an action of $C_{M|M}$ on $\colim\ j_*\circledast C_{M|M}= C_{M-1|M}$.

\begin{rem}
We note that we can also describe $j_*\circledast C_{M|M}$ as a $\Mir^t_M(\bfO)$-equivariant category instead of a $\Mir_M(\bfO)$-equivariant category. Namely, by $g\longrightarrow g^{t,-1}$, there is an equivalence $D^{\Mir^t_M(\bfO)}(\Gr_M)\simeq D^{\Mir_M(\bfO)}(\Gr_M)\simeq j_*\circledast C_{M|M}$.
\end{rem}

If we use the identification $j_*\circledast C_{M|M}\simeq  D^{\Mir^t_{M}}(\Gr_M)$, we need to replace the transition functor $\xi_{M,*}$ by $\xi_{M,*}^{-1}$ in \eqref{3.5-}, i.e.,
\begin{equation}
    C_{M-1|M}\simeq \colim_{\xi^{-1}_{M,*}}\ D^{\Mir^t_M(\bfO)}(\Gr_M).
\end{equation}





\subsection{Tensor product of Gaiotto categories}Now we are going to prove the D-module side analog of  \eqref{2.10}. 

\begin{thm}\label{claim1}
For $M\leq N$, we have
\begin{equation}\label{thm 3.1}
    C^{\mathsf{ren}}_{M-1|M}\underset{C^{\mathsf{ren}}_{M|M}}{\otimes}{C^{\mathsf{ren}}_{M|N}}\simeq C^{\mathsf{ren}}_{M-1|N}.
\end{equation}

Similar claim holds for other sheaf theories, {e.g}.,
\begin{equation}\label{thm 3.1-1}
    C_{M-1|M}\underset{C_{M|M}}{\otimes}{C_{M|N}}\simeq C_{M-1|N}.
\end{equation}
\end{thm}

\subsection{Mirabolic equivariant category}
To prove Theorem \ref{claim1}, we need to consider the mirabolic equivariant category. Let us denote by $D^{\Mir^t_{M+1}({\bfO}) \ltimes U_{M, N}(\bfF), \chi_{M, N}}(\Gr_N)$ the category of left $(\Mir^t_{M+1}({\bfO}) \ltimes U_{M, N}(\bfF), \chi_{M, N})$-equivariant D-modules on $\Gr_N$. This category is well-defined because $\chi_{M,N}$ is stable under the conjugation action of $\Mir^t_{M+1}(\bfO)$. Similarly, we can define the renormalized mirabolic equivariant category. Namely, we denote by $(D^{U_{M, N}(\bfF), \chi_{M, N}}(\Gr_N))^{\Mir^t_{M+1}({\bfO}),\ren}$ the ind-completion of $(D^{U_{M, N}(\bfF), \chi_{M, N}}(\Gr_N))^{\Mir^t_{M+1}({\bfO}),\locc}$. Here $(D^{U_{M, N}(\bfF), \chi_{M, N}}(\Gr_N))^{\Mir^t_{M+1}({\bfO}),\locc}$ is the  category of locally compact left $(\Mir^t_{M+1}({\bfO}) \ltimes U_{M, N}(\bfF), \chi_{M, N})$-equivariant D-modules on $\Gr_N$. 

Since $\GL_{M}(\bfO)$ is a subgroup of $\Mir^t_{M+1}(\bfO)$, there is a forgetful functor 
\begin{equation}
    \Res_{M,N}\colon (D^{U_{M, N}(\bfF), \chi_{M, N}}(\Gr_N))^{\Mir^t_{M+1}({\bfO})}\longrightarrow (D^{U_{M, N}(\bfF), \chi_{M, N}}(\Gr_N))^{\GL_{M}({\bfO})}.
\end{equation}

Denote by $\xi_{M,N}$ the automorphism of $\Gr_N$ which sends $e_i$ to $t\cdot e_i$, if $1\leq i\leq M$, and $e_{i}$ to $e_i$, if $M+1\leq i\leq N$. The action of $\xi_{M,N}$ on $\Gr_N$ is given by left multiplication with the $N\times N$ diagonal matrix $D_{M,N}=(t, t,..., t, 1,1,...,1)$. 
By a simple calculation, \[{\GL}_{M}({\bfO}) \ltimes U_{M, N}(\bfF)= D^{-1}_{M,N}({\GL}_{M}({\bfO}) \ltimes U_{M, N}(\bfF))D_{M,N}\] \[(\textnormal{resp.}\qquad \Mir^t_{M+1}(\bfO)) \ltimes U_{M, N}(\bfF)\subset D_{M,N}^{-1}(\Mir^t_{M+1}(\bfO)) \ltimes U_{M, N}(\bfF))D_{M,N}).\] Furthermore, the character $\chi_{M,N}$ of $U_{M,N}(\bfF)$ is stable under the conjugation by $D^{-1}_{M,N}$, hence the functor $\xi^{-1}_{M,N,*}$ induces an auto-equivalence (resp. endofunctor) of $D^{{\GL}_M({\bfO}) \ltimes U_{M, N}(\bfF), \chi_{M, N}}(\Gr_N)$ (resp. $D^{\Mir^t_{M+1}(\bfO)) \ltimes U_{M, N}(\bfF), \chi_{M, N}}(\Gr_N)$).

\begin{equation}
\begin{split}
    \xi^{-1}_{M,N,*}\colon (D^{U_{M, N}(\bfF), \chi_{M, N}}(\Gr_N))^{\GL_{M}({\bfO})}\longrightarrow (D^{U_{M, N}(\bfF), \chi_{M, N}}(\Gr_N))^{\GL_{M}({\bfO})},\\
    \xi^{-1}_{M,N,*}\colon (D^{U_{M, N}(\bfF), \chi_{M, N}}(\Gr_N))^{\Mir^t_{M+1}(\bfO)}\longrightarrow (D^{U_{M, N}(\bfF), \chi_{M, N}}(\Gr_N))^{\Mir^t_{M+1}(\bfO)}.
\end{split}
\end{equation}

We will prove the following Proposition after a short preparation.
\begin{prop}\label{lemma 3.1}
Taking colimit of $\Res_{M,N}$ with respect to the transition functor $\xi^{-1}_{M,N,*}$ induces an equivalence of categories. That is to say, the functor\[\colim_{\xi^{-1}_{M,N,*}} (D^{U_{M, N}(\bfF), \chi_{M, N}}(\Gr_N))^{\Mir^t_{M+1}({\bfO})}\longrightarrow \colim_{\xi^{-1}_{M,N,*}} (D^{U_{M, N}(\bfF), \chi_{M, N}}(\Gr_N))^{\GL_{M}({\bfO})}\]is an equivalence. In particular, we have
\begin{equation}\label{3.5}
    \colim_{\xi^{-1}_{M,N,*}} (D^{U_{M, N}(\bfF), \chi_{M, N}}(\Gr_N))^{\Mir^t_{M+1}({\bfO})}\simeq (D^{U_{M, N}(\bfF), \chi_{M, N}}(\Gr_N))^{\GL_{M}({\bfO})}.
\end{equation}
\end{prop}

Note that $\Mir^t_{M+1}(\bfO)$ is a pro-group scheme, we denote by $D_*(\Mir^t_{M+1}(\bfO))$ the category of D-modules defined in ~\S\ref{sec 3.1}. It is a monoidal category with the convolution monoidal structure. 

Conjugating by $D_{M+1}$ induces a monoidal functor
\begin{equation}
    D_*(\Mir^t_{M+1}(\bfO))\longrightarrow D_*(\Mir^t_{M+1}(\bfO))
\end{equation}
which factors through $D_*(\GL_M(\bO)\ltimes t \bO^M)$.

In particular, it induces a functor
\begin{equation}
    F_{M+1}: \cC^{\Mir^t_{M+1}(\bfO)}\longrightarrow \cC^{\Mir^t_{M+1}(\bfO)}
\end{equation}
for any category $\cC$ with an action of $\Mir^t_{M+1}(\bfO)$. It factors through $\cC^{\GL_M(\bO)\ltimes t \bO^M}$.

\begin{lem}\label{colimit lemma}
There is an equivalence of categories
\begin{equation}
    \colim_{F_{M+1}} \cC^{\Mir^t_{M+1}(\bfO)}\simeq \cC^{\GL_M(\bfO)}.
\end{equation}
\end{lem}
\begin{proof}
Note that there is a commutative diagram,
\[
\xymatrix{\cC^{\Mir^t_{M+1}(\bfO)}\ar[r]^{F_{M+1}}\ar[d]^{\wr}&\cC^{\Mir^t_{M+1}(\bfO)}\ar[d]^{\wr}\\
\cC^{\GL_M(\bfO)\ltimes t^i {\bfO^{M}}}\ar[r]^{\oblv}&\cC^{\GL_M(\bfO)\ltimes t^{i+1} {\bfO^M} }
.}
\]
So, we have
\begin{equation}
    \colim_{F_{M+1}} \cC^{\Mir^t_{M+1}(\bfO)}\simeq \colim_{\oblv} \cC^{\GL_M(\bfO)\ltimes t^i {\bfO^{M}}}.
\end{equation}
Here, the transition functors of the {right}-hand side are forgetful functors. 

By Proposition \ref{prop 3.3.4}, {we have an equivalence} 
\begin{equation}
    \cC^{\GL_M(\bfO)\ltimes t^i {\bfO^{M}}}\simeq (\cC^{t^i \bfO^{M}})^{\GL_M(\bfO)}.
\end{equation}
Since $\GL_M(\bfO)$ is a pro-group scheme, taking $\GL_M(\bfO)$-invariants commutes with colimit. So we have \[(\colim\ \cC^{t^i \bfO^{M}})^{\GL_M(\bO)}\simeq {\colim\ }{((\cC^{t^i \bfO^{M}})^{\GL_M(\bfO)})}.\]
Now the claim follows from \cite[Proposition 4.3.8]{[Ber]} 
\[\colim_{\oblv}\ \cC^{t^i \bfO^{M}}\simeq \cC.\qedhere\]
\end{proof}

\begin{proof}[{Proof of Proposiition \ref{lemma 3.1}}]
We let $\mathcal{C}:=D^{U_{M,N}(\bF),\chi_{M,N}}(\Gr_N)$ in the above lemma, now it is sufficient to notice that $F^{M+1}$ coincides with $\xi_{M,N,*}^{-1}$.
\end{proof}

By definition, $\cC^{\Mir^t_{M+1}(\bO)}$ is a full subcategory of $\cC^{\GL_M(\bO)}$. The endofunctor $F_{M+1}$ of $\cC^{\Mir^t_{M+1}(\bfO)}$ comes from the restriction of the {auto}-equivalence of $\cC^{\GL_{M}(\bfO)}$ induced by conjugating with $D_{M+1}$. In particular, $F_{M+1}$ preserves {local} compactness {on}  $\cC^{\Mir^t_{M+1}(\bO)}$.  Hence $F_{M+1}$ induces a functor
\[F^\ren_{M+1}:\cC^{\Mir^t_{M+1}(\bO),\ren}\longrightarrow \cC^{\Mir^t_{M+1}(\bO),\ren}.\]

It is expected that there is an equivalence
\begin{equation}
    \colim_{F^\ren_{M+1}} \cC^{\Mir^t_{M+1}(\bfO),\ren}\simeq \cC^{\GL_M(\bfO),\ren}.
\end{equation}
In other words, there is an equivalence
\begin{equation}
    \colim_{\oblv} \cC^{\GL_{M}(\bfO)\ltimes t^i \bO^M,\ren}\simeq \cC^{\GL_M(\bfO),\ren}.
\end{equation}

Since taking ind-completion commutes with taking colimit, we only need to show 
\begin{equation}\label{3.22}
    \colim'_{\oblv} \cC^{\GL_{M}(\bfO)\ltimes t^i \bO^M,\locc}\simeq \cC^{\GL_M(\bfO),\locc}.
\end{equation}
It is not known if \eqref{3.22} is true for general $\cC$. However, it is true in the following case.
\begin{cor}
There is an equivalence of categories
\begin{equation}\label{3.5+}
    \colim'_{\xi^{-1}_{M,N,*}} (D^{U_{M, N}(\bfF), \chi_{M, N}}(\Gr_N))^{\Mir^t_{M+1}({\bfO}),\locc}\simeq (D^{U_{M, N}(\bfF), \chi_{M, N}}(\Gr_N))^{\GL_{M}({\bfO}),\locc}.
\end{equation}    
In particular,
\begin{equation}\label{3.5++}
    \colim_{\xi^{-1}_{M,N,*}} (D^{U_{M, N}(\bfF), \chi_{M, N}}(\Gr_N))^{\Mir^t_{M+1}({\bfO}),\ren}\simeq (D^{U_{M, N}(\bfF), \chi_{M, N}}(\Gr_N))^{\GL_{M}({\bfO}),\ren}.
\end{equation}    
\end{cor}
\begin{proof}
It is sufficient to show that any $\GL_M(\bO)\ltimes U_{M,N}(\bF)$-orbit $\mathbb{O}$ of $\Gr_N$ is $(\GL_M(\bO)\ltimes t^n \bO^M)\ltimes U_{M,N}(\bF)$-invariant for {sufficiently} large $n$.
Note that $(\GL_M(\bO)\ltimes t^n \bO^M)\ltimes U_{M,N}(\bF)$ is generated by $(\GL_M(\bO)\ltimes U_{M,N}(\bF))$ and $t^n \bO^M$, so 
{it} is sufficient to prove that $\mathbb{O}$ is $t^n \bO^M$-invariant for {sufficiently} large $n$. Here $t^n\bO^M$ is regarded as a subgroup of $\GL_N(\bF)$ {as follows:} 
\begin{equation*}
    t^n\bfO^M=
\begin{pNiceArray}{ccc|ccc}[last-col,code-for-last-col = \quad\,]
1 & &&  t^n\bO&&  \\
  &\Ddots& &  \Vdots&  & &M\\
 & &1 & t^n\bO &  & \\ \hline
 &&&1\\ 
 &            && &\Ddots& &N-M\\
&             &&  &&1
\CodeAfter
\UnderBrace[shorten,yshift=0.5ex]{1-1}{6-3}{M}
\UnderBrace[shorten,yshift=0.5ex]{1-4}{6-6}{N-M}
\SubMatrix .{1-6}{3-6}{\}}[right-xshift=1em]
\SubMatrix .{4-6}{6-6}{\}}[right-xshift=1em]
\end{pNiceArray} \subset \GL_N(\bfF).
\end{equation*}

\vspace{1.6em}
Fix $g\in \GL_N(\bF)$. We assume that entries of both $g$ and $g^{-1}$ {belong} to $t^{-m}\bO$. We claim that the orbit passing through $g$ is $t^{2m}\bO^M$-invariant.

Indeed, we have 
\[\GL_M(\bO)\ltimes U_{M,N}(\bF) \cdot g \cdot \GL_N(\bO)= \GL_M(\bO)\ltimes U_{M,N}(\bF)\cdot g \GL_N(\bO)g^{-1} \cdot g \cdot \GL_N(\bO),\]
\[t^{2m}\bO^M\cdot \GL_{M}(\bO)\ltimes U_{M,N}(\bF)= \GL_{M}(\bO)\ltimes U_{M,N}(\bF)\cdot t^{2m}\bO^M,\]
and 
\[t^{2m}\bO^M\subset g\GL_N(\bO)g^{-1}. \qedhere\]



\end{proof}
\subsection{Fourier transform on Gaiotto category} 
By Proposition \ref{lemma 3.1}, $C_{M-1|N}$ can be written as a colimit of $(D^{U_{M-1, N}(\bfF), \chi_{M-1, N}}(\Gr_N))^{{\Mir^t_{M}(\bfO)}}$ with the transition functor $\xi^{-1}_{M-1,N,*}$. 

As a result, we can write both sides of \eqref{thm 3.1} as colimits. Namely,
\begin{equation}
    \textnormal{LHS of \eqref{thm 3.1}}\simeq \colim_{\xi^{-1}_{M*}}\ D^{\Mir^t_M(\bfO)}(\Gr_M)\underset{C_{M|M}}{\otimes}{C_{M|N}},
\end{equation}
and 
\begin{equation}
    \textnormal{RHS of \eqref{thm 3.1}}\simeq \colim_{\xi^{-1}_{M-1,N,*}}\ (D^{U_{M-1, N}(\bfF), \chi_{M, N}}(\Gr_N))^{{\Mir^t_{M}(\bfO)}}.
\end{equation} Similarly for \eqref{thm 3.1-1}. 
To prove Theorem \ref{claim1}, it suffices to show the following proposition.
\begin{prop}\label{prop 3.1}
For $M\leq N$, we have

\textup{a)}.
\begin{equation}\label{3.23}
D^{\Mir^t_M(\bfO),\ren}(\Gr_M)\underset{C^\ren_{M|M}}{\otimes}{C^\ren_{M|N}}\simeq (D^{U_{M-1, N}(\bfF), \chi_{M, N}}(\Gr_N))^{{\Mir^t_{M}(\bfO)},\ren}.
\end{equation}

\textup{b)}.
\begin{equation}\label{eq 3.24}
D^{\Mir^t_M(\bfO)}(\Gr_M)\underset{C_{M|M}}{\otimes}{C_{M|N}}\simeq (D^{U_{M-1, N}(\bfF), \chi_{M, N}}(\Gr_N))^{{\Mir^t_{M}(\bfO)}}.
\end{equation}
\end{prop}

Let us first prove \eqref{eq 3.24}. 

Recall the definition of $(C_{M|M}, \circledast)$. The pushforward along the closed embedding
\[1\times \bfF^M\longrightarrow \Gr_M\times \bfF^M\]
gives a monoidal fully faithful embedding
\begin{equation}
    D_!^{\GL_M(\bfO)}(\bfF^M)\longrightarrow C_{M|M}.
\end{equation}
{Note that the image of this embedding contains the object $j_*$, so below we use the same notation for the corresponding object in the source category.}
\begin{lem}
For any category $\mathcal{D}$ admitting a left action of $(C_{M|M}, \circledast)$, we have 
\begin{equation}\label{3.30}
    (j_*\circledast C_{M|M})\underset{C_{M|M}}{\otimes} \mathcal{D}\simeq D_!^{\GL_{M}(\bfO)}(\bfO^M\setminus t\bfO^M)\underset{D_!^{\GL_M(\bfO)}(\bfF^M)}{\otimes} \mathcal{D}.
\end{equation}
\end{lem}
\begin{proof}
It follows from the fact that:
\begin{equation*}
    \begin{split}
       (j_*\circledast C_{M|M})\underset{C_{M|M}}{\otimes} \mathcal{D}
        &\simeq {(j_*\oast D_!^{\GL_{M}(\bfO)}(\bfF^M))\underset{D_!^{\GL_M(\bfO)}(\bfF^M)}{\otimes}C_{M|M}\underset{C_{M|M}}{\otimes} \mathcal{D}}\\
        &\simeq D_!^{\GL_{M}(\bfO)}(\bfO^M\setminus t\bfO^M)\underset{D_!^{\GL_M(\bfO)}(\bfF^M)}{\otimes}C_{M|M}\underset{C_{M|M}}{\otimes} \mathcal{D}\\
         &\simeq D_!^{\GL_{M}(\bfO)}(\bfO^M\setminus t\bfO^M)\underset{D_!^{\GL_M(\bfO)}(\bfF^M)}{\otimes} \mathcal{D}.\qedhere
    \end{split}
\end{equation*}
\end{proof}
\begin{rem}
Both sides of \eqref{3.30} are equivalent to the image of the action functor
of $j_*$ on $\cD$.
\end{rem}

Furthermore, if we assume that $\mathcal{D}= (\mathcal{D}')^{\GL_M(\bfO)}$ and the action of $D_!^{\GL_M(\bfO)}(\bfF^M)$ on $(\mathcal{D}')^{\GL_M(\bfO)}$ is given by the action of $D_!(\bfF^M)$ on $\mathcal{D}'$, we have
\begin{equation}
\begin{split}
    D_!^{\GL_M(\bfO)}(\bfO^M\setminus t\bfO^M)\underset{D_!^{\GL_M(\bfO)}(\bfF^M)}{\otimes} (\mathcal{D}')^{\GL_M(\bfO)}    \simeq (D_!(\bfO^M\setminus t\bfO^M)\underset{D_!(\bfF^M)}{\otimes} \mathcal{D}')^{\GL_M(\bfO)}.
\end{split}
\end{equation}


Now by Lemma \ref{lem 3.3.5}, there is an equivalence 
\begin{equation}\label{3.32}
    (D_!(\bfO^M\setminus t\bfO^M)\underset{D_!(\bfF^M)}{\otimes} \mathcal{D}')^{\GL_M(\bfO)}\simeq (\Vect_{e_M} \underset{D_!(\bfF^M)}{\otimes} \mathcal{D}')^{\Mir_M(\bfO)}.
\end{equation}
Here, $D_!(\bfF^M)$ acts on $\Vect_{e_M}$ by pulling back along ${\{e_M\}}\longrightarrow \bfF^M$.

{Applying} the isomorphism \eqref{3.32} to $\mathcal{D}= C_{M|N}$ and $\mathcal{D}'= D^{U_{M,N}(\bfF),\chi_{M,N}}(\Gr_N)${, we get}
\begin{equation}
    (j_*\circledast C_{M|M})\underset{C_{M|M}}{\otimes} C_{M|N}\simeq (\Vect_{e_M} \underset{D_!(\bfF^M)}{\otimes} D^{U_{M,N}(\bfF),\chi_{M,N}}(\Gr_N))^{\Mir_M(\bfO)}.
\end{equation}

To prove \eqref{eq 3.24}, we have to prove:
\begin{equation}\label{3.17}
    (\Vect_{e_M} \underset{D_!(\bfF^M)}{\otimes} D^{U_{M,N}(\bfF),\chi_{M,N}}(\Gr_N))^{\Mir_M(\bfO)}\simeq (D^{U_{M-1, N}(\bfF), \chi_{M, N}}(\Gr_N))^{{\Mir^t_{M}(\bfO)}}.
\end{equation}
\subsubsection{}\label{section 3.11.1}
\begin{proof}[Proof of Proposition \ref{prop 3.1} \textup{b)}.]

By the Fourier equivalence of \cite[Theorem 5.2.11]{[Ber]}, we have 
\begin{equation}\label{3.24}
    \Vect_{e_M} \underset{D_!(\bfF^M)}{\otimes} D^{U_{M,N}(\bfF),\chi_{M,N}}(\Gr_N)\simeq (D^{U_{M,N}(\bfF),\chi_{M,N}}(\Gr_N))^{\bfF^M, \chi_{e_M}}.
\end{equation}
Here, $\chi_{e_M}$ denotes the character of taking residue of the coefficients of $e_M$, and $\bfF^M$ acts on $\Gr_N$ via
\begin{equation*}
    \bfF^M=
\begin{pNiceArray}{ccc|ccc}[last-col,code-for-last-col = \quad\,]
1 & &&  *\\
  &\Ddots& &  \Vdots&  & &M\\
 & &1 & * &  & \\ \Hline
 &&&1\\ 
 &            && &\Ddots& &N-M\\
&             &&  &&1
\CodeAfter
\UnderBrace[shorten,yshift=0.5ex]{1-1}{6-3}{M}
\UnderBrace[shorten,yshift=0.5ex]{1-4}{6-6}{N-M}
\SubMatrix .{1-6}{3-6}{\}}[right-xshift=1em]
\SubMatrix .{4-6}{6-6}{\}}[right-xshift=1em]
\end{pNiceArray} \subset \GL_N(\bfF).
\end{equation*}

\vspace{1.6em}
We have $U_{M-1,N}(\bfF)= \bfF^M \ltimes U_{M,N}(\bfF)$, and the conjugation action of $\bfF^M$ on $U_{M,N}(\bfF)$ preserves $\chi_{M,N}$. Hence, by Lemma \ref{3.3 lem},

\begin{equation}\label{3.25}
    (D^{U_{M,N}(\bfF),\chi_{M,N}}(\Gr_N))^{\bfF^M, \chi_{e_M}}\simeq D^{U_{M-1,N}(\bfF),\chi_{M-1,N}}(\Gr_N).
\end{equation}

{Combining} \eqref{3.24} and \eqref{3.25}, we have
\begin{equation}\label{FT 3.25}
    \Vect_{e_M} \underset{D_!(\bfF^M)}{\otimes} D^{U_{M,N}(\bfF),\chi_{M,N}}(\Gr_N)\simeq D^{U_{M-1,N}(\bfF),\chi_{M-1,N}}(\Gr_N).
\end{equation}

Since \eqref{FT 3.25} is induced by Fourier transform, we see that the action of the group $\Mir_{M}(\bfO)$ on $\Vect_{e_M} \underset{D_!(\bfF^M)}{\otimes} D^{U_{M,N}(\bfF),\chi_{M,N}}(\Gr_N)$ becomes the action of $(\Mir_M(\bfO))^{t,-1}=\Mir^t_M(\bfO)$. 
\end{proof}

\subsubsection{}
Then, we prove \eqref{3.23}. Let $\cD'= D^{U_{M,N}(\bfF),\chi_{M,N}}(\Gr_N)$. To prove the renormalized version \eqref{3.23}, we only need to notice that \[D_*(\GL_M(\bO))\otimes_{D_*(\Mir_M(\bO))}(\Vect_{e_M} \underset{D_!(\bfF^M)}{\otimes} \mathcal{D}')\simeq D_!(\bfO^M\setminus t\bfO^M)\underset{D_!(\bfF^M)}{\otimes} \mathcal{D}',\]
and $\GL_M(\bO)= \GL_M\ltimes \GL_{M,1}(\bO)$, $\Mir_M(\bO)= \Mir_M\ltimes \Mir_{M,1}(\bO)$. Here $\GL_{M,1}(\bO)$ and $\Mir_{M,1}(\bO)$ denote the corresponding congruence subgroups. 

{Applying Corollary}  \ref{prop 3.3.6}, we obtain that there is an equivalence 
\begin{equation}\label{3.56}
    (D_!(\bfO^M\setminus t\bfO^M)\underset{D_!(\bfF^M)}{\otimes} \mathcal{D}')^{\GL_M(\bfO),\ren}\simeq (\Vect_{e_M} \underset{D_!(\bfF^M)}{\otimes} \mathcal{D}')^{\Mir_M(\bfO),\ren}.
\end{equation}

To prove \eqref{3.23}, we only need to prove:
\begin{equation}\label{3.17+}
    (\Vect_{e_M} \underset{D_!(\bfF^M)}{\otimes} D^{U_{M,N}(\bfF),\chi_{M,N}}(\Gr_N))^{\Mir_M(\bfO),\ren}\simeq (D^{U_{M-1, N}(\bfF), \chi_{M, N}}(\Gr_N))^{{\Mir^t_{M}(\bfO)},\ren},
\end{equation}
which follows exactly from the same argument of ~\S\ref{section 3.11.1}. 

\section{Compatibility of actions} \label{act com}


Theorem \ref{bfgt2} implies that there is an equivalences of categories for any $M$,
\begin{equation}\label{N-1'}
    C^{\ren}_{M-1|M}\simeq {{D}}^{\GL_{M-1}\times \GL_M}(\mathfrak{B}_{M-1|M}^{2,0}).
\end{equation}

By ~\S\ref{def of mon}, we have a left action of $C^\ren_{M-1|M-1}$ on $C^\ren_{M-1|M}$ and a right action of $C^\ren_{M|M}$ on $C^\ren_{M-1|M}$. Similarly, by ~\S\ref{convolution}, we have a left action of ${{D}}^{\GL_{M-1}\times \GL_{M-1}}(\mathfrak{B}_{M-1|M-1}^{2,0})$ and a right action of ${{D}}^{\GL_M\times \GL_M}(\mathfrak{B}_{M|M}^{2,0})$ on ${{D}}^{\GL_{M-1}\times \GL_M}(\mathfrak{B}_{M-1|M}^{2,0})$.

In this section, we will prove that the equivalence \eqref{N-1'} is not only an equivalence of categories, but also an equivalence of module categories.

\begin{prop}\label{act}
\textup{a)}  The equivalence \eqref{N-1'}
 is compatible with the left action of \[C^\ren_{M-1|M-1}\simeq {{D}}^{\GL_{M-1}\times \GL_{M-1}}(\mathfrak{B}_{M-1|M-1}^{2,0}).\]

\textup{b)}  The equivalence \eqref{N-1'} is compatible with the right action of \[C^\ren_{M|M}\simeq {{D}}^{\GL_M\times \GL_M}(\mathfrak{B}_{M|M}^{2,0}).\]
\end{prop}
\begin{proof}[Proof of Proposition \ref{act} \textup{b)}]
Let us first check the compatibility of the right actions.

Since the equivalence $ C^\ren_{M|M}\simeq {{D}}^{\GL_{M}\times \GL_M}(\mathfrak{B}_{M|M}^{2,0})$ comes from the ind-completion of a monoidal equivalence functor $ C^\locc_{M|M}\simeq {{D}}_\perf^{\GL_{M}\times \GL_M}(\mathfrak{B}_{M|M}^{2,0})$, it is monoidal. The right convolution always commutes with left convolution, so the equivalence $j_* \circledast C^\ren_{M|M}\simeq \tilde{c} \overset{A}{\star} {{D}}^{\GL_M\times \GL_M}(\mathfrak{B}_{M|M}^{2,0})$ is compatible with the right action of \eqref{NN}. 

Furthermore, note that the transition functor $\xi_{M,*}$ is given by left multiplication with a diagonal element $\operatorname{diag}(t,t,...,t,1)$ in $\GL_M(\bfF)$. In particular, the transition functor commutes with right action of $C^\ren_{M|M}$. Hence, the action of $C^\ren_{M|M}$ on 
\begin{equation}\label{equ 4.2}
    C^\ren_{M-1|M}\simeq \colim_{\xi^{-1}_{M,*}} D^{\Mir^t_M(\bO),\ren}(\Gr_M)\simeq \colim_{\zeta_{M}}\ j_* \circledast C^\ren_{M|M}
\end{equation}
is given by taking colimit of the action of $C^\ren_{M|M}$ on $j_* \circledast C^\ren_{M|M}$. 

Similarly, the action of ${{D}}^{\GL_M\times \GL_M}(\mathfrak{B}_{M|M}^{2,0})$ on ${{D}}^{\GL_{M-1}\times \GL_M}(\mathfrak{B}_{M-1|M}^{2,0})$ is given by taking the colimit of the action of ${{D}}^{\GL_M\times \GL_M}(\mathfrak{B}_{M|M}^{2,0})$ on $\tilde{c} \overset{A}{\star} {{D}}^{\GL_M\times \GL_M}(\mathfrak{B}_{M|M}^{2,0})$. Now Proposition \ref{act} \textup{a)}  follows from the fact that the following diagram
\begin{equation*}
    \xymatrix{j_* \circledast C^\ren_{M|M}\ar[r]^{\zeta_M}\ar[d]^\wr  & j_* \circledast C^\ren_{M|M}\ar[d]^\wr \\ \tilde{c} \overset{A}{\star} {{D}}^{\GL_M\times \GL_M}(\mathfrak{B}_{M|M}^{2,0}) \ar[r]_{\eta}& \tilde{c} \overset{A}{\star} {{D}}^{\GL_M\times \GL_M}(\mathfrak{B}_{M|M}^{2,0})
}
\end{equation*}
commutes. 

\end{proof}

In order to check that \eqref{N-1'} is compatible with left actions of $C^\ren_{M-1|M-1}\simeq {{D}}^{\GL_{M-1}\times \GL_{M-1}}(\mathfrak{B}_{M-1|M-1}^{2,0})$, we need the following lemma.
\begin{lem}\label{lemma 4.2}
There is an equivalence of categories
\begin{equation}\label{4.2}
    C^\ren_{M-1|M}\otimes_{C^\ren_{M|M}} C^\ren_{M|M-1} \simeq C^\ren_{M-1|M-1}.
\end{equation}
\end{lem}
\begin{proof}
By Lemma \ref{lem 3.4.1}, $g\mapsto g^{-1}$ induces an equivalence $C^\ren_{M|M-1}\simeq C^\ren_{M-1|M}$. Furthermore, taking inverse induces a monoidal self-anti-equivalence of $C^\ren_{M|M}$,
 \begin{equation}
     \begin{split}
      \iota:   C^\ren_{M|M}\longrightarrow C^\ren_{M|M}\\
         \mathcal{F}\mapsto \iota(\mathcal{F}),
     \end{split}
 \end{equation}
 
 such that,
 \begin{equation}
     \iota(\mathcal{F})\circledast \iota(\mathcal{G})\simeq \iota(\mathcal{G}\circledast \mathcal{F}).
 \end{equation}
 
The right action of $C^\ren_{M|M}$ on $C^\ren_{M-1|M}$ induces a left action $C^\ren_{M|M}$ on $C^\ren_{M|M-1}$.

In particular, $C^\ren_{M|M-1}\simeq \colim_{\xi'_{M,*}} D(\Gr_M)^{\Mir_M^t(\bO), \ren} \simeq \colim_{\zeta_{M}'}\ C^\ren_{M|M}\circledast j_*$. Here $D(\Gr_M)^{\Mir_M^t(\bO), \ren}$ denotes the ind-completion of the category of locally compact \textit{right} $\Mir_M^t(\bO)$-equivariant D-modules on $\GL_M(\bO)\backslash \GL_M(\bF)$, the transition functors $\xi'_{M,*}$
is given by taking pushforward along the map of right multiplication with the diagonal element $D_M = \operatorname{diag}(t, t,\cdots, t, 1)$, and $\zeta_M'$ is the corresponding transition functor under the equivalence $D(\Gr_M)^{\Mir_M^t(\bO), \ren} \simeq C^\ren_{M|M}\circledast j_*$.

Applying the above formula to the left-hand side of \eqref{4.2}, we obtain
\begin{equation}
\begin{split}
    {C^\ren_{M-1|M}\otimes_{C^\ren_{M|M}} C^\ren_{M|M-1}}
&\simeq C^\ren_{M-1|M}\otimes_{C^\ren_{M|M}} \colim_{\zeta_M'}\ C^\ren_{M|M}\circledast j_*\\
&\simeq \colim_{\zeta_M'}\ C^\ren_{M-1|M}\otimes_{C^\ren_{M|M}} C^\ren_{M|M}\circledast j_*\\
&\simeq  \colim_{\zeta_M'}\ C^\ren_{M-1|M}\circledast j_*.
\end{split}
\end{equation}
By Lemma \ref{4.0.5} below, we have
\[\colim_{\zeta_M'} C^\ren_{M-1|M}\circledast j_*\simeq \colim_{\xi_{M,*}'} D_!(\Gr_{M-1}\times \bfF^{M-1})^{\Mir_M^t(\bO),\ren}.\]
Here $D(\Gr_{M-1}\times \bfF^{M-1})^{\Mir_M^t(\bO),\ren}$ denotes the ind-completion of the category of locally compact \textit{right} $\Mir_M^t(\bO)$-equivariant D-modules on $\GL_{M-1}(\bO)\backslash \Mir^t_M(\bF)\simeq  \Gr_{M-1}\times \bfF^{M-1}.$

Now it is sufficient to prove that 
\begin{equation}
    \colim_{\xi_{M,*}'} D_!(\Gr_{M-1}\times \bfF^{M-1})^{\Mir_M^t(\bO),\ren}\simeq D_!(\Gr_{M-1}\times \bfF^{M-1})^{\GL_{M-1}(\bO),\ren}.
\end{equation}
Here $D_!(\Gr_{M-1}\times \bfF^{M-1})^{\GL_{M-1}{(\bO)},\ren}$ denotes the ind-completion of the category of locally compact \textit{right} $\GL_{M-1}{(\bO)}$-equivariant D-modules on $\GL_{M-1}(\bO)\backslash \Mir^t_M(\bF)\simeq  \Gr_{M-1}\times \bfF^{M-1}.$

We only need to show that
\begin{equation}
    \colim'_{\xi_{M,*}'} D_!(\Gr_{M-1}\times \bfF^{M-1})^{\Mir_M^t(\bO),\locc}\simeq D_!(\Gr_{M-1}\times \bfF^{M-1})^{\GL_{M-1}(\bO),\locc}.
\end{equation}

It is equivalent to
\begin{equation}
    \colim'_{\oblv} D_!(\Gr_{M-1}\times \bfF^{M-1})^{\GL_{M-1}{(\bO)}\ltimes t^i \bO^{M-1},\locc}\simeq D_!(\Gr_{M-1}\times \bfF^{M-1})^{\GL_{M-1}(\bO),\locc}.
\end{equation}

By a left-right symmetry, we should prove the equivalence for the corresponding categories of \textit{left} equivariant D-modules,
\begin{equation}
    \colim'_{\oblv} D_!^{\GL_{M-1}{(\bO)}\ltimes t^i \bO^{M-1},\locc}(\Gr_{M-1}\times \bfF^{M-1})\simeq D_!^{\GL_{M-1}(\bO),\locc}(\Gr_{M-1}\times \bfF^{M-1}).
\end{equation}

By the construction of \cite[Section 3.2]{[BFGT]}, $D_!^{\GL_{M-1}(\bO),\locc}(\Gr_{M-1}\times \bfF^{M-1})$ is generated by $\IC$-sheaves on $\GL_{M-1}{(\bO)}$-orbits in $ \Mir'_M(\bfF)/\GL_{M-1}(\bfO)$. Here $\Mir'_M(\bfF):= { \{(h,v)\in \Mir^t_M(\bfF)\mid v\neq 0\}}$. So we only need to prove that any $\GL_{M-1}(\bO)$-orbit $\mathbb{O}$ in $ \Mir'_M(\bfF)/\GL_{M-1}(\bfO)$ is $t^n\bO^{M-1}$-invariant for {sufficiently} large~$n$.

Assume  $\mathbb{O}=\{[hg,v]\in \Gr_{M-1}\times {(\bfF^{M-1}\setm0)}\mid {h}\in \GL_{M-1}(\bO)\}$ (see ~\S\ref{restr of mon} for the definition of notation). 
Let $m$ be a positive {integer} such that
\begin{equation}\label{cap}
    \GL_{M-1,m}(\bO)\subset \bigcap_{h\in \GL_{M-1}(\bfO)} hg \GL_{M-1}(\bfO) g^{-1}h^{-1}.
\end{equation}
Here $\GL_{M-1,m}(\bO):=1+ t^m \End_{M-1}(\bfO)$ denotes the {$m$}-th congrence subgroup of $\GL_{M-1}(\bO)$. It acts trivially on the affine Grassmannian part of $[hg,v]$.

{To show that  this positive integer $m$ exists, assume $g, g^{-1}\in t^{-l}\End_{M-1}(\bfO)$ for some $l> 0$.  Then  \eqref{cap} is satisfied for any $m\geq 2l$.}

 For any point $[hg,v]\in \mathbb{O}$, as $\GL_{M-1,m}(\bO)\subset \GL_{M-1}(\bfO)$, we have $[\GL_{M-1,m}(\bO) hg,v]\subset \mathbb{O}$. We only need to prove that there exists a positive integer $m_{[g,v]}$ which does not depend on $h$, such that
\begin{equation}
    \GL_{M-1,m}(\bO) hgv\supset hgv+t^{n} \bfO^{M-1},\  \textnormal{for}\ any\ n\geq m_{[g,v]},
\end{equation}
i.e., $\End_{M-1}(\bfO)hgv\supset t^{n} \bfO^{M-1}$, for any $n\geq m_{[h,v]}$. 

We can choose $h'\in \GL_{M-1}(\bfO)$, such that $h'hgv=\sum_{i=1}^{M-1} t^{m_i} e_i$.  Assume $m_i$ is the smallest number among $\{m_1, m_2, \cdots m_{M-1}\}$, it is independent of $h$ and $h'$. We have $\End_{M-1}(\bfO)hgv\supset t^{m_i} \bfO^{M-1}${, and thus we can} take $m_{[g,v]}=m_i$. 

\end{proof}
\begin{lem}\label{4.0.5}
There is an equivalence
\begin{equation}
C^\ren_{M-1|M}\circledast j_*\simeq D(\Gr_{M-1}\times \bfF^{M-1})^{\Mir_M^t(\bO),\ren}.
\end{equation}
\end{lem}
\begin{proof}
By construction, we have 
\[C^\ren_{M-1|M}\circledast j_*\simeq \Ind(C^\locc_{M-1|M}\circledast j_*).\]
Hence, it is sufficient to prove that
\[C^\locc_{M-1|M}\circledast j_*\simeq D(\Gr_{M-1}\times \bfF^{M-1})^{\Mir_M^t(\bO),\locc}.\]
Note that there is an equivalence
\begin{equation}
    \begin{split}
        C_{M-1|M}\circledast j_*&\simeq D(\GL_{M-1}(\bfO)\backslash \GL_M(\bfF)/ \GL_{M}(\bfO))\circledast j_*\\
        &\simeq D(\GL_{M-1}(\bfO)\backslash S'/ \GL_{M}(\bfO))\\
        &\simeq D(\GL_{M-1}(\bfO)\backslash \Mir^t_M(\bfF)/ \Mir^t_M(\bfO)).
    \end{split}
\end{equation}
Here $S'$ denotes the subscheme of $\GL_M(\bF)$ consisting of $g\in \GL_M(\bF)$ such that $g^t\cdot e_M\in \bO^M\setminus t \bO^M$. We should prove that two full subcategories
\[C^\locc_{M-1|M}\circledast j_*\subset C_{M-1|M}\circledast j_*\] and \[D(\Gr_{M-1}\times \bfF^{M-1})^{\Mir_M^t(\bO),\locc}\subset D(\GL_{M-1}(\bfO)\backslash \Mir_M^t(\bfF)/ \Mir^t_M(\bfO))\] coincide. 

The category $D(\GL_{M-1}(\bfO)\backslash S'/ \GL_{M}(\bfO))$ can be regarded as the category of right $\GL_M(\bO)$-equivariant D-modules on $\GL_{M-1}(\bO)\backslash S'$. The category $C^\locc_{M-1|M}\circledast j_*$ is generated by the $\IC$ (intersection cohomology) D-modules on $\GL_M(\bO)$-orbits of $\GL_{M-1}(\bO)\backslash S'$. Under the equivalence $D(\GL_{M-1}(\bfO)\backslash S'/ \GL_{M}(\bfO))\simeq D(\GL_{M-1}(\bfO)\backslash \Mir^t_M(\bfF)/ \Mir^t_M(\bfO))$, $C^\locc_{M-1|M}\circledast j_*$ corresponds to the full subcategory generated by the $\IC$ D-modules on $\Mir_M^t(\bO)$-orbits of $\GL_{M-1}(\bfO)\backslash \Mir_M(\bfF)$. It is exactly the category $D(\Gr_{M-1}\times \bfF^{M-1})^{\Mir_M^t(\bO),\locc}$.
\end{proof}

\subsection{Colimit description of monoidal structures}
Note that Lemma \ref{lemma 4.2} presents $C^\ren_{M-1|M-1}$ as a colimit of monoidal categories. Namely,
\begin{equation}
\begin{split}
\lhs C^\ren_{M-1|M}\otimes_{C^\ren_{M|M}} C^\ren_{M|M-1}\\
&\simeq C^\ren_{M-1|M}\otimes_{C^\ren_{M|M}} C^\ren_{M|M}\otimes_{C^\ren_{M|M}} C^\ren_{M|M-1}\\
&\simeq    (\colim_{\zeta_M} j_*\circledast C^\ren_{M|M}) \otimes_{C^\ren_{M|M}} C^\ren_{M|M}\otimes_{C^\ren_{M|M}} (\colim_{\zeta_M'}\ C^\ren_{M|M} \circledast j_*)\\
&\simeq  \colim_{\zeta_M, \zeta_M'}\ j_*\circledast C^\ren_{M|M} \circledast j_*.  
\end{split}
\end{equation}
Denote by $\tilde{\zeta}_{M}$ the composition of $\zeta_M$ and $\zeta_M'$. Since {the index set $\{(x,x), x\in \mathbb{N}\}$ is cofinal in $\NN\x\NN$}, we have
\begin{equation*}
    \colim_{\zeta_M, \zeta_M'}\ j_*\circledast C^\ren_{M|M} \circledast j_*\simeq \colim_{\tilde{\zeta}_{M}} j_*\circledast C^\ren_{M|M} \circledast j_*.
\end{equation*}

In this section, we will describe the monoidal structure $\textasteriskcentered$ (equivalently, $\circledast$) {on} $C^\ren_{M-1|M-1}$ as a colimit of the restricted monoidal strcuture {on} $j_*\circledast C^\ren_{M|M} \circledast j_*$.

\subsubsection{Colimit description of $\textasteriskcentered$}\label{sec 4.1.1}
We note that $j_*$ is an idempotent algebra object of $C^\ren_{M|M}$, so the monoidal structure on $C^\ren_{M|M}$ induces a monoidal structure of $j_*\circledast C^\ren_{M|M}\circledast j_*$, 
\[\circledast: (j_*\circledast C^\ren_{M|M}\circledast j_*)\otimes (j_*\circledast C^\ren_{M|M}\circledast j_*)\to j_*\circledast C^\ren_{M|M}\circledast j_*.\]

It is obtained from the ind-completion of 
\[\circledast: (j_*\circledast C^\locc_{M|M}\circledast j_*)\otimes (j_*\circledast C^\locc_{M|M}\circledast j_*)\to j_*\circledast C^\locc_{M|M}\circledast j_*,\]
which is the restriction of 
\[\circledast: (j_*\circledast C_{M|M}\circledast j_*)\otimes (j_*\circledast C_{M|M}\circledast j_*)\to j_*\circledast C_{M|M}\circledast j_*.\]

Note that there is an equivalence 
\begin{equation}\label{eq 4.10}
    j_*\circledast C_{M|M}\circledast j_*\simeq D(\Mir^t_M(\bfO)\backslash \Mir^t_M(\bfF)/ \Mir^t_M(\bfO)).
\end{equation}
Indeed, by the definition of $\circledast$, we have 
\begin{equation}
\begin{split}
     (j_*\circledast C_{M|M}) \circledast j_*\simeq&  j_*\circledast D(\GL_M(\bfO)\backslash \Mir_{M+1}(\bfF)/ \GL_M(\bfO)) \circledast j_*\\ \simeq& D(\GL_M(\bfO)\backslash S/ \GL_M(\bfO)) ,
\end{split}
\end{equation}
{where} $S:= \{(g,v)\in \Mir_{M+1}(\bfF)= \GL_M(\bfF)\ltimes \bfF^M\mid v\in \bfO^M\setminus t \bfO^M, gv\in  \bfO^M\setminus t \bfO^M\}$.
Note that we have the following isomorphism
\begin{equation*}  
\begin{split}
	\GL_M(\bfO)\bsl S/ \GL_M(\bO)  &\simeq (\bO^M\setm t\bO^M)/\GL_M(\bO) \underset{\bF^M/\GL_M(\bF)}\x  (\bO^M\setm t\bO^M)/\GL_M(\bO)\\
	&\simeq (\bO^M\setm t\bO^M)/\GL_M(\bO) \underset{(\bF^M \setm 0)/\GL_M(\bF)}\x  (\bO^M\setm t\bO^M)/\GL_M(\bO)\\
	&\simeq \pt/\Mir_M(\bO)   \underset{\pt/\Mir_M(\bF)}\x  \pt/\Mir_M(\bO)\\
	&\simeq \Mir_M(\bfO)\backslash \Mir_M(\bfF)/ \Mir_M(\bfO).
\end{split}
\end{equation*}
We obtain a monoidal equivalence, \[(D(\GL_M(\bfO)\backslash S/ \GL_M(\bfO)),\oast)\simeq (D(\Mir_M(\bfO)\backslash \Mir_M(\bfF)/ \Mir_M(\bfO),{*}),\]
where the monoidal structure $\textasteriskcentered$ of the right-hand side is given by the restriction of the convolution monoidal of $C_{M-1|M-1}$ to the full subcategory $D(\Mir_M(\bfO)\backslash \Mir_M(\bfF)/ \Mir_M(\bfO)$.


{
} 


 The transition functor of $D({\Mir}_M(\bfO)\backslash {\Mir}_M(\bfF)/ {\Mir}_M(\bfO))$ corresponding to $\tilde{\zeta}_{M}$ is given by pushforward along
\begin{equation}
\begin{split}
    \tilde{\xi}^{-1}_M\colon {\Mir}_M(\bfF)\longrightarrow
    {\Mir}_M(\bfF)\\
    g\mapsto D_MgD_M^{-1}.
\end{split}
\end{equation}

We have
\begin{equation}
    \colim_{\tilde{\xi}_{M,*}}\ D({\Mir}_M(\bfO)\backslash {\Mir}_M(\bfF)/ {\Mir}_M(\bfO))\simeq D(\GL_{M-1}(\bfO)\backslash {\Mir}_M(\bfF)/ \GL_{M-1}(\bfO)).
\end{equation}
\begin{rem}
We can also give a description of the above monoidal structure in terms of the transposed mirabolic subgroup. Namely, by taking transpose inverse
\begin{equation}
\begin{split}
    {\Mir}_M(\bfF)&\longrightarrow
    {\Mir}^t_M(\bfF)\\
    g&\mapsto g^{t,-1},
\end{split}
\end{equation}
we can obtain a monoidal equivalence
\[(D(\Mir_M(\bfO)\backslash \Mir_M(\bfF)/ \Mir_M(\bfO),{*})\simeq (D(\Mir^t_M(\bfO)\backslash \Mir^t_M(\bfF)/ \Mir^t_M(\bfO),{*}).\]

The transition functor of $D({\Mir}^t_M(\bfO)\backslash {\Mir}^t_M(\bfF)/ {\Mir}^t_M(\bfO))$ corresponding to $\tilde{\zeta}_{M}$ is given by pushforward along
\begin{equation}
\begin{split}
    \tilde{\xi}_M\colon {\Mir}^t_M(\bfF)\longrightarrow
    {\Mir}^t_M(\bfF)\\
    g\mapsto D^{-1}_MgD_M.
\end{split}
\end{equation}
\end{rem}

It is easy to see that $\tilde{\xi}_{M,*}$ is an {auto}-equivalence of monoidal category. That is to say, the monoidal structure on $D({\Mir}^t_M(\bfO)\backslash D({\Mir}^t_M(\bfF)/{\Mir}^t_M(\bfO)))$ commutes with the transition functors, i.e.,
the following diagram commutes
\[\xymatrix{
\strut D(\bM(\bO)\backslash \bM(\bF)/\bM(\bO))\otimes D(\bM(\bO)\backslash \bM(\bF)/\bM(\bO))\ar[r]\ar[dd]^{\tilde{\xi}_{M,*}\otimes \tilde{\xi}_{M,*}}&  D(\bM(\bO)\backslash \bM(\bF)/\bM(\bO))\ar[dd]^{\tilde{\xi}_{M,*}}\\
&
\\
D(\bM(\bO)\backslash \bM(\bF)/\bM(\bO))\otimes D(\bM(\bO)\backslash \bM(\bF)/\bM(\bO))\ar[r]&D(\bM(\bO)\backslash \bM(\bF)/\bM(\bO)).\\
}\]
Here $\bM(\bO):= \Mir^t_M(\bfO)$ and $\bM(\bF):= \Mir^t_M(\bfF)$.
In particular, the convolution product of $D(\Mir^t_M(\bfO)\backslash \Mir^t_M(\bfF)/ \Mir^t_M(\bfO))$ gives a monoidal structure $\textasteriskcentered'$ on $C_{M-1|M-1}\simeq \colim_{\tilde{\xi}_{M,*}}\ D(\Mir^t_M(\bfO)\backslash \Mir^t_M(\bfF)/ \Mir^t_M(\bfO))$. 

Its restriction to $C^\locc_{M-1|M-1}$ induces a monoidal structure $\textasteriskcentered'$ on $C^\locc_{M-1|M-1}$. Taking its ind-completion, we extend $\textasteriskcentered'$ to $C^\ren_{M-1|M-1}$.

Similarly, taking colimit of the action of $j_*\circledast C^\ren_{M|M}\circledast j_*$ on $j_*\circledast C^\ren_{M|M}$ gives rise to an action of $(C^\ren_{M-1|M-1}, \textasteriskcentered')$ on $C^\ren_{M-1|M}$,
\begin{equation}\label{sec act}
    C^\ren_{M-1|M-1}\otimes C^\ren_{M-1|M}\to C^\ren_{M-1|M}
\end{equation}

Now we claim
\begin{prop}
The monoidal structure $\textasteriskcentered$ of $C^\ren_{M-1|M-1}$ is isomorphic to $\textasteriskcentered'$. Furthermore, the action \eqref{sec act} of $C^\ren_{M-1|M-1}$ on $C^\ren_{M-1|M}$ is isomorphic to the action in ~\S\ref{def of mon}.
\end{prop}
\begin{proof}
We only prove the first claim, the second claim follows from the same argument.

Note that $C^\ren_{M-1|M-1}$ is the ind-completion of $C^\locc_{M-1|M-1}$ and both monoidal structures $\textasteriskcentered$ and $\textasteriskcentered'$ are induced from $C^\locc_{M-1|M-1}$, so we only need to compare the monoidal structures of $C^\locc_{M-1|M-1}$, which are obtained from the restriction from monoidal structures of $C_{M-1|M-1}$. 

{Just}  as the proof of Lemma \ref{colimit lemma}, we can rewrite $C_{M-1|M-1}$ as the colimit of {$D((\GL_{M-1}(\bO)\ltimes t^i\bO^{M-1})\backslash \Mir^t_M(\bF)/(\GL_{M-1}(\bO)\ltimes t^i\bO^{M-1}))$}, and the transition functors are forgetful functors. Now the claim follows from the fact that the {restrictions} of both $\textasteriskcentered'$ and $\textasteriskcentered$ to {$D((\GL_{M-1}(\bO)\ltimes t^i\bO^{M-1})\backslash \Mir^t_M(\bF)/(\GL_{M-1}(\bO)\ltimes t^i\bO^{M-1}))$} are {given by the usual} convolution.
\end{proof}
\begin{rem}
In terms of the monoidal structure $\oast$, by the analysis above, the analog of the above proposition says that the monoidal structure $\oast$ of $C_{M-1|M-1}$ is the colimit of the monoidal structure $\oast$ of $j_*\oast C_{M|M}\oast j_*$ and the action of $(C_{M-1|M-1},\oast)$ on $C_{M-1|M}\simeq \colim j_*\oast C_{M|M}$ is the colimit of the action of $j_*\oast C_{M|M}\oast j_*$ on $j_*\oast C_{M|M}$.
\end{rem}

\subsection{Compatibility of left actions}
Now let us go back to the proof of Proposition \ref{act} \textup{a)}. 
\begin{proof}[Proof of Proposition \ref{act} \textup{a)}]
Note that the equivalence of Theorem \ref{bgft1} is monoidal, and $j_*$ goes to $\tilde{c}$ under the equivalence. So there is an equivalence of monoidal categories
\begin{equation}\label{4.19}
    j_*\circledast C^\ren_{M|M}\circledast j_*\simeq \tilde{c} \overset{A}{\star} {{D}}^{\GL_{M}\times \GL_{M}}(\mathfrak{B}_{M|M}^{2,0}) \overset{A}{\star} \tilde{c}.
\end{equation}

By the same reason, we have an equivalence
\begin{equation}
    j_*\circledast C^\ren_{M|M}\simeq \tilde{c} \overset{A}{\star} {{D}}^{\GL_{M}\times \GL_{M}}(\mathfrak{B}_{M|M}^{2,0}),
\end{equation}
which is compatible with the action of \eqref{4.19}.

Note that 
\begin{equation}
\begin{split}
    {{D}}^{\GL_{M-1}\times \GL_{M-1}}(\mathfrak{B}_{M-1|M-1}^{2,0})
    \simeq C^\ren_{M-1|M-1}
    &\simeq\colim\ j_*\circledast C^\ren_{M|M}\circledast j_*\\
    &\simeq \colim\ \tilde{c} \overset{A}{\star} {{D}}^{\GL_{M}\times \GL_{M}}(\mathfrak{B}_{M|M}^{2,0}) \overset{A}{\star} \tilde{c},
\end{split}
\end{equation}
and \begin{equation}
    {{D}}^{\GL_{M-1}\times \GL_{M}}(\mathfrak{B}_{M-1|M}^{2,0})\simeq \colim\ \tilde{c} \overset{A}{\star} {{D}}^{\GL_{M}\times \GL_{M}}(\mathfrak{B}_{M|M}^{2,0}).
\end{equation} So the monoidal structure of ${{D}}^{\GL_{M-1}\times \GL_{M-1}}(\mathfrak{B}_{M-1|M-1}^{2,0})$ also comes from the colimit of $\tilde{c} \overset{A}{\star} {{D}}^{\GL_{M}\times \GL_{M}}(\mathfrak{B}_{M|M}^{2,0}) \overset{A}{\star} \tilde{c}$, and the action of ${{D}}^{\GL_{M-1}\times \GL_{M-1}}(\mathfrak{B}_{M-1|M-1}^{2,0})$ on ${{D}}^{\GL_{M-1}\times \GL_{M}}(\mathfrak{B}_{M-1|M}^{2,0})$ comes from the colimit of action of $\tilde{c} \overset{A}{\star} {{D}}^{\GL_{M}\times \GL_{M}}(\mathfrak{B}_{M|M}^{2,0}) \overset{A}{\star} \tilde{c}$ on $\tilde{c} \overset{A}{\star} {{D}}^{\GL_{M}\times \GL_{M}}(\mathfrak{B}_{M|M}^{2,0})$.

Now the claim follows from the fact that the equivalence \eqref{4.19} is monoidal.
\end{proof}
\subsection{Proof of the main theorem}
Finally, we can give a proof of Theorem \ref{Gaiotto}.
\begin{proof}[Proof of Theorem \ref{Gaiotto}]
We will prove a stronger theorem: for $N\geq M$, we have a left action of $C^\ren_{M|M}$ and a right action of $C^\ren_{N|N}$ on $C^\ren_{M|N}$, such that there is an equivalence:

\begin{equation}\label{ker}
    C^\ren_{M|N}\simeq {{D}}^{\GL_{M}\times \GL_N}(\mathfrak{B}_{M|N}^{2,0}),
\end{equation}
which is compatible with the left action of 
\[C^\ren_{M|M}\simeq {{D}}^{\GL_{M}\times \GL_M}(\mathfrak{B}_{M|M}^{2,0}),\]
and the right action of 
\[C^\ren_{N|N}\simeq {{D}}^{\GL_{N}\times \GL_N}(\mathfrak{B}_{N|N}^{2,0}).\]

Fix $N\geq 1$, let us prove the theorem using the descending induction on $M$.

If $M=N$ or $M=N-1$, the theorem has already been proved in \cite{[BFGT]} and Proposition \ref{act}.

We assume that the theorem holds for $M$. We take the tensor product of \eqref{N-1'} and \eqref{ker} over \eqref{NN}. By Proposition \ref{act}, we obtain an equivalence 
\begin{equation}\label{ker+}
    C^\ren_{M-1|M}\otimes_{C^\ren_{M|M}} C^\ren_{M|N}\simeq {{D}}^{\GL_{M-1}\times \GL_M}(\mathfrak{B}_{M-1|M}^{2,0})\otimes_{{{D}}^{\GL_{M}\times \GL_M}(\mathfrak{B}_{M|M}^{2,0})} {{D}}^{\GL_{M}\times \GL_N}(\mathfrak{B}_{M|N}^{2,0})
\end{equation}
which is compatible with the left action of 
\[C^\ren_{M|M}\simeq {{D}}^{\GL_{M}\times \GL_M}(\mathfrak{B}_{M|M}^{2,0})\]
and the right action of 
\[C^\ren_{N|N}\simeq {{D}}^{\GL_{N}\times \GL_N}(\mathfrak{B}_{N|N}^{2,0}).\]

By Proposition \ref{prop 2.1} and Theorem \ref{claim1},  the left-hand side of \eqref{ker+} is equivalent to $C^\ren_{M-1|N}$, and the right-hand side is equivalent to  ${{D}}^{\GL_{M-1}\times \GL_N}(\mathfrak{B}_{M-1,N}^{2,0})$. Hence, we obtain the claim for $M-1$.

Taking subcategories of compact objects, we obtain Theorem \ref{Gaiotto}. 
\end{proof}

\subsection{Restricted equivalence}
In this section, we will prove the following theorem.
\begin{thm}\label{thm 4.4.1}
For $M\leq N$, there is an equivalence
\begin{equation}\label{4.1}
 C_{M|N}\simeq {{D}}^{\GL_M\times \GL_N}(\mathfrak{B}^{2,0}_{M|N})_\nilp.
\end{equation}
\end{thm}
Before we prove the above theorem, let first check the above theorem in the case $M=N-1$ or $M=N$. 
\begin{prop}\label{prop 4.4.2}
There are equivalences of categories
\begin{equation}\label{eq 4.2}
 C_{M|M}\simeq {{D}}^{\GL_M\times \GL_M}(\mathfrak{B}^{2,0}_{M|M})_\nilp,
\end{equation}
and 
\begin{equation}\label{4.3}
 C_{M-1|M}\simeq {{D}}^{\GL_{M-1}\times \GL_M}(\mathfrak{B}^{2,0}_{M-1|M})_\nilp.
\end{equation}
\end{prop}
\begin{proof}
We only prove the statement 
\begin{equation}\label{N, nilp}
    C^c_{M|M}\simeq D_\perf^{\GL_{M}\times \GL_M}(\mathfrak{B}_{M|M}^{2,0})_\nilp,
\end{equation}
then \eqref{eq 4.2} follows from taking ind-completion and \eqref{4.3} follows from a similar argument. 

Let $D^{\GL_M(\bO)-\mathsf{mon}}(\Gr_M\times \bfF^M)$ be the full subcategory of $D_!(\Gr_M\times \bfF^M)$ generated by locally compact $\GL_M(\bO)$-equivariant D-modules on $\Gr_M\times \bfF^M$. Since objects in $D^{\GL_M(\bO)-\mathsf{mon}}(\Gr_M\times \bfF^M)$ are ind-holonomic, the left adjoint functor \[\Av_!^{\GL_M(\bfO)}: D^{\GL_M(\bO)-\mathsf{mon}}(\Gr_M\times \bfF^M)\longrightarrow C_{M|M}\]
of the forgetful functor is well-defined. In addition, since the forgetful functor \[C_{M|M}\longrightarrow D^{\GL_M(\bO)-\mathsf{mon}}(\Gr_M\times \bfF^M)\] is conservative, the essential image of compact objects of $(D^{\GL_M(\bO)-\mathsf{mon}}(\Gr_M\times \bfF^M))^c$ under $\Av_!^{\GL_M(\bfO)}$ generates $C_{M|M}^c$ (\cite[Lemma 5.4.3]{[GR]}). 

Let $\IC_{0,0}\in C^{\mathsf{loc.c}}_{M|M}$ be the IC-extension of the perverse constant sheaf on $1\times \bfO^M$. It is known from the construction in \cite{[BFGT]} that $F_{M,M}(\IC_{0,0})=\mathfrak{B}_{M,M}^{2,0}$. Here $F_{M,M}: C^\locc_{M|M}\to {{D}}_\perf^{\GL_{M}\times \GL_M}(\mathfrak{B}_{M|M}^{2,0})$ denotes the equivalence functor in Theorem \ref{bgft1}. Note that $\IC_{0,0}$ is the pullback of $\mathbb{C}$ along $1\times \bfO^M\longrightarrow \pt$. By \cite[Lemma 12.6.5]{[AG]} and the fact that $!$-averaging functor commutes with pullback, there is 
\[F_{M,M}(\Av_!^{\GL_M(\bfO)}(\IC_{0,0}))\simeq F_{M,M}(\IC_{0,0})\underset{\Sym(\mathfrak{gl}_M[-2])^{\GL_M}}{\otimes}\mathbb{C}.\]

Since 
\[\mathfrak{B}_{M|M}^{2,0}\underset{\Sym(\mathfrak{gl}_M[-2])^{\GL_M}}{\otimes}\mathbb{C}\simeq \mathcal{O}_{\nilp_{M|M}},\]
 there is \[F_{M,M}(\Av_!^{\GL_M(\bfO)}(\IC_{0,0})) \simeq  \mathcal{O}_{\nilp_{M|M}}\]under the equivalence of Theorem \ref{bgft1}.

The equivalence in Theorem \ref{bgft1} commutes with the convolution with $\Perv_{\GL_M(\bO)}(\Gr_M)\simeq \Rep^\fin(\GL_M)$, so $\Av_!^{\GL_M(\bfO)}(\IC_\lambda\star \IC_{0,0}\star \IC_\mu)$ corresponds to $V_{\GL_M}^{\lambda}\otimes \mathcal{O}_{\nilp_{M|M}}\otimes V_{\GL_M}^{\mu}$ under the equivalence $F_{M,M}$. Now the claim follows from the fact that the objects $V_{\GL_M}^{\lambda}\otimes \mathcal{O}_{\nilp_{M|M}}\otimes V_{\GL_M}^{\mu}$ generate ${{D}}_\perf^{\GL_{M}\times \GL_M}(\mathfrak{B}_{M|M}^{2,0})_\nilp$.
\end{proof}

The equivalence \eqref{eq 4.2} is an ind-completion of a monoidal equivalence functor, so it is monoidal. In addition, the equivalence \eqref{4.3} is compatible with the left action and right action of \eqref{eq 4.2}.

\begin{cor}\label{cor 4.4.3}
\textup{a)}  The equivalence \eqref{4.3}
 is compatible with the left action of \[C_{M-1|M-1}\simeq {{D}}^{\GL_{M-1}\times \GL_{M-1}}(\mathfrak{B}_{M-1|M-1}^{2,0})_\nilp.\]

\textup{b)}  The equivalence is compatible with the right action of \[C_{M|M}\simeq {{D}}^{\GL_M\times \GL_M}(\mathfrak{B}_{M|M}^{2,0})_\nilp.\]
\end{cor}

Now let us repeat the argument of the proof of Theorem \ref{Gaiotto} to prove Theorem \ref{thm 4.4.1}.

\begin{proof}(of Theorem \ref{thm 4.4.1}.)

By Proposition \ref{prop 4.4.2}, Theorem \ref{thm 4.4.1} holds in the case $M=N$ or $M=N-1$.

Fix $N\geq 1$, assume that we have already proved Theorem \ref{thm 4.4.1} for $M$.
 We take the tensor product of \eqref{4.3} and
 \[C_{M|N}\simeq {{D}}^{\GL_{M}\times \GL_N}(\mathfrak{B}_{M|N}^{2,0})_\nilp\]
  over \eqref{eq 4.2}. By Proposition \ref{prop 4.4.2}, Corollary \ref{cor 4.4.3}, and our assumption, we obtain an equivalence 
\begin{equation}\label{ker++}
\begin{split}
        C_{M-1|M}&\otimes_{C_{M|M}} C_{M|N}\\ &\simeq \\{{D}}^{\GL_{M-1}\times \GL_M}(\mathfrak{B}_{M-1|M}^{2,0})_\nilp&\otimes_{{{D}}^{\GL_{M}\times \GL_M}(\mathfrak{B}_{M|M}^{2,0})_\nilp} {{D}}^{\GL_{M}\times \GL_N}(\mathfrak{B}_{M|N}^{2,0})_\nilp
\end{split}
\end{equation}
which is compatible with the left action of 
\[C_{M|M}\simeq {{D}}^{\GL_{M}\times \GL_M}(\mathfrak{B}_{M|M}^{2,0})_\nilp\]
and the right action of 
\[C_{N|N}\simeq {{D}}^{\GL_{N}\times \GL_N}(\mathfrak{B}_{N|N}^{2,0})_\nilp.\]

By Proposition \ref{prop 2.1} and Theorem \ref{claim1},  the left-hand side of \eqref{ker++} is equivalent to $C_{M-1|N}$, and the right-hand side is equivalent to  ${{D}}^{\GL_{M-1}\times \GL_N}(\mathfrak{B}_{M-1,N}^{2,0})_\nilp$. Hence, we obtain the claim for $M-1$.
\end{proof}

\begin{rem}
Using the equivalence interchanging left and right actions, we can also prove Theorem \ref{Gaiotto} and Theorem \ref{thm 4.4.1} for $N\leq M$.
\end{rem}

\section{A symmetric definition of $C_{M|N}$}\label{5}
In the definition of ${{D}}^{\GL_{M}\times \GL_{N}}(\mathfrak{B}_{M|N}^{2,0})$ and ${{D}}^{\GL_{M}\times \GL_{N}}(\mathfrak{B}_{M|N}^{2,0})_\nilp$, we do not need to require $N\leq M$ or $M\leq N$. But in the definition of $C_{M|N}$, we need to require $M\leq N$ or $N\leq M$, or the definition does not make sense. Although there is a canonical equivalence between $C_{N|M}\simeq C_{M|N}$, we want to introduce a more 'symmetric' definition of $C_{M|N}$ for any $M, N$.

Furthermore, in \eqref{1.3 GL}, it seems unreasonable to require that the embedding of $\GL_M(\bfO)$ into $\GL_N(\bfF)$ concentrates in the left top part. Indeed, in this section we will also study the corresponding equivariant category if we move $\GL_M(\bfO)$ to other columns and rows, and we will see that the resulting category does not depend on where we put $\GL_M(\bfO)$. It is expected that the choice of columns and rows corresponds to the choice of a Borel subgroup of the supergroup.

\subsection{Definition}

Assume $L>N$, $M$ and $r+s= L-M-1$, we denote by 
$
U_{M,L}^{r,s}=
\left(
\begin{array}{ccc}
U_r &* &*\\
  &1_{M+1}&* \\
 & &U_s
\end{array}
\right)
$, the unipotent radical of the parabolic subgroup $P_{r,s}(\bfF)$ corresponding to the partition $(1,1,...,1,M+1,1,...,1)$. We denote by $\chi_{M,L}^{r,s, (k)}$ the character
\begin{equation}
    (u_{ij})\in U_{M,L}^{r,s}\longrightarrow Res_{t=0}(\sum_{i=1}^{r-1} u_{i,i+1}+ u_{r,k}+u_{k,{L-s+1}}+ \sum_{i=L-s+1}^{L-1} u_{i,i+1})
\end{equation}
for any choice of $k\in \{r+1,...,L-s\}$. If $r=0$ or $s=0$, one of $u_{r,k}$ and $u_{k,L-s-1}$ should be omitted.

We embed $\GL_M(\bfO)$ into $P_{r,s}(\bfF)$ along the columns and rows $r+1,r+2,...,k-1,$ $k+1,..., L-s$. Since $\GL_M(\bfO)$ belongs to the normalizer group of $U_{M,L}^{r,s}$, we can form the semidirect product $\GL_M(\bfO)\ltimes U_{M,L}^{r,s}$. Furthermore, the conjugation action of $\GL_M(\bfO)$ on $U_{M,L}^{r,s}$ preserves $\chi_{M,L}^{r,s,(k)}$. In particular, $(\GL_M(\bfO)\ltimes U_{M,L}^{r,s}, \chi_{M,L}^{r,s,(k)})$-invariants are well-defined.

\begin{rem} Assume $k_1, k_2\in \{r+1,r+2,...,L-s\}$. Let $s_{k_1,k_2}\in \GL_L(\bfF)$ be
 the permutation matrix which permutes $k_1$ and $k_2$. Left multiplication with $s_{k_1, k_2}$ gives an equivalence: \[\mathcal{C}^{\GL_M(\bfO)\ltimes U_{M,L}^{r,s}, \chi_{M,L}^{r,s,(k_1)}}\simeq \mathcal{C}^{\GL_M(\bfO)\ltimes U_{M,L}^{r,s}, \chi_{M,L}^{r,s,(k_2)}}\] for any category $\mathcal{C}$ which admits an action of $\GL_{L}(\bfF)$. Hence, we may omit $k$ in the definition of $\chi_{M,L}^{r,s,(k)}$.
 \end{rem}
 
 \begin{defn}
 We define \[C_{M,N, r,s,r',s',L}:= D(\GL_M(\bfO)\ltimes U_{M,L}^{r,s}, \chi_{M,L}^{r,s}\backslash \Mir^t_{L}(\bfF)/ \GL_N(\bfO)\ltimes U_{N,L}^{r',s'}, \chi_{N,L}^{r',s'}).\]
 \end{defn}

The main goal of this section is to prove the following theorem.

\begin{thm}
The category $C_{M,N, r,s,r',s',L}$ only depends on $M,N$. That is to say, for any 
\begin{equation*}
    \begin{split}
        r_1+s_1= L_1-M-1,\ r_2+s_2=L_2- M-1,\\
r_1'+s_1'= L_1-N-1,\ r_2'+s_2'=L_2- N-1,
    \end{split}
\end{equation*}
 we have

\begin{equation}\label{main 2}
C_{M,N, r_1,s_1,r_1',s_1',L_1}\simeq C_{M,N, r_2,s_2,r_2',s_2',L_2}.
\end{equation}
In particular, $C_{M,N, r,s,r',s',L}\simeq C_{M|N}$.

\end{thm}
\subsection{Examples} 
\begin{enumerate}
    \item If $M=N=L-1$, $r=r'=0$, $s=s'=0$, then $\GL_M(\bfO)\ltimes U_{M,L}^{r,s}=$ $\GL_N(\bfO)\ltimes U_{M,L}^{r',s'}= \GL_M(\bfO)$ and the characters are trivial. Hence, $C_{M,N, r,s,r',s',L}\simeq D(\GL_M(\bfO)\backslash \Mir_{M+1}^t(\bfF)/\GL_M(\bfO))\simeq C_{M|M}$.\\
    \item If $M=L-2$, $N=L-1$, $r=r'=0$, $s=1$, and $s'=0$. Then, $\GL_M(\bfO)\ltimes U_{M,L}^{r,s}=\GL_M(\bfO)\ltimes \bfF^{L-1}$, and $\GL_N(\bfO)\ltimes U_{M,L}^{r',s'}= \GL_N(\bfO)$. Here, $\chi_{M,L}^{r,s}$ equals taking residue of coefficients of $e_{M+1}$, while $\chi_{N,L}^{r',s'}$ is trivial. We have
\begin{equation*}
    \begin{split}
        C_{M,N,r,s,r',s',L}&\simeq D(\GL_M(\bfO)\ltimes \bfF^{M+1}, \chi_{M,L}^{r,s}\backslash \Mir_{L}(\bfF)/ \GL_N(\bfO))\\
        &\simeq D^{\GL_M(\bfO)\ltimes \bfF^{M+1}, \chi_{M,L}^{r,s}}(\Gr_{M+1}\times \bfF^{M+1})\\
        &\simeq D^{\GL_M(\bfO)} (\Gr_{M+1}).
    \end{split}
\end{equation*}
Here, the last equivalence is given by Fourier transform.\\
\item 
If $M=0$ and $L=N+1$, in this case, $C_{M,N, r,s,r',s',L}\simeq Whit(\Gr_N)$. Here, $Whit(\Gr_N)$ denotes the Whittaker model of $D(\Gr_N)$.
\end{enumerate}

The rest of this section is devoted to the proof of \eqref{main 2}.
 
\subsection{Independence of \textit{s,r}}
 
In this section, we will prove if $L_1=L_2$, then \eqref{main 2} holds. By induction, we only need to show,
\begin{equation}
    C_{M,N, r,s,r',s',L}\simeq C_{M,N, r+1,s-1,r',s',L}.
\end{equation}

It is easy to see that the above statement follows from the following lemma,
\begin{lem}
For any category $\mathcal{C}$ admitting an action of $\GL_L(\bfF)$. we have
\begin{equation}
    \mathcal{C}^{U_{M,L}^{r,s}, \chi_{M,L}^{r,s}}\simeq \mathcal{C}^{U_{M,L}^{r+1,s-1}, \chi_{M,L}^{r+1,s-1}}
\end{equation}
\end{lem}

\begin{proof}
Consider the following groups, 
\begin{equation*}
    \xymatrix @R=1em{&U_{M,L}^{r,s}\ar@{}[dr]|-*[@]{\subset}&\\
U_{M+1,L}^{r,s-1}\ar@{}[ur]|-*[@]{\subset} \ar@{}[rd]|-*[@]{\subset}&& {U_{M-1,L}^{r+1,s}.}\\
&U_{M,L}^{r+1,s-1}\ar@{}[ur]|-*[@]{\subset}&
}
\end{equation*}

By definition and Lemma \ref{3.3 lem}, we have
 \[\mathcal{C}^{U_{M,L}^{r,s}, \chi_{M,L}^{r,s}}= (\mathcal{C}^{U_{M+1,L}^{r,s-1},\chi_{M,L}^{r,s}|_{U_{M+1,L}^{r,s-1}}})^{U_{M,L}^{r,s}/U_{M+1.L}^{r,s-1}, \chi_{M,L}^{r,s}},\]  
\[\mathcal{C}^{U_{M,L}^{r+1,s-1}, \chi_{M,L}^{r+1,s-1}}= (\mathcal{C}^{U_{M+1,L}^{r,s-1},\chi_{M,L}^{r+1,s-1}|_{U_{M+1,L}^{r,s-1}}})^{U_{M,L}^{r+1,s-1}/U_{M+1.L}^{r,s-1}, \chi_{M,L}^{r+1,s-1}}.\]

We note that $\chi_{M,L}^{r,s, (r+1)}|_{U_{M+1,L}^{r,s-1}}=\chi_{M,L}^{r+1,s-1, (L-s+1)}|_{U_{M+1,L}^{r,s-1}}$, we should prove that for any category $\mathcal{C}'$ admitting an action of $U_{M-1,L}^{r,s+1}/ U_{M+1,L}^{r,s-1}$, we have 
\begin{equation}\label{4.5}
    (\mathcal{C}')^{U_{M,L}^{r,s}/U_{M+1.L}^{r,s-1}, \chi_{M,L}^{r,s, {(r+1)}}}\simeq (\mathcal{C}')^{U_{M,L}^{r+1,s-1}/U_{M+1,L}^{r,s-1}, \chi_{M,L}^{r+1,s-1, {(L-s+1)}}}.
\end{equation}

We can identify $U_{M,L}^{r+1,s-1}/U_{M+1.L}^{r,s-1}$ with $V\oplus \bfF$ for a vector space $V$, identify $U_{M,L}^{r,s}/U_{M+1.L}^{r,s-1}$ as $V^*\oplus \bfF$, and identify $U_{M-1,L}^{r,s+1}/ U_{M+1, L}^{r,s-1}$ with the Heisenberg group $\Hei:= V\oplus V^*\oplus \bfF$. The group structure of $\Hei$ is given by 
\begin{equation}
\begin{split}
        \Hei\times \Hei&\longrightarrow \Hei\\
        (v,v^*,c), (v', v^{*,\prime}, c')&\mapsto (v+v', v^*+v^{*,\prime}, c+c'+\langle v, v^{*,\prime}\rangle).
\end{split}
\end{equation}
Taking restriction along $V^*\hookrightarrow \Hei$ (resp. $V\hookrightarrow \Hei$) defines an equivalence 

\begin{equation}
    D(\Hei/ {(V\oplus \bfF)}, \chi)\simeq D(V^*).
\end{equation}

\begin{equation}
  \operatorname{(resp.}\   D(\Hei/ {(V^*\oplus \bfF)}, \chi)\simeq D(V)).
\end{equation}

The action of $V$ on $D(V)$ is given by shifts and the action of $v^*\in V^*$ on ${D(}V{)}$ is given by tensoring with the local system $e^{v^*}$ on $V$. In the $D(V^*)$ realization, it is the other way around, and the Fourier transform intertwines these two actions. 

Hence, we obtain an equivalence of left $\Hei$-module categories {by re-averaging}:
\begin{equation}
   D^{V\oplus \bfF,\chi}(\Hei)\simeq D(\Hei/ {(V\oplus \bfF)}, \chi)\simeq D(\Hei/ {(V^*\oplus \bfF)}, \chi)\simeq  D^{V^*\oplus \bfF,\chi}(\Hei).
\end{equation}

Now, the equivalence \eqref{4.5} follows from the facts:
\begin{equation}
    \Fun_{D(\Hei)}( D^{V\oplus \bfF,\chi}(\Hei), \mathcal{C}')\simeq (\mathcal{C}')^{V\oplus \bfF,\chi}
\end{equation}
and 
\begin{equation}
    \Fun_{D(\Hei)}( D^{V^*\oplus \bfF,\chi}(\Hei), \mathcal{C}')\simeq (\mathcal{C}')^{V^*\oplus \bfF,\chi}.\qedhere
\end{equation}

\end{proof}

\subsection{Independence of \itshape L}
In this section, we will show that the definition of $C_{M,N,L,r,s,r',s'}$ is independent of the choice of $L$.

\begin{rem} Since we have already seen that $C_{M,N,L,r,s,r',s'}$ is independent of the choice of $r,s,r',s'$, we could just denote it by $C_{M,N,L}$. We denote by $\chi_{M,L}$ the character $\chi_{M,L}^{r,s}$ for any choice of $r,s$.
\end{rem}

We are going to prove the following lemma first,
\begin{lem}
{We have 
\begin{equation}\label{eq 5.13}
    D(\Mir^t_{L}(\bfF))\simeq D(\bfF^{L},\chi_{L-1,L+1}\backslash \Mir^t_{L+1}(\bfF)/ \bfF^{L},\chi_{L-1,L+1})
\end{equation}
}
\end{lem}
\begin{proof}
By the Fourier transform,
\begin{equation}
\begin{split}
D(\bfF^{L},\chi_{L-1,L+1}\backslash &\Mir^t_{L+1}(\bfF)/ \bfF^{L},\chi_{L-1,L+1})\\
    & {\simeq}\\
    \Vect_{e_{L}} \otimes_{D(\bfF^{L})} D(&\Mir^t_{L+1}(\bfF)) \otimes_{D(\bfF^{L})} \Vect_{e_{L}}.
\end{split}
\end{equation}
Here, both functors $D(\bfF^{L})\longrightarrow \Vect_{e_{L}}{{}\simeq\Vect}$ are given by taking restriction at $e_{L}\in \bfO^{L}\setminus t \bfO^{L}$. The left (resp. right) action of $D(\bfF^{L})$ on $D(\Mir^t_{L+1}(\bfF))$ is given by the pullback along 
\begin{equation}
    \Mir^t_{L+1}(\bfF)\simeq \bfF^{L}\rtimes \GL_{L}(\bfF)\longrightarrow \bfF^{L}
\end{equation}
\begin{equation}
    \operatorname{(resp.}\ \Mir^t_{L+1}(\bfF)\simeq \GL_{L}(\bfF)\ltimes \bfF^{L}\longrightarrow \bfF^{L}).
\end{equation}

Hence, 
\begin{equation}
     \Vect_{e_{L}} \otimes_{D(\bfF^{L})} D(\Mir^t_{L+1}(\bfF)) \otimes_{D(\bfF^{L})} \Vect_{e_{L}}\simeq 
     D(H''),
\end{equation}
where $H'':= {\{(g,v)\in \GL_{L}(\bfF)\ltimes \bfF^{L}\mid v= e_{L}, gv=e_{L} \}}\simeq \Mir_{L}(\bfF)$. We proved the lemma.
\end{proof}

Now let us take $(\GL_M(\bfO)\ltimes U_{M,L}(\bfF),\chi_{M,L})$-invariants on the left and $(\GL_N(\bfO)\ltimes U_{N,L}(\bfF),\chi_{N,L})$-invariants on the right \eqref{eq 5.13}. We have
\begin{equation}
  \resizebox{0.95\hsize}{!}{$\begin{split}
        &C_{M,N,L}\\
        &\simeq D(\GL_M(\bfO)\ltimes U_{M,L}(\bfF),\chi_{M,L}\backslash \Mir^t_{L}(\bfF)/ \GL_M(\bfO)\ltimes U_{N,L}(\bfF),\chi_{N,L})\\
        &\simeq D((\GL_M(\bfO)\ltimes U_{M,L}(\bfF))\ltimes \bfF^{L},\chi_{M,L+1} \backslash \Mir^t_{L}(\bfF)/ (\GL_M(\bfO)\ltimes U_{N,L}(\bfF))\ltimes \bfF^{L}, \chi_{N,L+1})\\
        &\simeq D(\GL_M(\bfO)\ltimes U_{M,L+1}(\bfF),\chi_{M,L+1}\backslash \Mir^t_{L+1}(\bfF)/ \GL_M(\bfO)\ltimes U_{N,L+1}(\bfF),\chi_{N,L+1})\\
        &\simeq C_{M,N,L+1}.
    \end{split}$}
\end{equation}

{
In particular, the category $C_{M,N,L}$ does not depend on $L$.
\begin{rem}
 The Gaiotto meta conjecture claimed that the Langlands dual of the category $\hat{\gl}(M|N)-\on{mod}$ is the  Whittaker bi-invariants of the category of D-modules on $\Mir^t_{L}(\bF)$ for any $L>M,N$. The lemma above asserts that this Whittaker bi-invariants category is independent of $L$. In fact, taking $(U_{M,L}(\bF),\chi_{M,L})$-invariants on the left and $(U_{N,L}(\bF),\chi_{N,L})$-invariants on the right of \eqref{eq 5.13}, we obtain that 
  \begin{equation}
      \begin{split}
          D(U_{M,L}(\bfF),\chi_{M,L}\backslash \Mir^t_{L}(\bfF)/U_{N,L}(\bfF),\chi_{N,L})\\ 
          \simeq D(U_{M,L+1}(\bfF),\chi_{M,L+1}\backslash \Mir^t_{L+1}(\bfF)/U_{N,L+1}(\bfF),\chi_{N,L+1}).
      \end{split}
  \end{equation}
\end{rem}
}
\begin{rem} From the construction of the equivalence $C_{M,N,L}\simeq C_{M,N,L+1}$, the equivalence is compatible with the left action of $C_{N|N}$ and the right action of $C_{M|M}$.
\end{rem}

As a corollary,
\begin{cor} If $M'\leq M$ and $N'\leq N$, we have
\begin{equation}
    C_{M'|M}\otimes_{C_{M|M}} C_{M,N,L}\otimes_{C_{M|M}} C_{N|N'}\simeq C_{M',N',L}.
\end{equation}
which is compatible with the left action of $C_{M'|M'}$ and the right action of $C_{N'|N'}$.
\end{cor}

\begin{cor}
The left action of $C_{M|M}$ on $C_{M|N}$ is equivalent to the left action of $C_{M,M, L}\simeq C_{M|M}$ on $C_{M,N,L}\simeq C_{M|N}$. 

Similarly for the right action.
\end{cor}

\subsection*{Acknowledgments}We are grateful to Michael Finkelberg and Alexander Braverman for very helpful discussions and generous help in writing this paper. 
We also thank M.~Finkelberg for constant support and encouragement during our working on this project, and A.~Braverman for communicating to us the conjecture by D.~Gaiotto.
This paper belongs to a long series of papers about mirabolic Hecke algebras and later about their categorification, with Victor Ginzburg and Michael Finkelberg.   We are grateful to them for sharing many ideas.
Finally, we are thankful to Dennis Gaitsgory for teaching us a lot of techniques of the geometric Langlands theory.


\begin{thebibliography}{9}


  \bibitem[1]{[AG]}
\textit{D.~Arinkin} and  \textit{D.~Gaitsgory},
{Singular support of coherent sheaves, and the geometric Langlands conjecture}, Selecta Mathematica \textbf{21}(1).
doi:10.1007/s00029-014-0167-5. \arXiv{1201.6343} [math.AG]
\bibitem[2]{[B]}
\textit{R.~Bezrukavnikov},
  {On two geometric realizations of an affine Hecke algebra}, Publ.
Math. Inst. Hautes Etudes Sci. \textbf{123} (2016), 1-67.
  \bibitem[3]{[Ber]}
   \textit{D.~Beraldo},
  {Loop Group Actions on Categories and Whittaker Invariants},
   	Advances in Mathematics \textbf{322} (2017) 565-636. \arXiv{1310.5127}[math.RT].
       \bibitem[4]{[BF]}
\textit{R.~Bezrukavnikov} and \textit{M.~Finkelberg},
{Equivariant Satake category and Kostant Whittaker reduction},  
 Moscow Math. Journal \textbf{8} (2008), no. 1, 39-72.
       \bibitem[5]{[BFGT]}
\textit{A.~Braverman}, \textit{M.~Finkelberg}, \textit{V.~Ginzburg}, and \textit{R.~Travkin},
{Mirabolic Satake equivalence and supergroups},  
Compositio Mathematica, \textbf{157}(8), 1724-1765. doi:10.1112/S0010437X21007387
\bibitem[6]{[BFT]}
\textit{A.~Braverman}, \textit{M.~Finkelberg}, and \textit{R.~Travkin},
{Orthosymplectic Satake equivalence},  
 \arXiv{1912.01930} [math.RT]
\bibitem[7]{[BZSV]}
\textit{D.~Ben-Zvi}, \textit{Y.~Sakellaridis}, and \textit{A.~Venkatesh}, {Relative Langlands duality},
\url{https://math.jhu.edu/~sakellar/BZSVpaperV1.pdf},  
  	       \bibitem[8]{[FGT]}
   	       \textit{M.~Finkelberg}, \textit{V.~ Ginzburg}, and \textit{R.~ Travkin},
{Mirabolic affine Grassmannian and character sheaves}, 
Selecta Mathematica volume \textbf{14}, pages 607-628 (2009)
     \bibitem[9]{[GR]}
\textit{D.~Gaitsgory} and
\textit{N.~Rozenblyum},
{A study in derived algebraic geometry
Volume I: Correspondences and duality},  Vol. \textbf{221}. American Mathematical Society, 2019. \url{https://people.math.harvard.edu/~gaitsgde/GL/Vol1.pdf}
  \bibitem[10]{[L]}
   	     \textit{J.~Lurie}, 
{Higher algebra.}, \url{https://people.math.harvard.edu/~lurie/papers/HA.pdf} (2017).
\end{thebibliography}
\end{document}